\documentclass[leqno]{amsart}
\usepackage{graphicx, amsmath, amssymb, amsthm, hyperref, verbatim, multicol, mathrsfs, tikz, caption, subcaption}
\usepackage{stmaryrd}
\usepackage{mathtools}  
\usepackage{cleveref}   
\usepackage{todonotes}
\usepackage{tikz-cd}
\usepackage{physics}
\usepackage{enumitem}   

\newtheorem{thm}{Theorem}[section]
\newtheorem{lemm}[thm]{Lemma}
\newtheorem{cor}[thm]{Corollary}

\newtheorem{prop}[thm]{Proposition}
\newtheorem{ques}[thm]{Question}
\newtheorem{prob}[thm]{Problem}

\theoremstyle{definition}
\newtheorem{defi}[thm]{Definition}

\newtheorem{rem}[thm]{Remark}

\theoremstyle{remark}

\newtheorem{example}{Example}

\makeatletter
\newcommand*\rel@kern[1]{\kern#1\dimexpr\macc@kerna}
\newcommand*\widebar[1]{%
  \begingroup
  \def\mathaccent##1##2{%
    \rel@kern{0.8}%
    \overline{\rel@kern{-0.8}\macc@nucleus\rel@kern{0.2}}%
    \rel@kern{-0.2}%
  }%
  \macc@depth\@ne
  \let\math@bgroup\@empty \let\math@egroup\macc@set@skewchar
  \mathsurround\z@ \frozen@everymath{\mathgroup\macc@group\relax}%
  \macc@set@skewchar\relax
  \let\mathaccentV\macc@nested@a
  \macc@nested@a\relax111{#1}%
  \endgroup
}
\makeatother

\DeclareMathOperator{\diam}{diam}
\DeclareMathOperator{\dist}{dist}
\DeclareMathOperator{\Int}{int}
\DeclareMathOperator{\Mod}{mod}
\DeclareMathOperator{\loc}{loc}

\DeclareMathOperator{\aplim}{\mathrm{ap}\lim}

\newcommand{\N}{\mathbb{N}}
\newcommand{\D}{\mathbb{D}}

\newcommand{\R}{\mathbb{R}}
\newcommand{\C}{\mathbb{C}}

\newcommand{\Z}{\mathbb{Z}}

\newcommand{\br}[1]{\overline{#1}}
\makeatletter

\newcommand{\Rom}[1]{\expandafter\@slowromancap\romannumeral #1@}

\numberwithin{equation}{section}
\numberwithin{example}{section}
\numberwithin{figure}{section}

\begin{document}

\title[Polyhedral approximation and uniformization]{Polyhedral approximation and uniformization for non-length surfaces}

\author{Dimitrios Ntalampekos}
\author{Matthew Romney}
\address{Mathematics Department, Stony Brook University, Stony Brook NY, 11794, USA.}
\email{dimitrios.ntalampekos@stonybrook.edu}
\email{matthew.romney@stonybrook.edu}

\date{\today}
\thanks{The first author is partially supported by NSF Grant DMS-2000096}

\maketitle

\begin{abstract}
We prove that any metric surface (that is, metric space homeomorphic to a $2$-manifold with boundary) with locally finite Hausdorff $2$-measure is the Gromov--Hausdorff limit of polyhedral surfaces with controlled geometry.  We use this result, together with the classical uniformization theorem, to prove that any metric surface homeomorphic to the $2$-sphere with finite Hausdorff $2$-measure admits a weakly quasiconformal parametrization by the Riemann sphere, answering a question of Rajala--Wenger. These results have been previously established by the authors under the assumption that the metric surface carries a length metric. As a corollary, we obtain new proofs of the uniformization theorems of Bonk--Kleiner for quasispheres and of Rajala for reciprocal surfaces. Another corollary is a simplification of the definition of a reciprocal surface, which answers a question of Rajala concerning minimal hypotheses under which a metric surface is quasiconformally equivalent to a  Euclidean domain.
\end{abstract}

\setcounter{tocdepth}{1}
\tableofcontents

\section{Introduction}
We say that a \textit{metric surface} is a metric space homeomorphic to a $2$-manifold, possibly with non-empty boundary. In our paper \cite{NR:21}, we give results on two related problems for metric surfaces: first, approximating a given surface by polyhedral surfaces with controlled geometry, and, second, finding a uniformizing map from a corresponding constant curvature surface. The results in \cite{NR:21} are proved assuming that the metric on the surface is a length metric with locally finite Hausdorff $2$-measure. The length metric assumption means that the distance between two points is the infimum of the lengths of paths connecting them; this is a natural condition built into the definition of many standard classes of spaces, such as Riemannian or Finsler manifolds. However, an arbitrary metric surface need not resemble a length space. For example, one can construct a surface as a subset of some ambient space, such as $\mathbb{R}^3$, and equip it with the metric inherited from this ambient space. Such a surface may have points that are inaccessible by rectifiable curves.

The purpose of this paper is to extend the results in \cite{NR:21} to the case of non-length metric surfaces. Instead, we assume merely that the Hausdorff $2$-measure on the surface is locally finite. This is the minimal assumption needed to use conformal modulus methods, as we do in our approach to the topic of uniformization. Our main result is a generalization of Rajala's uniformization theorem \cite{Raj:17}, one of the major recent works on the topic. Rajala's theorem (\Cref{thm:rajala} below) states that a metric surface of locally finite Hausdorff $2$-measure homeomorphic to the plane is quasiconformally equivalent to the Euclidean plane if and only if it satisfies a geometric condition called \textit{reciprocity} (see \Cref{sec:reciprocal_surfaces}). Our generalization gives the existence of a \textit{weakly quasiconformal} parametrization even without the reciprocity assumption. In the case of reciprocal surfaces, such a parametrization is a quasiconformal homeomorphism, and hence we obtain a new proof of Rajala's theorem by a different method. In fact, our approach improves his result by showing that one of the two conditions in the original definition of {reciprocity} is implied by the other and hence is redundant.

The overall scheme of the present paper follows that of our paper \cite{NR:21}. However, there are significant differences in the details, since many of the original arguments rely on the length metric assumption and no longer apply. On the other hand, removing the requirement of a length metric simplifies some steps of the argument, leading to various technical improvements over \cite{NR:21}. 

\smallskip

\subsection{Polyhedral approximation}

In our first result, we extend our polyhedral approximation theorem (Theorem 1.1 in \cite{NR:21}) to cover the setting of non-length surfaces as well.

\begin{thm} \label{thm:extended_polyhedral_approximation}
Let $X$ be a metric surface of locally finite Hausdorff $2$-measure. There exists a sequence of polyhedral surfaces $\{(X_n,d_{X_n})\}_{n=1}^\infty$ each homeomorphic to $X$, where $d_{X_n}$ is a metric that is locally isometric to the polyhedral metric on $X_n$, such that the following properties hold for an absolute constant $K \geq 1$.  
\begin{enumerate}[label=\normalfont(\arabic*)]
    \item \label{item:main_1} There exists an approximately isometric sequence of maps $f_n \colon X_n \to X$, $n \in \mathbb{N}$. Moreover, for each $n\in \N$, the  map $f_n$ is a proper topological embedding.
    \item \label{item:main_2} For each {compact} set $A \subset X$, \[\limsup_{n \to \infty} \mathcal{H}^2(f_n^{-1}(A)) \leq K \mathcal{H}^2(A).\] 
    \item\label{item:main_3} There exists an approximately isometric sequence of retractions $R_n \colon X\to f_n(X_n)$, $n\in \N$. 
\end{enumerate} 
If $X$ is a length space, then we may take $d_{X_n}$ to be the polyhedral metric on $X_n$.
\end{thm}
In particular, \ref{item:main_1} implies that the sequence $X_n$ converges to $X$ in the Gromov--Hausdorff sense. The last statement of the theorem regarding the length space case follows from our polyhedral approximation result in \cite[Thm.\ 1.1]{NR:21}. Condition \ref{item:main_3} compensates for the lack of a length metric and is used in order to prove the uniformization result in Theorem \ref{thm:one-sided_qc}. See Example \ref{example:crack} for some undesirable phenomena that may occur without condition \ref{item:main_3}.

A toy version of this theorem is the analogous $1$-dimensional statement for metric Jordan curves. Namely, if $X$ is a rectifiable metric Jordan curve one can construct an approximation of $X$ by polygonal Jordan curves $(X_n,d_{X_n})$ that converge to $X_n$ in the Gromov--Hausdorff sense and the length of each arc in $X$ is the limit of the lengths of corresponding arcs in $X_n$ as $n\to\infty$; in particular, \ref{item:main_2} holds with equality and with $K=1$. To do so, consider a sufficiently dense set $\{x_1,\dots,x_n\}$ in $X$, in cyclic order, and for each $i\in \{1,\dots,n\}$ consider a line segment $L_i$ of length equal to the diameter of the arc from $x_{i-1}$ to $x_i$ in $X$, where $x_0=x_n$. Then glue metrically the endpoints of each $L_i$ to the points $x_{i-1}$, $x_i$. The resulting space $X_n$ is a polygonal Jordan curve with the desired properties. 

Unlike the $1$-dimensional setting, in the case of surfaces we cannot guarantee the convergence of the Hausdorff $2$-measure in \ref{item:main_2}. Indeed, if $X$ is the unit square in $\R^2$ with the $\ell^\infty$ metric,  then any sequence of polyhedral surfaces $X_n$ satisfying the conclusions of the theorem will necessarily satisfy
$$ \liminf_{n\to\infty} \mathcal H^2(X_n) \geq \frac{4}{\pi} \mathcal H^2(X).$$
We provide the details in Example \ref{example:area}. Based on this example, it is reasonable to conjecture that the optimal constant $K$ in Theorem \ref{thm:extended_polyhedral_approximation} is $4/\pi$.

We briefly outline the proof of \Cref{thm:extended_polyhedral_approximation}. The assumption of locally finite Hausdorff $2$-measure implies that the surface $(X,d)$ contains a large supply of rectifiable curves. Although the metric $d$ does not induce a length metric on $X$ in general, it induces an extended length metric $\bar d$, which is allowed to take the value infinity. In particular, $X$ contains many geodesics with respect to the extended length metric $\bar{d}$, which we refer to as $\bar{d}$-geodesics. The triangulation result in \cite{CR:21}, which was the main technical ingredient for proving the polyhedral approximation theorem in \cite{NR:21}, adapts to give a covering of $X$ by non-overlapping triangular disks whose edges are $\bar{d}$-geodesics. By replacing each triangular disk with a corresponding polyhedral disk, we are able to define the polyhedral surfaces $X_n$ appearing in \Cref{thm:extended_polyhedral_approximation}. If $(X,d)$ is not a length space, then we cannot give $X_n$ the usual polyhedral metric while satisfying \ref{item:main_1} in \Cref{thm:extended_polyhedral_approximation}. However, this is easily fixed by modifying the polyhedral metric on $X_n$ on large scales to match the original metric $d$.

We remark that the use of the induced extended length metric $\bar{d}$ is essential for constructing the surfaces $X_n$. However, the conclusions in \Cref{thm:extended_polyhedral_approximation} relate to the original metric $d$ rather than $\bar{d}$. This is important, since $\bar{d}$ may be poorly behaved in various respects. Most immediately, $\bar{d}(x,y)$ may be infinite for many pairs of points $x,y \in X$, or at least very large relative to $d(x,y)$. Moreover, $(X,\bar{d})$ may fail to have locally finite Hausdorff $2$-measure, even when $\bar{d}$ exists as a metric (rather than just an extended metric). Thus, one cannot reduce the general case to the case of a length metric and appeal directly to the results of \cite{NR:21}. One theme in this paper is the relation between the metric $d$ and the extended length metric $\bar{d}$, which is discussed in detail in \Cref{sec:extended_metric}.

\smallskip

\subsection{Uniformization of metric surfaces}
The \textit{uniformization problem} asks one to determine which metric spaces admit a geometrically well-behaved parametrization by a canonical model space such as the $n$-sphere or $n$-dimensional Euclidean space. Most commonly, we are interested in parametrizations by quasiconformal maps or their generalizations; roughly speaking, these are maps that infinitesimally distort relative distance within some prescribed amount.

Our main result in this direction is concerned with the two-dimensional case of the uniformization problem under the minimal geometric assumption of merely locally finite area.

\begin{thm} \label{thm:one-sided_qc}
Let $X$ be a metric surface of locally finite Hausdorff 2-measure homeomorphic to $\widehat {\mathbb C}$, $\br{\mathbb D}$, or $\C$. Then there is a continuous, surjective, proper, and cell-like map $h\colon \Omega \to X$, where $\Omega$ is $\widehat {\mathbb C}$, $\br{\mathbb D}$, or $\D$ or $\C$, respectively, satisfying the modulus inequality
\begin{equation} \label{equ:qc_inequality}
    \Mod \Gamma \leq \frac{4}{\pi} \Mod h(\Gamma)
\end{equation}
for every path family $\Gamma$ in $\Omega$. 
\end{thm}

Here $\D$ denotes the open unit disk in the complex plane $\C$, and $\widehat\C$ is the Riemann sphere equipped with the spherical distance. A map is \textit{cell-like} if the preimage of each point is a continuum that is contractible in all small neighborhoods; see Section \ref{sec:topological}. A continuous, surjective, proper, and cell-like map $h\colon X\to Y$ between metric surfaces of locally finite Hausdorff $2$-measure is called \textit{weakly quasiconformal} if there exists $K\geq 1$ such that 
\begin{align}\label{ineq:qc}
    \Mod \Gamma\leq K\Mod h(\Gamma)
\end{align}
for every path family $\Gamma$ in $\Omega$; here $\Mod$ refers to $2$-modulus. For surjective maps between homeomorphic compact surfaces the topological assumption of cell-likeness is equivalent to the weaker requirement that $h$ is a \textit{monotone} map; that is, the preimage of each point is a continuum. The topological assumptions on a weakly quasiconformal map $h$ ensure that it is very close to being a homeomorphism; for example, if $X,Y$ have no boundary or if $X,Y$ are compact and homeomorphic to each other, then $h$ is the uniform limit of homeomorphisms. See Section \ref{sec:topological} for more details. It was shown in \cite[Thm.\ 1.4]{NR:21} that \eqref{ineq:qc} may be replaced with the equivalent statement that $h$ lies in the Sobolev space $N^{1,2}_{\loc}(X, Y)$ and
$$g_h(x)^2 \leq K J_h(x)$$
for a.e.\ $x\in X$, where $g_h$ denotes the minimal weak upper gradient of $h$ and $J_h$ is the Jacobian of the map $h$, i.e., the Radon--Nikodym derivative of the measure $\mathcal H^2\circ h$ with respect to $\mathcal H^2$.

\Cref{thm:one-sided_qc} seems to be essentially the strongest result possible for the non-fractal case of the uniformization problem for surfaces, i.e., where the Hausdorff $2$-measure is locally finite. It verifies a conjecture of Rajala and Wenger found as Question 1.3 in \cite{IR:20}. \Cref{thm:one-sided_qc} with the additional assumption that the metric on $X$ is a length metric was recently proved by the authors in \cite{NR:21}, as well as independently by Meier--Wenger \cite{MW:21}. 

In contrast with the classical uniformization theorem,  the map $h$ of Theorem \ref{thm:one-sided_qc}  may fail to be unique up to precomposition by conformal maps. In fact, \cite[Prop.\ 1.5]{NR:21} gives an example of a metric surface $X$ of locally finite Hausdorff $2$-measure admitting weakly quasiconformal parametrizations both by $\D$ and $\C$. 

\Cref{thm:one-sided_qc} deals with the case of simply connected surfaces whose boundary is uncomplicated. Using the classical uniformization theorem, we also establish the following general uniformization result for all metric surfaces of locally finite Hausdorff $2$-measure.

\begin{thm}\label{thm:uniformization_global}
Let $X$ be a metric surface of locally finite Hausdorff $2$-measure. Then there exists a complete Riemannian surface of constant curvature $(Z,g)$ that is homeomorphic to $X$ and a weakly $(4/\pi)$-quasiconformal map from $Z$ onto $X$.    
\end{thm}

The value $4/\pi$ is sharp in both Theorems \ref{thm:one-sided_qc} and \ref{thm:uniformization_global}, as was observed in \cite[Example 2.2]{Raj:17}. A version of this theorem has been proved by Ikonen \cite{Iko:19} for reciprocal surfaces (see \Cref{sec:reciprocal_surfaces}) without boundary, using local quasiconformal coordinates in order to construct isothermal coordinates. This approach cannot be employed in our case, since metric surfaces do not have local quasiconformal coordinates, but only weakly quasiconformal parametrizations that are generally not homeomorphic, as provided by Theorem \ref{thm:one-sided_qc}. We note that our theorem also covers surfaces with boundary.  

\smallskip

\subsection{Previous uniformization results}
If the space $X$ in \Cref{thm:one-sided_qc} satisfies additional good geometric assumptions, then the uniformizing map $h$ can also be shown to satisfy stronger properties. In particular, as an application, we are able to deduce from \Cref{thm:one-sided_qc} two of the main uniformization theorems in the existing literature. The first is the Bonk--Kleiner theorem characterizing Ahlfors $2$-regular quasispheres \cite{BK:02}.

\begin{thm}[Bonk--Kleiner] \label{thm:bonk_kleiner}
Let $X$ be a metric space homeomorphic to $\widehat {\mathbb C}$ that is Ahlfors $2$-regular. Then there is a quasisymmetric homeomorphism from $X$ to $\widehat {\mathbb C}$ if and only if $X$ is linearly locally connected. 
\end{thm}

It is the case that any metric space $X$ as in \Cref{thm:bonk_kleiner} is quasiconvex and hence bi-Lipschitz equivalent to a surface with a length metric that is also Ahlfors $2$-regular and linearly locally connected; see \cite{Sem:96} or \cite{Wil:10} for a proof. Thus \Cref{thm:bonk_kleiner} can also be derived from the weaker version of \Cref{thm:one-sided_qc} as given in \cite{MW:21} or \cite{NR:21}. However, our approach allows one to avoid this technical point regarding quasiconvexity. That Theorem \ref{thm:bonk_kleiner} follows from Theorem \ref{thm:one-sided_qc} is an immediate consequence the considerations in \cite[Sect.\ 6.2]{NR:21}, which are based on classical results that allow the upgrade of quasiconformal maps to quasisymmetric maps; see \cite[Thm.\ 2.5]{LW:20}.

The second uniformization theorem, due to Rajala \cite{Raj:17}, characterizes the situation when there is a quasiconformal homeomorphism from $X$ to $\widehat{\C}$. 

\begin{thm}[Rajala] \label{thm:rajala}
Let $X$ be a metric surface of locally finite Hausdorff 2-measure homeomorphic to $\widehat {\mathbb C}$, $\br{\mathbb D}$, or $\C$. Then there is a quasiconformal map $h\colon \Omega \to X$, where $\Omega$ is $\widehat {\mathbb C}$, $\br{\mathbb D}$, or $\D$ or $\C$, respectively, if and only if $X$ is locally reciprocal. 
\end{thm}

Here, quasiconformality is defined using the so-called geometric definition, which requires the conformal modulus of curve families to be quasipreserved. See \Cref{sec:qc_maps} below.  We give the definition of reciprocity below in the next subsection. The necessity of the theorem follows from the fact that reciprocity is invariant under quasiconformal maps and all Riemannian surfaces are reciprocal. The sufficiency follows from \Cref{thm:one-sided_qc} together with the observation of Meier--Wenger \cite{MW:21} that the local reciprocity of a space allows one to upgrade weak quasiconformality of a local parametrization to quasiconformality. We refer the reader to Section 3 of \cite{MW:21} for the argument.  Analogously, under the reciprocity assumption, we obtain a strengthening of Theorem \ref{thm:uniformization_global} that has been proved by Ikonen \cite{Iko:19} for surfaces without boundary. 

\begin{thm}\label{thm:unif_global_reciprocal}
Let $X$ be a metric surface of locally finite Hausdorff $2$-measure. Then there exists a complete Riemannian surface of constant curvature $(Z,g)$ that is homeomorphic to $X$ and a quasiconformal map from $Z$ onto $X$ if and only if $X$ is locally reciprocal.   
\end{thm}

The results of this paper concern the non-fractal case of the uniformization problem, meaning that the Hausdorff $2$-measure on the surface is locally finite. We also hope that the ideas of this paper can give insight into the uniformization of fractal surfaces. We briefly discuss this topic of current interest. One motivation comes from Cannon's conjecture in geometric group theory, which predicts that every hyperbolic group having the $2$-sphere as its boundary at infinity is a Kleinian group. This is equivalent to the statement that every such $2$-sphere, equipped with some visual metric, is a quasisphere. A seminal work on fractal uniformization is Cannon's combinatorial Riemann mapping theorem \cite{Can:94}, which gives an abstract procedure for uniformizing $2$-spheres. The Bonk--Kleiner paper contains, in addition to \Cref{thm:bonk_kleiner}, a characterization of quasispheres without any restriction. Both of these results, however, seem difficult to apply in practice.  Various constructions have been studied in detail, including Meyer's snowspheres \cite{Mey:10} as well as spheres arising from expanding Thurston maps \cite{BM:17}. See also \cite{RRR:19} for an approach to fractal uniformization by adapting conformal modulus methods.

\smallskip

\subsection{Reciprocal surfaces}  \label{sec:reciprocal_surfaces}
As mentioned above, a further consequence of \Cref{thm:one-sided_qc} is a simplification of Rajala's definition of reciprocal surface.

Let $X$ be a metric surface of locally finite Hausdorff $2$-measure. For a subset $G \subset X$ and disjoint subsets $E,F\subset G$, we define $\Gamma(E,F;G)$ to be the family of curves in $G$ joining $E$ and $F$. A \textit{(topological) quadrilateral} in $X$ is a closed Jordan region $Q$ together with a partition of $\partial Q$ into four non-overlapping edges $\zeta_1,\zeta_2,\zeta_3,\zeta_4\subset \partial Q$ enumerated in cyclic order. When we refer to a quadrilateral $Q$, it will be implicitly understood that there exists such a marking on its boundary. We define $\Gamma(Q)=\Gamma(\zeta_1,\zeta_3;Q)$  and $\Gamma^*(Q) =\Gamma(\zeta_2,\zeta_4;Q)$. We state the definition of reciprocity as originally given in \cite{Raj:17}.
\begin{defi}
A metric surface $X$ is \textit{reciprocal} if there exist constants $\kappa, \kappa' \geq 1$ such that
\begin{align}\label{ireciprocality:12}
    \kappa^{-1}\leq \Mod \Gamma(Q) \cdot \Mod\Gamma^*(Q) \leq \kappa' \quad \textrm{for each quadrilateral $Q\subset X$}
\end{align}
and 
\begin{align}\label{ireciprocality:3}
        &\lim_{r\to 0} \Mod \Gamma( \br B(a,r), X\setminus B(a,R);X )=0 \quad \textrm{for each ball $B(a,R)$.} 
\end{align}
A metric surface $X$ is \textit{locally reciprocal} if each point has an open neighborhood that is reciprocal. 
\end{defi}

Roughly speaking, a surface is reciprocal if the modulus of non-constant curves passing through a point is zero and the modulus of curves joining two opposite sides of every quadrilateral is reciprocal, up to a uniform multiplicative constant, to the modulus of curves joining the other two opposite sides. Note that condition \eqref{ireciprocality:3} can be regarded as a pointwise condition on the surface. In the definition of local reciprocity we do not require that the constants are uniform throughout the surface $X$; however, this is a consequence of a result of Rajala \cite[Thm.\ 1.5]{Raj:17}.

It was shown in \cite{RR:19} by Rajala and the second-named author that the lower bound in \eqref{ireciprocality:12} always holds for some universal constant $\kappa^{-1}$ independent of the surface $X$. Later, Poggi-Corradini and Eriksson-Bique \cite{EBC:21} found the sharp value of this constant to be $\kappa^{-1} = \pi^2/16$; this also follows from Theorem \ref{thm:one-sided_qc}. The sharpness is seen by taking $X$ to be the plane with the $\ell^\infty$ metric. See Example 2.2 in \cite{Raj:17} for details. 

As a consequence of Theorem \ref{thm:one-sided_qc},  we can further improve Rajala's uniformization theorem by showing that \eqref{ireciprocality:3} follows from the upper bound in \eqref{ireciprocality:12}. Thus condition \eqref{ireciprocality:3} is redundant. This answers affirmatively a question of Rajala \cite[Question 17.4]{Raj:17}.  We say that a metric surface is \textit{upper reciprocal} if there exists $\kappa'\geq 1$ such that the upper bound in \eqref{ireciprocality:12} holds for each quadrilateral.

\begin{thm}\label{theorem:reciprocal}
A metric surface of locally finite Hausdorff $2$-measure is reciprocal if and only if it is upper reciprocal. 
\end{thm}
Combining this with Theorem \ref{thm:unif_global_reciprocal}, we see that in order to verify the reciprocity of a surface we only need to check upper reciprocity locally. The idea of the proof of \Cref{theorem:reciprocal} is to show that upper reciprocity by itself is sufficient to promote the weakly quasiconformal parametrization $h$ given by \Cref{thm:one-sided_qc} to a homeomorphism.  In \cite{MW:21} (applying Theorem 3.6 in \cite{LW:20}) and also in \cite[Thm.\ 7.4]{NR:21}, this is accomplished by applying condition \eqref{ireciprocality:3} in the definition of reciprocity instead. Once $h$ is shown to be a homeomorphism, the upper reciprocity condition can be used to upgrade it to a quasiconformal map, as was shown by Meier--Wenger \cite[Sect.\ 3]{MW:21}. Our task then is, assuming a non-homeomorphic weakly quasiconformal parametrization, to find a sequence of quadrilaterals $Q$ such that the product of the moduli of  $\Gamma(Q)$ and $\Gamma^*(Q)$ are unbounded. The proof of Theorem \ref{theorem:reciprocal} would be simpler if we knew that condition \eqref{ireciprocality:3} can only fail on a totally disconnected set. However, we show in Example \ref{example:continuum} that this is not the case. 

\begin{prop}\label{prop:continuum}
There exists a metric surface of locally finite Hausdorff $2$-measure such that \eqref{ireciprocality:3} fails at all points in a non-degenerate continuum.
\end{prop}
Whether such a space exists had been asked by the second-named author as Question 5.6 in \cite{RRR:19}, and this construction provides an affirmative answer.

Moreover, combining Theorem \ref{theorem:reciprocal} with the uniformization result of Theorem \ref{thm:one-sided_qc} or Theorem \ref{thm:rajala} and with a result of Ikonen \cite{Iko:21}, we also show that in order to obtain reciprocity at points of $\partial X$ it suffices to verify condition \eqref{ireciprocality:3} rather than upper reciprocity.

\begin{thm}\label{theorem:reciprocal:interior}
A metric surface $X$ of locally finite Hausdorff $2$-measure is reciprocal if and only if $\Int(X)$ is upper reciprocal and \eqref{ireciprocality:3} holds at each point of $\partial X$.
\end{thm}

We remark that without requiring any condition on $\partial X$, the reciprocity of $\Int(X)$ does not imply the reciprocity of $X$ in general; this was observed in \cite[Sect.\ 1.1]{Iko:21}. 

Conversely, it is natural to ask whether upper reciprocity is implied by condition \eqref{ireciprocality:3}. We show in Example \ref{example:rec3} that this is not the case. 

\begin{prop}\label{prop:example:rec3}
There exists a metric surface $X$ of locally finite Hausdorff $2$-measure such that \eqref{ireciprocality:3} holds at each point, but $X$ is not reciprocal. Moreover, $X$ can be written as the union of two reciprocal surfaces.  
\end{prop}

\smallskip

\subsection{Minimal surfaces}
As a corollary to Theorem \ref{thm:one-sided_qc}, we obtain a result on the existence of \textit{minimal disks} or \textit{solutions to Plateau's problem} in metric spaces. This topic has been studied in great depth by Lytchak--Wenger and collaborators in \cite{FW:21,GW:20,LW:17,LW:17b,LW:18a}, and we direct the reader to these references for definitions. The following corollary was established in \cite{NR:21} for length surfaces and in fact was derived from the version Theorem \ref{thm:one-sided_qc} for length surfaces. Its proof remains unchanged under the more general setting.

\begin{cor}\label{cor:plateau}
Let $X$ be a metric surface of finite Hausdorff $2$-measure homeomorphic to a topological closed disk and let $\Gamma=\partial X$. Then Plateau's problem for $\Gamma$ has a solution.
\end{cor}

We pose the following natural question.

\begin{ques}
Does Plateau's problem have a solution for every metric Jordan curve $\Gamma$ of finite Hausdorff $2$-measure?
\end{ques}
This was established for metric Jordan curves of finite length by Lytchak--Wenger \cite[Cor.\ 1.5]{LW:17}. In order to answer the question, by Corollary \ref{cor:plateau} it suffices to solve the following problem.
\begin{prob} 
Let $\Gamma$ be a metric Jordan curve of finite Hausdorff $2$-measure.  Then there exists a metric space $X$ homeomorphic to a topological closed disk with $\mathcal H^2(X) \leq K (\diam(\partial X)^2+\mathcal H^2(\partial X))$ for an absolute constant $K>0$  such that $\Gamma$ embeds isometrically into $\partial X$.
\end{prob}

\smallskip

\subsection{Outline} 
Section \ref{sec:preliminaries} contains the required preliminaries. In Section \ref{sec:extended_metric} we introduce the extended length metric $\bar d$ associated to a metric surface $(X,d)$ and we study its relation to $d$. The main result in this section is Theorem \ref{thm:triangle_area}, which provides an estimate for the area of triangles in $X$ whose edges are $\br d$-geodesics. In Section \ref{sec:triangulations}, we first sketch an argument for triangulating metric surfaces of locally finite area, by modifying slightly the argument of \cite{CR:21} for triangulating length surfaces. Then, we use the triangulation result to prove the polyhedral approximation theorem, Theorem \ref{thm:extended_polyhedral_approximation}.  This completes the first half of the paper, which can be read independently of the later sections.

In Section \ref{sec:uniformization}, we establish the uniformization result of Theorem \ref{thm:one-sided_qc} in the case of compact surfaces. In this section we also include several further required preliminaries in Gromov--Hausdorff convergence, weakly quasiconformal maps, and classical uniformization theory. Section \ref{sec:global} establishes the global uniformization result of Theorem \ref{thm:uniformization_global}, based on the local results of Section \ref{sec:uniformization}. Moreover, we show how to obtain the minimal value  $4/\pi$ of the quasiconformal dilatation. This section also contains several topological preliminaries related to cell-like and monotone maps between surfaces. The results of Sections \ref{sec:uniformization} and \ref{sec:global} rely only on the statement of the polyhedral approximation theorem (Theorem \ref{thm:extended_polyhedral_approximation}) and can be read independently of the previous sections.

In Section \ref{sec:reciprocal} we prove Theorem \ref{theorem:reciprocal} and Theorem \ref{theorem:reciprocal:interior}. This section relies only on Theorem \ref{thm:one-sided_qc} and can also be read independently. Finally, in Section \ref{sec:examples} we present four examples. The first two examples are related to the polyhedral approximation theorem and show the importance of the retractions in conclusion \ref{item:main_3} of Theorem \ref{thm:extended_polyhedral_approximation}, as well as the sharpness of conclusion \ref{item:main_2}, in the sense that the areas of the approximating surfaces need not converge to the area of the limiting surface. The last two examples of the section justify Proposition \ref{prop:example:rec3} and Proposition \ref{prop:continuum}, respectively.

\bigskip

\section{Preliminaries} \label{sec:preliminaries}

\subsection{Metric spaces}

We refer the reader to \cite{BH:99} and \cite{BBI:01} for the basics of metric geometry. Let $X$ be a set. A function $d \colon X \times X \to [0,\infty)$ is a \textit{metric} on $X$ if it is symmetric, satisfies the triangle inequality, and has the property that $d(x,y) = 0$ if and only if $x=y$. A function $d$ satisfying the same properties but potentially taking the value $\infty$ is called an \textit{extended metric}. 

Let $\gamma \colon I \to X$ be a curve in a metric space $X$, where $I$ is an interval. The length of $\gamma$ with respect to $d$, denoted by $\ell_d(\gamma)$, is defined as 
\[ \ell_d(\gamma) = \sup \sum_{i=1}^n d(\gamma(t_{i-1}),\gamma(t_i)), \]
the supremum taken over all finite increasing sequences $t_0 < t_1 < \cdots < t_n$ in $I$. If the metric is clear from the context, we may write $\ell(\gamma)$ in place of $\ell_d(\gamma)$. A curve is \textit{rectifiable} if it has finite length. A metric $d$ on $X$ is a \textit{length metric} if $d(x,y) = \inf \ell_d(\gamma)$ for all $x,y \in X$, the infimum taken over all curves $\gamma$ from $x$ to $y$. 

A \textit{Jordan curve} (resp.\ \textit{Jordan arc}) in $X$ is an embedding of the unit circle $\mathbb S^1$ (resp.\ the unit interval $[0,1]$) into $X$. We also use the alternative terminology simple curve and simple arc, respectively. The \textit{trace} of a path $\gamma\colon I\to X$ is the set $\gamma(I)$ and is denoted by $|\gamma|$.

We use the notation $B_d(x,r)$ for the open ball $\{y\in X: d(x,y)<r\}$, $\br B_d(x,r)$ for the closed ball, and $S_d(x,r)$ for the sphere $\{y\in X: d(x,y)=r\}$. Again, the subscript $d$ may be dropped if the metric is clear from the context. 

For any metric space $X$ and $s>0$, the \textit{Hausdorff $s$-measure} of a set $A \subset X$ is defined by
\[\mathcal{H}^s(A) = \lim_{\delta\to 0} \mathcal H^s_{\delta}(A),\]
where
\[\mathcal H^s_\delta(A) =\inf \left\{ \sum_{j=1}^\infty C(s) \diam(A_j)^s\right\} \]
and the infimum is taken over all collections of sets $\{A_j\}_{j=1}^\infty$ such that $A \subset \bigcup_{j=1}^\infty A_j$ and $\diam(A_j) < \delta$ for each $j$. Here $C(s)$ is a positive normalization constant, chosen so that the Hausdorff $n$-measure coincides with Lebesgue measure in $\R^n$. The quantity $\mathcal H^s_{\delta}(A)$ is called the \textit{$\delta$-Hausdorff $s$-content of $A$}. If we need to emphasize the metric $d$ being used for the Hausdorff $s$-measure, we write $\mathcal{H}_{d}^s$ instead of $\mathcal{H}^s$. 

A map $f\colon X \to Y$ between metric spaces is \textit{bi-Lipschitz} if there exists $L\geq 1$ such that
\[L^{-1}d_X(x,y) \leq d_Y(f(x),f(y)) \leq L d_X(x,y)\]
for all $x,y \in X$. In this case, we say that $f$ is \textit{$L$-bi-Lipschitz}. A map $f\colon X \to Y$ is  \textit{Lipschitz} if the right inequality holds for all $x,y \in X$. In this case, we say that $f$ is \textit{$L$-Lipschitz}.

We use $\partial X$ to denote the boundary of a manifold $X$ and $\Int(X)$ to denote its interior. Throughout this paper, unless otherwise specified, the terms \textit{boundary} and \textit{interior} refer to manifold boundary and interior rather than topological boundary and interior. 

\smallskip

\subsection{Convergence}\label{sec:convergence}
Let $X$ be a metric space and let $E\subset X$ and $\varepsilon>0$. We denote by $N_{\varepsilon}(E)$ the open $\varepsilon$-neighborhood of $E$. We say that $E$ is \textit{$\varepsilon$-dense} (in $X$) if for each $x\in X$ we have $d(x,E)<\varepsilon$ or equivalently $N_{\varepsilon}(E)=X$. A map $f \colon X \to Y$ (not necessarily continuous) between metric spaces is an \textit{$\varepsilon$-isometry} if $f(X)$ is $\varepsilon$-dense in $Y$ and $|d_X(x,y) - d_Y(f(x),f(y))| < \varepsilon$ for each $x,y \in X$.

We define the \textit{Hausdorff distance} of two sets $E,F\subset X$ to be the {infimal value} $r>0$ such that $E\subset N_r(F)$ and $F\subset N_r(E)$. We denote the Hausdorff distance by $d_H(E,F)$. A sequence of sets $E_n\subset X$ \textit{converges in the Hausdorff sense} to a set $E\subset X$ if $d_H(E_n,E)\to 0$ as $n\to\infty$.

{The \textit{Gromov--Hausdorff distance} between two metric spaces $X,Y$ is defined as the infimal value $r>0$ such that there is a metric space $Z$ with subsets $\widetilde{X}, \widetilde{Y} \subset Z$ such that $X$ and $Y$ are isometric to $\widetilde{X}$ and $\widetilde{Y}$, respectively, and $d_H(\widetilde{X}, \widetilde{Y}) < r$. This is denoted by $d_{GH}(X,Y)$. We say that a sequence of metric spaces $X_n$ \textit{converges in the Gromov--Hausdorff sense} to a metric space $X$ if $d_{GH}(X_n,X) \to 0$ as $n \to \infty$. By \cite[Cor.\ 7.3.28]{BBI:01}, this is equivalent to the property that} there exists a sequence of $\varepsilon_n$-isometries $f_n\colon X_n\to X$, where $\varepsilon_n>0$ and $\varepsilon_n\to 0$ as $n\to\infty$. In this case, we say that $f_n$, $n\in \N$, is an \textit{approximately isometric sequence}.

\smallskip

\subsection{Metric Sobolev spaces}
Let $X$ be a metric space and $\Gamma$ be a family of curves in $X$. A Borel function $\rho\colon X \to [0,\infty]$ is \textit{admissible} for the path family $\Gamma$ if $\int_{\gamma}\rho\, ds\geq 1$
for all locally rectifiable paths $\gamma\in \Gamma$. We define the \textit{$2$-modulus} of $\Gamma$ as 
$$\Mod \Gamma = \inf_\rho \int_X \rho^2 \, d\mathcal H^2,$$
where the infimum is taken over all admissible functions $\rho$ for $\Gamma$. By convention, $\Mod \Gamma = \infty$ if there are no admissible functions for $\Gamma$. Observe that we consider $X$ to be equipped with the Hausdorff 2-measure. This definition may be generalized by allowing for an exponent different from $2$ or a different measure, though this generality is not needed for this paper.

Let $h\colon X\to Y$ be a map between metric spaces. We say that a Borel function $g\colon X\to [0,\infty]$ is an \textit{upper gradient} of $h$ if 
\begin{align}\label{ineq:upper_gradient}
    d_Y(h(a),h(b)) \leq \int_{\gamma} g \, ds
\end{align}
for all $a,b\in X$ and every locally rectifiable path $\gamma$ in $X$ joining $a$ and $b$. This is called the \textit{upper gradient inequality}. If, instead the above inequality holds for all curves $\gamma$ outside a curve family of $2$-modulus zero, then we say that $g$ is a \textit{weak upper gradient} of $h$. In this case, there exists a curve family $\Gamma_0$ with $\Mod \Gamma_0=0$ such that all paths outside $\Gamma_0$ and all subpaths of such paths satisfy the upper gradient inequality.

{We equip the space $X$ with the Hausdorff $2$-measure $\mathcal{H}^2$.} Let $L^p(X)$ denote the space of $p$-integrable Borel functions from $X$ to the extended real line $\widehat{\mathbb{R}}$, where two functions are identified if they agree $\mathcal{H}^2$-almost everywhere. The Sobolev space $N^{1,p}(X,Y)$ is defined as the space of Borel maps $h \colon X \to Y$ with a weak upper gradient $g$ in $L^p(X)$ such that the function $x \mapsto d_Y(y,h(x))$ is in $L^p(X)$ for some $y \in Y$, again where two functions are identified if they agree almost everywhere. If $Y=\R$, we simply write $N^{1,p}(X)$. The spaces $L_{\loc}^p(X)$ and $N_{\loc}^{1,p}(X, Y)$ are defined in the obvious manner. See the monograph \cite{HKST:15} for background on metric Sobolev spaces.

\bigskip

\section{The extended length metric} \label{sec:extended_metric}

In this section, we develop the properties of the extended length metric, which are all new ingredients in this paper. The main result in this section, \Cref{thm:triangle_area}, gives an estimate on the area enclosed by a triangle whose edges are $\bar{d}$-geodesics. This estimate plays the same role as the Besicovitch inequality in \cite{NR:21}. As in \cite{NR:21}, this estimate is based on bi-Lipschitz embedding an arbitrary triangle into the plane. The innovation here is that, although the bi-Lipschitz embedding is with respect to the extended length metric $\bar{d}$, our area estimate uses the Hausdorff $2$-measure derived from the original metric $d$. 

\smallskip

\subsection{Definition and basic properties}

Let $(X,d)$ be a metric space. One can obtain an extended length metric $\bar{d}\colon X\times X\to [0,\infty]$ by defining
\[\bar{d}(x,y) = \inf \ell_d(\gamma),\]
where the infimum is taken over all curves from $x$ to $y$. One can check that $\bar{d}$ is indeed an extended length metric. It may be that there are no rectifiable curves between two given points $x,y \in X$, in which case we have $\bar{d}(x,y) = \infty$. 

If $X$ is a surface with locally finite Hausdorff $2$-measure, then there is a totally disconnected set $E \subset X$ such that any two points $x,y \in X \setminus E$ can be joined by a rectifiable curve. In particular, $\bar{d}(x,y)<\infty$ for such $x,y$. The proof of this statement is a simple application of the co-area inequality, as noted by Rajala in \cite[Sect. 3]{Raj:17}. Hence the co-area inequality connects the assumption of locally finite Hausdorff $2$-measure to the geometry of the metric space. We state a $2$-dimensional version of the co-area inequality sufficient for our purposes. See \cite{EH:21} for a full proof of the general co-area inequality. 

\begin{lemm}[Co-area inequality]
Let $X$ be a metric space and $L>0$. For any $L$-Lipschitz function $f \colon X \to \mathbb{R}$ and any Borel function $g \colon X \to [0,\infty]$,
\begin{align} \label{equ:coarea}
   \int_{\mathbb{R}} \int_{f^{-1}(t)} g(s)\, d\mathcal{H}^1(s)\,dt \leq \frac{4L}{\pi} \int_X g(x)\, d\mathcal{H}^2(x). 
\end{align}
\label{lemm:coarea}
\end{lemm}

Given that $X$ has locally finite Hausdorff $2$-measure, a consequence of \Cref{lemm:coarea} is that for each point $x \in X$, almost every metric sphere $S(x,r)$ has finite Hausdorff $1$-measure and hence {contains} rectifiable curves. This guarantees an abundance of rectifiable curves in $X$.

Throughout this entire section, distances are measured with respect to the original metric $d$ of $X$, unless otherwise indicated. For example, $B(x,r)$ denotes the ball $\{y\in X: d(x,y)<r\}$ and $B_{\bar{d}}(x,r)=\{y\in X: \bar{d}(x,y)<r\}$. 

Recall that a \textit{geodesic} between two points $x,y \in X$ is a curve $\gamma\colon [a,b] \to X$ of finite length such that $\gamma(a) = x$, $\gamma(b) = y$ and $\ell(\gamma) = d(x,y)$. On a metric space $(X,d)$, our main interest will be in geodesics with respect to the induced length metric $\bar{d}$. We will refer to such a curve as a \textit{$\bar{d}$-geodesic}. Note that a $\bar{d}$-geodesic may fail to exist between two given points, either because no curve of finite length connecting $x$ and $y$ exists or because there is no curve having minimal length. Moreover, a $\bar{d}$-geodesic between a given pair of points is not necessarily unique, as is the case for geodesics in general. However, as observed, the assumption of locally finite Hausdorff $2$-measure is enough to guarantee that a surface carries many rectifiable curves. In particular, it is shown in \cite[Prop.\ 3.5]{Raj:17} that for any two disjoint continua the family of curves connecting them has positive modulus. 

In general, $\bar{d}$ may fail to be continuous with respect to the metric $d$. However, a consequence of the Arzel\`a--Ascoli theorem is that $\bar{d}$ is lower semi-continuous with respect to $d$ in the following sense.

\begin{lemm} \label{lemm:lsc}
Let $(X,d)$ be a metric surface of locally finite Hausdorff $2$-measure. For each $x_0\in X$ the function $u\colon X\to [0,\infty]$ defined by $u(x)=\bar{d}(x,x_0)$ is lower semi-continuous. Moreover, for each $r>0$ the  set $B_{\bar{d}}(x_0,r)$ is connected and $\br B_{\bar{d}}(x_0,r)$ is the closure of $B_{\bar{d}}(x_0,r)$. In particular, $\br B_{\bar{d}}(x_0,r)$ is also connected. 
\end{lemm}

All topological notions in the statement refer to the original topology of $X$ induced by $d$. The proof relies on the following separation lemma.

\begin{lemm}\label{lemma:separation}
Let $X$ be a metric surface that is homeomorphic to a topological closed disk. Let $f\colon X\to [0,\infty)$ be a continuous function such that $f^{-1}(0)$ is non-empty and $f^{-1}(t)$ has finite Hausdorff $1$-measure for a.e.\ $t>0$. Then for each component $A$ of $f^{-1}(0)$, each $x_0\in \partial X\setminus A$, and for a.e.\ $r\in (0,f(x_0))$ there exists a continuum $E\subset f^{-1}(r)$ with $\mathcal H^1(E)<\infty$ that separates $A$ from $x_0$. Moreover, there exists a closed Jordan region $U\subset X\setminus \{x_0\}$ containing a neighborhood of $A$ such that $\partial U$ coincides with $E$ or is the union of $E$ and an arc of $\partial X$. 
\end{lemm}

This lemma can be applied to Lipschitz functions on metric spaces of finite Hausdorff $2$-measure, since the co-area inequality of Lemma \ref{lemm:coarea} implies that a.e.\ level set of a Lipschitz function has finite Hausdorff $1$-measure.

\begin{proof}
Let $\delta_0=f(x_0)$. We apply \cite[Thm.\ 1.5]{Nt:monotone}, which implies that for a.e.\ $r\in (0,\delta_0)$,  each component of $f^{-1}(r)$ is either a point or a rectifiable Jordan curve or a rectifiable Jordan arc. We fix such $r$ and let $E$ be a component of $f^{-1}(r)$ that separates $A$ from $x_0$; such a component exists since $f^{-1}(r)$ separates $A$ from $x_0$.  If $E$ is a Jordan curve, then it necessarily bounds a closed Jordan region $U\subset X\setminus \{x_0\}$ containing $A$ in its interior. If $E$ is a Jordan arc, then by passing to a subarc we may assume that only the endpoints of $E$ lie in $\partial X$ and $E$ still separates $A$ from $x_0$. The arc $E$, together with an arc of $\partial X$ bound a closed Jordan region $U\subset X\setminus \{x_0\}$ that contains a neighborhood of $A$.
\end{proof}

\begin{proof}[Proof of Lemma \ref{lemm:lsc}]
Let $x_n\in X$, $n\in \N$, be a sequence converging to a point $x\in X$. Our goal is to show that $u(x)\leq \liminf_{n\to\infty} u(x_n)$. By passing to a subsequence, suppose that $\lim_{n\to\infty}u(x_n)$ exists. The statement is trivial if $x=x_0$ or if $\lim_{n\to\infty}u(x_n)=\infty$, so suppose that $x\neq x_0$ and $\lim_{n\to\infty}u(x_n)=M$ for some $M\in (0,\infty)$. Let $\gamma_n$ be a sequence of paths joining $x_0$ to $x_n$ such that $\lim_{n\to\infty}\ell(\gamma_n)=M$.

Suppose first that $x\in \Int(X)$. Consider a closed Jordan region $Y\subset \Int(X)\setminus \{x_0\}$ containing $x$ in its interior and define the Lipschitz function $f(y)=d(x,y)$ on $Y$.  Let $y_0\in \partial Y$ be arbitrary and $r_0=\dist(x,\partial Y)$. By Lemma \ref{lemma:separation}, there exists $r<r_0$ and a rectifiable Jordan curve $E\subset f^{-1}(r)\subset Y$ that bounds a closed Jordan region $U\subset Y\setminus \{y_0\}$ containing $x$.  If $x\in \partial X$, we consider a closed Jordan region $Y\subset X\setminus \{x_0\}$ such that $Y\cap \partial X$ is a non-degenerate Jordan arc containing $x$ in its interior. Define $r_0=\dist(x,\partial Y \cap \Int(X))>0$. Again, we let $f(y)=d(x,y)$ on $Y$ and fix a point $y_0\in \partial Y\cap \Int(X)$. By Lemma \ref{lemma:separation}, there exists $r<r_0$ and a rectifiable Jordan arc $E\subset f^{-1}(r)$ with endpoints on $\partial Y$ that separates $x$ from $y_0$. The endpoints of $E$ necessarily lie on $\partial Y\cap \partial X$ since $r<r_0$. The arc $E$ together with an arc of $\partial X$ bound a closed Jordan region $U\subset Y$ that contains an open neighborhood of $x$. 

Summarizing, in both cases there exists a continuum $E$ that is a rectifiable Jordan curve or a Jordan arc and separates $x$ from $x_0$. Moreover, the curve $E$ is the topological boundary of a compact neighborhood $U$ of $x$.

For each $n\in \N$, we parametrize the curve $\gamma_n\colon [0,\ell(\gamma_n)]\to X$ by arclength so that $\gamma_n(0)=x_0$ and $\gamma_n(\ell(\gamma_n))=x_n$. If $n$ is sufficiently large so that $x_n\in U$, then $\gamma_n$ intersects $E$ at a point $y_n= \gamma_n(t_n)$; we assume that $t_n$ is the largest possible parameter with this property, so that the trace of $\eta_n\coloneqq \gamma_n|_{[t_n,\ell(\gamma_n)]}$ is contained in $U$. By the Arzel\`a--Ascoli theorem \cite[Thm.\ 2.5.14]{BBI:01} and the compactness of $U$, after passing to a subsequence, we may assume that the paths $\eta_n$ converge uniformly (in an appropriate sense since the domains are variable) to a path $\eta \colon [t,M] \to U$, where $t$ is an accumulation point of $t_n$. Moreover, $\ell(\eta)\leq M-t$. We note that $\eta(t)\in E$ and $y_n=\eta_n(t_n) \to \eta(t)$ as $n\to\infty$. If we parametrize $E$ injectively by arclength, then we see that there exists a subpath $\zeta_n$ of $E$ that connects $y_n$ to $\eta(t)$ with $\ell(\zeta_n)\to 0$ as $n\to\infty$. We now concatenate  $\gamma_n|_{[0,t_n]}$ with $\zeta_n$ and with $\eta$ to obtain a path $\gamma$ that connects $x_0$ to $x$ and satisfies
\begin{align*}
    \bar{d}(x_0,x)\leq \ell(\gamma)= t_n + \ell(\zeta_n) + \ell(\eta)\leq t_n+\ell(\zeta_n)+M-t
\end{align*}
By letting $n\to\infty$ we see that $\bar{d}(x_0,x) \leq M$. This completes the proof of the lower semi-continuity.

For the second part of the lemma, note that for each $r>0$, if $\bar{d}(x,x_0)<r$, then there exists a rectifiable curve $\gamma$ connecting $x_0$ to $x$ with $\ell(\gamma)<r$. It is trivial that for each point $y\in |\gamma|$ we have $\bar{d}(y,x_0)<r$, so $|\gamma|\subset B_{\bar{d}}(x_0,r)$ and the open ball $B_{\bar{d}}(x_0,r)$ is connected. By the lower semi-continuity of $u$, the closed ball $\br B_{\bar{d}}(x_0,r)=\{y\in X: \bar d(x_0,y)\leq r\}$ is a closed set that contains $B_{\bar{d}}(x_0,r)$, so it also contains its closure. Conversely,  if $\bar{d}(x,x_0)\leq r$, then there exists a sequence of curves $\gamma_n$ connecting $x_0$ to $x$ with $\ell(\gamma_n)< r+1/n$. It follows that there exists a sequence of subcurves $\eta_n$ of $\gamma_n$ connecting $x_0$ to points $y_n$ with $\ell(\eta_n)<r$ and $d(x,y_n)<1/n$.  Therefore, $y_n\in B_{\bar{d}}(x_0,r)$ and $y_n\to x$ as $n\to\infty$, which shows that $x$ lies in the closure of $B_{\bar{d}}(x_0,r)$, as desired. 
\end{proof}

\smallskip

\subsection{Basic area estimate}

The use of $\bar{d}$-geodesics allows us to extend various standard geometric notions for length spaces to the non-length space case. We define a \textit{triangular disk} to be a closed Jordan domain in a metric surface whose boundary consists of three $\bar{d}$-geodesics as defined above. Similarly, a \textit{polygonal disk} is a closed Jordan domain with piecewise $\bar{d}$-geodesic boundary. Each boundary $\bar{d}$-geodesic is called an \textit{edge}, and each endpoint is called a \textit{vertex}. Where there is no room for confusion, we will sometimes refer to triangular and polygonal disks as simply \textit{triangles} and \textit{polygons}. The next theorem is the main result of this section and gives a statement about the Hausdorff $2$-measure of an arbitrary triangular disk. 

\begin{thm} \label{thm:triangle_area}
Let $(X,d)$ be a metric surface of locally finite Hausdorff $2$-measure and let $T$ be a triangular disk in $X$. If  $\Omega\subset \R^2$ is a closed Jordan region such that for some $L>0$ there exists an $L$-Lipschitz map $f\colon (\partial T,\bar{d})\to \partial \Omega$ of non-zero topological degree, then 
$$\mathcal H^2(\Omega)\leq K L^2 \mathcal H^2(T)$$
for some uniform constant $K>0$.
\end{thm}

For the proof we will need to employ a co-area type result for functions $f \colon X \to \mathbb{R}$ contained in the Sobolev space $N^{1,p}(X)$ for some $p \geq 1$ but not necessarily Lipschitz. We are not aware of any generalization of \Cref{lemm:coarea} for $N^{1,p}(X)$ functions that covers our situation; presumably, the $L$ in \eqref{equ:coarea} would be replaced by the minimal weak upper gradient of $f$ inside the integral on the right-hand side. However, a weaker recent result of Eriksson-Bique--Poggi-Corradini \cite{EBC:21} turns out to be sufficient. We state a restricted version that covers our situation.

\begin{lemm} \label{lemm:EBPC}
Let $X$ be a compact metric space of finite Hausdorff $2$-measure and $E,F \subset X$ be disjoint continua. Suppose that $u\colon X\to \R$ is a function in $N^{1,1}(X)$ with  $E\subset u^{-1}(a)$ and $F\subset u^{-1}(b)$ for some $a<b$. Then there exists a Lipschitz function $v\colon X \to \mathbb{R}$ such that $E\subset v^{-1}(a)$, $F\subset v^{-1}(b)$, and 
\begin{equation} \label{equ:coarea2}
  \int_a^b \mathcal{H}^1(v^{-1}(t))\,dt \leq 2 \int g_u\, d\mathcal{H}_X^2 .
\end{equation}
\end{lemm}
In fact, the formulation in \cite{EBC:21} involves $\partial (v^{-1}((-\infty,t)))$ instead of $v^{-1}(t)$ in the integral. However, the two sets are the same for a.e.\ $t\in \R$, because the set of local extremal values of a real-valued function on a separable space is countable \cite[Lemma 2.10]{Nt:monotone}.

Our method of proof of Theorem \ref{thm:triangle_area} is based on a particular bi-Lipschitz embedding of an arbitrary metric triangle into the plane that we gave in \cite[Sect.\ 3]{NR:21}, which we now review. A metric triangle $\Delta$ is by definition a metric space homeomorphic to a Jordan curve that consists of three non-overlapping edges $I_1, I_2, I_3$ in cyclic order, each isometric to an interval. Suppose that the endpoints of the edges are $p_1,p_2,p_3$, where $I_j$ joins $p_j$ to $p_{j+1}$; here and below the indices are taken $\Mod 3$. For each $x \in I_j$, let $\widehat{I}(x)$ denote the union of the other two edges of $\Delta$. If $x$ belongs to two edges, then assign $x$ to one of the $I_j$ arbitrarily. Note in particular that $x \in \widehat{I}(x)$ in this case. Let $\bar{\Delta}$ denote a corresponding tripod in $\C$ having central vertex $0$ and outer vertices $u_j = (p_{j+1} \cdot p_{j+2})_{p_j}e^{(2j-2)\pi i/3}$, where $(p\cdot q)_r$ denotes the Gromov product
 \[(p\cdot q)_r = \frac{1}{2}(d(p,r) + d(q,r) - d(p,q)).\]
There is a unique 1-Lipschitz projection map from $\Delta$ to $\bar{\Delta}$ taking each vertex $p_j$ to the vertex $u_j$; we denote the image of $x$ under this map by $\bar{x}$. Let $v_j = e^{\pi(2j-1) i/3}$ for $j \in \{1,2,3\}$. As shown in \cite[Sect.\ 3]{NR:21}, there exists an embedding $F\colon \Delta \to \mathbb{C}$ given by the formula
\[F(x) = \bar{x} + \text{dist}(x,\widehat{I}(x))v_j, \]
where $x \in I_j$; see Figure \ref{fig:embedding}. Proposition 3.2 in \cite{NR:21} states that the map $F$ is $4$-bi-Lipschitz for any metric triangle $\Delta$.

\begin{figure}
    \centering
    \begin{tikzpicture}
           \node () at (0,0) {\includegraphics[width=.35\textwidth]{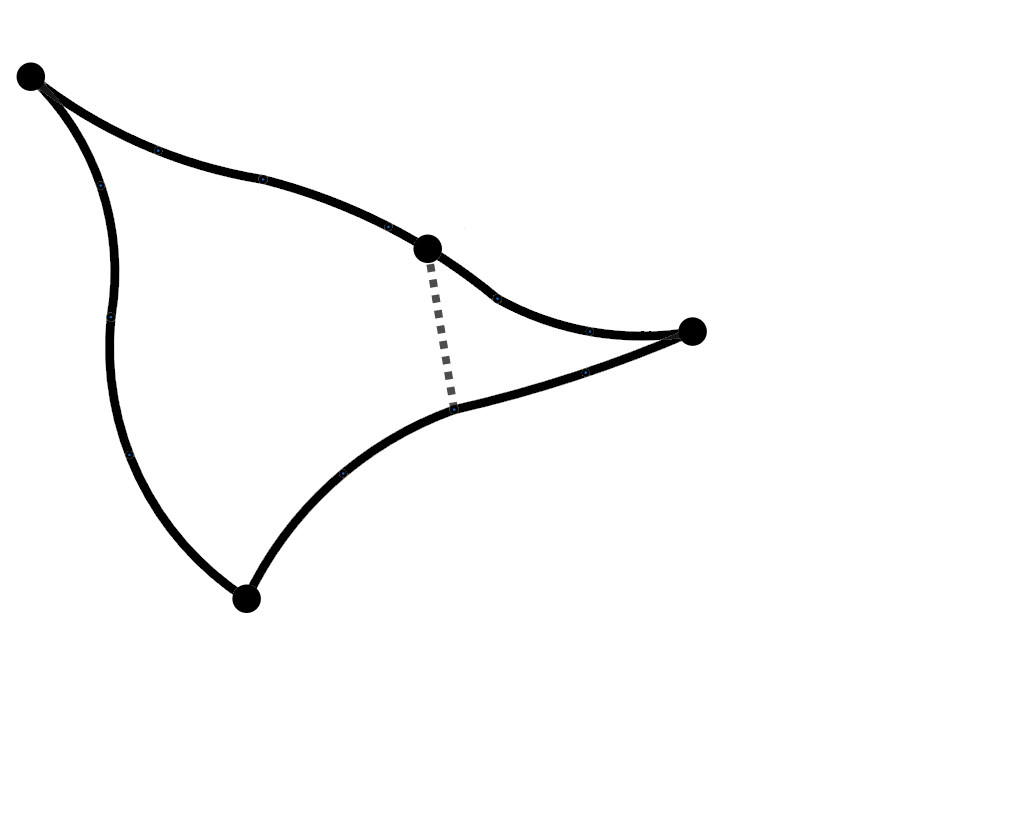}
           \includegraphics[width=.4\textwidth]{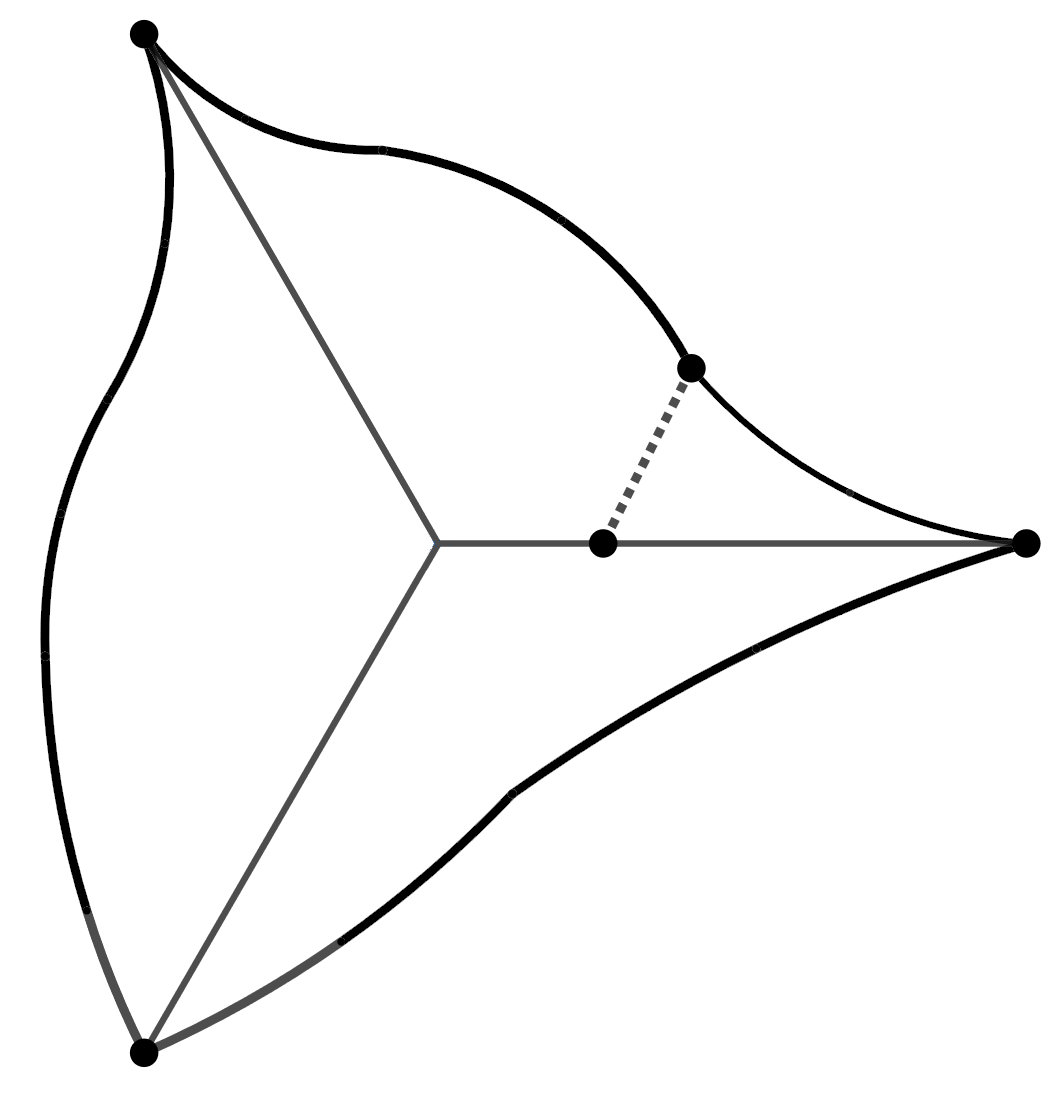}};
           \node () at (-2.7,0.1) {$x$};
           \node () at (-2.9,-1.3) {$\widehat I(x)$};
           \node () at (3.4,1.2) {$F(x)$};
           \node () at (2.5,-0.25) {$\bar x$};
           \node () at (-4.,-.2) {$\Delta$};
           \node () at (1.4,0.2) {$\bar \Delta$};
           \draw [->] (-2.7,.5) to [out=30,in=150] (-.5, .5);
           \node () at (-1.7,1.1) {$F$};
    \end{tikzpicture}
    \caption{The embedding $F$ from an arbitrary metric triangle $\Delta$ into the plane.}
    \label{fig:embedding}
\end{figure}
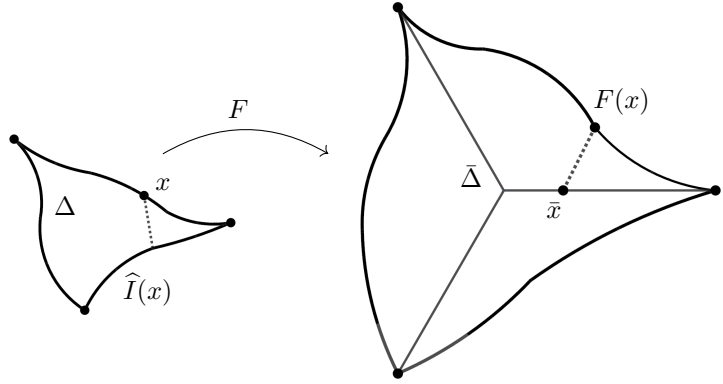

Observe that if $T$ is a triangular disk in a metric surface $(X,d)$ of locally finite Hausdorff $2$-measure, then $(\partial T,\bar d)$ is a metric triangle. Thus, we can exploit the above device to embed such metric triangles into the plane.

\begin{proof}[Proof of \Cref{thm:triangle_area}]
Let  $F \colon (\partial T, \bar{d}) \to \mathbb{R}^2$ be the $4$-bi-Lipschitz embedding given above and $\Omega$ be the closed Jordan domain bounded by $F(\partial T)$. Suppose for the moment that $\mathcal H^2(\Omega)\leq K \mathcal H^2(T)$ for a uniform constant $K>0$. Let $f\colon (\partial T,\bar{d})\to \partial \Omega'$ be an arbitrary $L$-Lipschitz map of non-zero topological degree, where $\Omega'\subset \R^2$ is a closed Jordan region. Then $f\circ F^{-1}\colon \partial \Omega \to \partial \Omega'$ is $4L$-Lipschitz and has non-zero topological degree. By the McShane--Whitney extension theorem \cite[Cor.\ 2.4]{Hei:05}, there exists an extension $g\colon \R^2\to \R^2$ of $f\circ F^{-1}$ that is $4\sqrt{2}L$-Lipschitz. Note that $g(\Omega)\supset \Omega'$ since $g|_{\partial \Omega}\colon \partial \Omega\to \partial \Omega'$ has non-zero topological degree. Therefore, 
$$\mathcal H^2(\Omega')\leq 32L^2\mathcal H^2(\Omega)\leq 32K L^2 \mathcal H^2(T).$$

Now, we focus on showing the statement for the map $F$. For a given point $x \in \partial T$, we write $\widetilde{x} = F(x)$ and $B_{\widetilde{x}} = \overline{B}(\widetilde{x},r_x)$, where $r_x = \bar{d}(x, \widehat{I}(x))$. Note that $r_x>0$ unless $x$ is a vertex of $T$. It follows directly from the definition of $F$ that the collection $\mathcal{B} = \{B_{\widetilde{x}}: x \in \partial T\}$ covers $\overline{\Omega}$. We apply the basic covering lemma (see \cite[Thm.\ 1.2, p.~2]{Hei:01}) to find a disjoint subcollection $\mathcal{B}' \subset \mathcal{B}$ such that $\{5B: B \in \mathcal{B}'\}$ also covers ${\Omega}$. It follows that 
\begin{align*}
    \mathcal{H}^2(\Omega) \leq \sum_{B_{\widetilde{x}} \in \mathcal{B}'} 25 \pi r_{x}^2.
\end{align*}
We next pass to a finite subcollection $\mathcal{B}'' \subset \mathcal{B}'$ that covers most of $\Omega$, say
\begin{equation} \label{equ:omega_estimate}  
   \mathcal{H}^2(\Omega) \leq \sum_{B_{\widetilde{x}} \in \mathcal{B}''} 30 \pi r_{x}^2.
\end{equation}
Moreover, we require that $r_x>0$ for each $B_{\widetilde x} \in \mathcal B''$.

On the other hand, we consider the collection 
\[\mathcal{A}'' = \{\overline{B}_{\bar{d}}(x,r_x/4): \overline{B}(\widetilde{x},r_x) \in \mathcal{B}''\}.\]
We claim that any two distinct balls in $\mathcal{A}''$ are disjoint. Otherwise, there are two balls $A_1, A_2 \in \mathcal{A}''$, where $A_i = \overline{B}_{\bar{d}}(x_i,r_{x_i}/4)$, that intersect in a point $y$. Then $\bar{d}(x_1,x_2) \leq \bar{d}(x_1,y) + \bar{d}(y,x_2) \leq (r_{x_1} + r_{x_2})/4$. However, $|\widetilde{x}_1 - \widetilde{x}_2| > r_{x_1} + r_{x_2}$ since the balls $B(\widetilde{x}_1,r_{x_1})$ and $B(\widetilde{x}_2, r_{x_2})$ are disjoint. This contradicts the fact that $F$ is a $4$-bi-Lipschitz embedding with respect to $\bar{d}$. 

We now want to obtain a lower bound for the Hausdorff $2$-measure on $T$. We observe first that, by \Cref{lemm:lsc},  each ball $A_i =\overline{B}_{\bar{d}}(x_i,s_i)$ in $\mathcal{A}''$ is closed with respect to the topology on $X$ given by $d$ and is connected. Since $A_i\cap \partial T$ is connected, we see that $A_i\cap T$ is connected. Moreover, $A_i$ does not intersect the edges of $\partial T$ that do not contain $x_i$; hence, $\partial T\setminus A_i$ is connected. Using Lemma \ref{lemma:separation} (in fact we only need the topological consequences and not the rectifiability) for the function $x\mapsto \dist(A_i\cap T,x)$ on $T$, one can find a Jordan arc $F_i \subset T\setminus A_i$, arbitrarily close to $A_i\cap T$, whose endpoints lie on the edge of $T$ that contains $x_i$ and with the property that the arc $F_i$ together with a subarc of $\partial T$ bound a closed Jordan region $U_i$ that contains $A_i\cap T$. Since $F_i$ can be taken arbitrarily close to $A_i\cap T$, we may have that the regions $U_i$ are mutually disjoint.

Next, we take $E_i=\br B_{\bar{d}}(x_i,s_i/2)\cap \partial T$, which is a subarc of $A_i\cap \partial T$ of length $s_i$ centered at $x_i$. We use \Cref{lemm:EBPC} to give a lower bound on the Hausdorff $2$-measure of $U_i$ as follows. Define the function $u_i \colon U_i \to [s_i/2,s_i]$ by $u_i(y) = \min\{\max\{\bar{d}(x_i ,y),s_i/2\}, s_i\}$. Then $u_i$ is in $L^1(U_i)$, and moreover the constant function $\chi_{U_i}$ is an upper gradient of $u_i$ also in the space $L^1(U_i)$. We conclude that $u_i \in N^{1,1}(U_i)$. Moreover, $E_i \subset u_i^{-1}(s_i/2)$ and $F_i \subset u_i^{-1}(s_i)$. 

Let $v_i$ be the Lipschitz function given by \Cref{lemm:EBPC} with respect to the continua $E_i$ and $F_i$, so that $E_i\subset v_i^{-1}(s_i/2)$ and $F_i\subset v_i^{-1}(s_i)$. By Lemma \ref{lemma:separation}, for a.e.\ $r\in (s_i/2,s_i)$ we can find a Jordan arc $G_r\subset v_i^{-1}(r)\subset U_i$ that separates $E_i$ from $F_i$ and therefore intersects both components of $( A_i\cap \partial T) \setminus E_i$. Since $A_i \cap \partial T$ is a $\bar{d}$-geodesic, it follows that $\mathcal{H}^1(G_r) \geq \mathcal H^1(E_i)\geq s_i$ for a.e.\ $r \in (s_i/2,s_i)$. By \Cref{lemm:EBPC}, we have
\[s_i^2/2 \leq \int_{s_i/2}^{s_i} \mathcal{H}^1(v_i^{-1}(t))\,dt \leq 2 \int_{U_i} 1 \,d\mathcal{H}^2 = 2\mathcal{H}^2(U_i).  \]
Since the sets $U_i$ are disjoint, we have
\[ \sum_{B_{x_i} \in \mathcal{B}''} \frac{r_{x_i}^2}{64} \leq \sum_i \frac{s_i^2}{4} \leq \sum_i \mathcal H^2(U_i)\leq  \mathcal{H}^2(T).\]
This inequality combines with \eqref{equ:omega_estimate} to establish the lemma for $K = 64 \cdot 30\pi$.
\end{proof}

\smallskip

\subsection{Polyhedral fillings of triangular disks}
The following lemma is an adaptation of Theorem 4.2 in \cite{NR:21}. One feature is the interplay of the metrics $d$ and $\bar{d}$.

\begin{lemm} \label{thm:triangular_surface}
Let $(X,d)$ be a metric surface of locally finite Hausdorff $2$-measure, and let $T$ be a triangular disk in $X$ with edges $\alpha_j$, $j\in \{1,2,3\}$. There exists a polyhedral surface $(S,d_S)$ that is a triangular disk with edges $\beta_j$, $j\in \{1,2,3\}$, and a homeomorphism  {$\varphi\colon S \to T$} such that the following hold for an absolute constant $L>0$.
\begin{enumerate}[label=\normalfont(\arabic*)]
    \item $\diam_{d_S}(S)\leq L \diam_{\bar{d}}(\partial T)$.   \label{item:filling_1}
    \item $\mathcal H_{d_S}^2(S)\leq L \mathcal  H_d^2(T)$. \label{item:filling_2}
    \item {$\varphi|_{|\beta_j|}$ maps $|\beta_j|$ isometrically onto $|\alpha_j|$ (with respect to $\bar{d}$) for each $j \in \{1,2,3\}$. In particular,} $\varphi|_{\partial S}$ is length-preserving.
    \label{item:filling_3}
    \item For all $x,y \in \partial S$, $d_S(x,y) \geq \bar{d}(\varphi(x),\varphi(y))\geq d(\varphi(x),\varphi(y))$. \label{item:filling_4}
\end{enumerate}
\end{lemm}

The only essential difference between the statement of this lemma and the statement of Theorem 4.2 in \cite{NR:21} is that in \ref{item:filling_1}, \ref{item:filling_3} and \ref{item:filling_4} we use the metric $\bar d$ here instead of the metric $d$. The construction of the surface $S$ is exactly the same as in \cite{NR:21}, where we use the metric $\bar d$ on $\partial T$ instead of the metric $d$. Namely, $S$ is constructed as follows. We consider the $4$-bi-Lipschitz embedding $F\colon (\partial T, \bar{d})\to \R^2$ given by \cite[Prop.\ 3.2]{NR:21} and define $\Omega\subset \R^2$ to be the closed Jordan region bounded by $F(\partial T)$. Then the surface $S$ is constructed by taking a fine polygonal approximation of the region $\Omega$, replacing each polygon with a certain small polyhedron, and then rescaling appropriately the Euclidean metric.

Properties \ref{item:filling_1}, \ref{item:filling_3} and \ref{item:filling_4} follow immediately from the construction. On the other hand, property \ref{item:filling_2} requires more work since we consider the Hausdorff measure on $T$ with respect to the original metric $d$ rather than the length metric $\bar{d}$. In \cite{NR:21}, we obtain  \ref{item:filling_2} as a consequence of the Besicovitch inequality; see Theorem 2.1 in \cite{NR:21}. Here, the same relationship follows, with a larger constant as a consequence of \Cref{thm:triangle_area}. Indeed, by the construction of $S$, $\mathcal H^2_{d_S}(S)$ is comparable to $\mathcal H^2(\Omega)$. Using \Cref{thm:triangle_area}, applied to the embedding $F$, we obtain $\mathcal H^2(\Omega)\leq 16 K \mathcal H^2_d(T)$.  This implies the inequality in \ref{item:filling_2}.

\bigskip

\section{Triangulations of surfaces and polyhedral approximation} \label{sec:triangulations}

\subsection{Triangulations}

The triangle decomposition theorem in \cite{CR:21} extends to our situation with small modifications. First we introduce the relevant terminology. Let $(X,d)$ be a metric surface of locally finite Hausdorff $2$-measure and consider the extended length metric $\bar d$. Recall that a polygonal disk is a closed Jordan domain with piecewise $\bar{d}$-geodesic boundary. Two polygonal disks are \textit{non-overlapping} if their interiors are disjoint.  If every boundary component of the surface $X$ is a piecewise $\bar{d}$-geodesic curve, we say that $X$ has \textit{polygonal boundary}. A set $A$ is \textit{convex} if every pair of points in the topological boundary of $A$ can be connected by a $\bar{d}$-geodesic contained in $A$. Note that we restrict to connecting boundary points because we cannot expect that all points of $A$ are accessible by rectifiable curves. Observe that any two points of a rectifiable curve that are sufficiently close to each other can be joined by a $\bar{d}$-geodesic. It is evident that, if a polygonal disk $A$ is convex, then both the metric $d$ on $X$ and the restriction of $d$ to $A$ induce the same extended length metric when restricted to $A$.  

The following theorem is a version of the conclusion of Theorem 1.2 in \cite{CR:21}, where the assumption of a length metric is replaced with the assumption that the surface has locally finite Hausdorff $2$-measure.

\begin{thm}
\label{thm:triangulation}
Let $(X,d)$ be a metric surface of locally finite Hausdorff $2$-measure such that every boundary component of $X$ is a piecewise $\bar{d}$-geodesic curve, and let $\varepsilon >0$. Then $X$ may be covered by a locally finite collection of non-overlapping triangular disks $\{T_i\}_{i \in \mathcal{I}}$ such that each disk $T_i$ is convex and has diameter and perimeter at most $\varepsilon$.
\end{thm}

The basic idea is that the Hausdorff $2$-measure assumption guarantees enough rectifiable curves for any point to be enclosed by a polygonal curve of arbitrarily small diameter and perimeter. Then, all the operations in the proof of the version of \Cref{thm:triangulation} in \cite{CR:21} necessarily avoid the non-rectifiably connected points. Note that \Cref{thm:triangulation} does not necessarily yield a triangulation in the usual topological sense, since we do not require adjacent triangular disks to intersect along entire edges or at only a vertex. We remark that convexity does not play a direct role in the proof of \Cref{thm:extended_polyhedral_approximation}, although convexity is used in obtaining \Cref{thm:triangulation} itself since it enables one to decompose an arbitrary polygon into triangles. 

The proof proceeds similarly to the proof in \cite{CR:21} but with a pair of modifications in the beginning stages. First, we need to modify the proof  \cite[Lemma 5.1]{CR:21} on the existence of polygonal neighborhoods at every point, which in turn is based on \cite[Lemma III.3]{AZ:67}. We state this lemma in slightly simpler form.

\begin{lemm} \label{lemm:polygonal_nbhd}
Suppose that $X$ is as in Theorem \ref{thm:triangulation} and let $x \in X$ and $\varepsilon>0$. Then there is a polygonal disk $P$ with $x\in \Int (P)$ such that $\diam (P) < \varepsilon$ and $\mathcal{H}^2(P) < \varepsilon$. 
\end{lemm}

The proof is the same as that of \cite[Lemma 5.1]{CR:21} except for one step, where we use a substitute argument based on Lemma \ref{lemma:separation}.

\begin{proof}
Let $U$ be a neighborhood of $x$ homeomorphic to a closed disk such that $\diam U < \varepsilon$  and $\mathcal{H}^2(U) < \varepsilon$. Note that if $x \in \partial X$, then $x$ belongs to the manifold boundary of $U$. Assume moreover that $U$ does not contain any vertex of $\partial X$ except for possibly $x$ itself. 

By Lemma \ref{lemma:separation}, there exists $r>0$ with $S(x,r)\subset U$ and a rectifiable Jordan curve or Jordan arc $E\subset S(x,r)$ that separates $x$ from $X\setminus U$. Moreover, there exists a closed Jordan region $W\subset U$ containing a neighborhood of $x$ such that $\partial W$ coincides with $E$ (if $x\in \Int(X)$) or is the union of $E$ with an arc of $\partial X$ (if $x\in \partial X$). 

Recall that any two points of $\partial W$ that are sufficiently close to each other can be joined by a $\bar{d}$-geodesic. Let $\{w_i\}_{i=1}^m$ be a finite collection of points in $\partial W$ in cyclic order such that the subarc of $\partial W$ between $w_i$ and $w_{i+1}$ has length at most $\delta$ for some $\delta>0$ sufficiently small.  If $x \in \partial X$, then we take $w_1=x$. Connect $w_i$ to $w_{i+1}$ by a $\bar{d}$-geodesic $\gamma_i$, where we require $\gamma_i$ to lie in $\partial X$ if $w_i$ and $w_{i+1}$ do. Take $\gamma$  to be the concatenation $\gamma_1 * \gamma_2 * \cdots * \gamma_m$. The final step is to extract a subcontinuum of $\gamma$ bounding a polygonal neighborhood $P$ of $x$. For this step, we refer back to proof for Lemma 5.1 in \cite{CR:21}.  Since $P$ is contained in $U$, it follows that $\mathcal{H}^2(P) < \varepsilon$.
\end{proof}

The next step in the proof concerns finding a neighborhood of an arbitrary point $x$ consisting of finitely many polygons of small perimeter. This step is carried out as Lemma 5.2 in \cite{CR:21}. Its proof is based on connecting a fixed vertex of the polygon $P$ given by \Cref{lemm:polygonal_nbhd} to the other vertices of $P$ by geodesics. However, this argument does not work in our setting, since $\bar{d}$-geodesics between these vertices may fail to exist, and, even if they do exist, they might not have small length. Instead, we can give a different proof applying to our setting based on the reciprocal lower bound on modulus of quadrilaterals in \eqref{ireciprocality:12} as proved in \cite{RR:19} and \cite{EBC:21}.

\begin{lemm} \label{lemm:small_perimeter}
Suppose that $X$ is as in Theorem \ref{thm:triangulation} and let $x \in X$ and $\varepsilon>0$. Then there is a closed neighborhood $P$ of $x$ that is a polygonal disk such that $\diam(P) < \varepsilon$ and $P$ is the union of finitely many non-overlapping polygonal disks $P_1, \ldots, P_n$, where each $P_i$ has perimeter at most $\varepsilon$.
\end{lemm} 
\begin{proof}
Let $\varepsilon' = \pi \varepsilon^2/256$ and consider the polygonal disk $P$ from the previous lemma for such $\varepsilon'$ so that $\diam(P) <\varepsilon'$ and $\mathcal{H}^2(P) < \varepsilon'$. We subdivide $P$ inductively according to the following procedure. 

Divide the boundary $\partial P$ into four non-overlapping arcs of equal length $\mathcal H^1(\partial P)/4$, thus making $P$ a topological quadrilateral. If $\mathcal H^1(\partial P) \leq \varepsilon$, then $P$ satisfies the conclusions of the lemma. Otherwise, we let $\Gamma (P), \Gamma^*(P)$ denote the two families of curves connecting opposite sides of the quadrilateral $P$. Let $M = \inf \ell(\gamma)$, the infimum taken over all curves in $\Gamma(P) \cup \Gamma^*(P)$. Observe that the constant function $\rho = 1/M$ is admissible for both $\Gamma(P)$ and $\Gamma^*(P)$. Thus
\[\Mod \Gamma(P) \cdot \Mod \Gamma^*(P) \leq \left( \int_P \frac{1}{M^2}\,d\mathcal{H}^2 \right)^2 = \frac{(\mathcal{H}^2(P))^2}{M^4} \leq \frac{(\varepsilon')^2}{M^4}.  \] 
Applying the left-hand inequality in \eqref{ireciprocality:12} with $\kappa^{-1} = \pi^2/16$ gives the inequality $M^2 \leq 4\varepsilon'/\pi$. 

Since $P$ is compact, as are its sides, there is a curve $\gamma \in \Gamma_1 \cup \Gamma_2$ having length $M$. In particular, $\gamma$ has length at most $\sqrt{4\varepsilon'/\pi} = \varepsilon/8$. Moreover, it follows from the minimality of length that $\gamma$ is injective and piecewise $\bar{d}$-geodesic. In addition, applying Lemma 4.1 in \cite{CR:21} to remove any ``superfluous intersections'' of $\gamma$ with $\partial P$, we may assume that $\gamma$ intersects each edge of the polygon $P$ in a connected set. We see from this that $\gamma$ splits $P$ into finitely many non-overlapping polygonal disks $P_{1,i}$, $i\in I_1$. The boundary of each disk $P_{1,i}$ is a subset of $|\gamma|$ and at most three of the sides of the quadrilateral $P$; this follows from the minimality of $\ell(\gamma)$. In particular, $\mathcal H^1(\partial P_{1,i}) \leq (3/4)\mathcal H^1(\partial P) + \varepsilon/8$. Given that $\mathcal H^1(\partial P) > \varepsilon$, we have
$$\mathcal H^1(\partial P_{1,i}) \leq (7/8) \mathcal H^1(\partial P).$$   

We now repeat this step to each $P_{1,i}$, $i\in I_1$, that satisfies $\mathcal H^1(\partial P_{1,i})>\varepsilon$. As a result, we obtain a decomposition of such $P_{1,i}$ into finitely many non-overlapping polygonal disks $P_{2,j}$, $j\in I_{2,i}$ with perimeter at most $$(7/8)\mathcal H^1(\partial P_{1,i}) \leq (7/8)^2\mathcal H^1(\partial P).$$
It is evident that after finitely many subdivisions we will have the desired decomposition of the original polygonal disk $P$. 
\end{proof} 

For convenience of the reader, we sketch the remainder of the proof of \Cref{thm:triangulation}.  For each point $x \in X$, we consider a fixed neighborhood $U$ of $x$ homeomorphic to a closed disk. We apply \Cref{lemm:small_perimeter} at $x$, taking $\varepsilon>0$ sufficiently small based on the neighborhood $U$, to obtain a collection of polygonal disks $P_1, \ldots, P_n$ each of perimeter at most $\varepsilon$ whose union is a polygonal neighborhood $P$ of $x$. As the next step, for each polygon $P_i$, we consider the  set of polygonal disks in $U$ of minimal perimeter containing $P_i$. Among such disks, we choose the unique one which is maximal with respect to set inclusion, denoted by $Q_i$. By taking $\varepsilon>0$ sufficiently small, we ensure that $Q_i$ is well-separated from the boundary of $U$ and consequently is convex, in fact satisfying a very strong form of convexity called \textit{absolute convexity}; see Section 3 of \cite{CR:21} for the definition. In this way we obtain a covering of the surface by potentially overlapping absolutely convex polygons. The next step is a subdivision procedure to obtain a covering by non-overlapping convex polygons. This procedure is delicate and relies heavily on the absolute convexity of the polygonal disks. The details can be found in Lemmas 5.6 and 5.7 of \cite{CR:21}. The final step is to split each polygon into triangles by connecting a fixed base vertex to the other vertices by $\bar{d}$-geodesics. 

We next state a variation on Lemma 5.3 in \cite{NR:21}. 

\begin{lemm} \label{lemm:polygonal_boundary}
Let $X$ be a surface of locally finite Hausdorff $2$-measure and $\varepsilon>0$.  There is an $\varepsilon$-isometric retraction from $X$ onto a subset $\widetilde{X}$ that is homeomorphic to $X$ and has polygonal boundary. Moreover, the inclusion from $\widetilde X$ into $X$ is proper.
\end{lemm}
Recall that a retraction of a topological space $Y$ onto a subset $Z$ is a continuous map from $Y$ to $Z$ whose restriction to $Z$ is the identity map. Note that Lemma 5.3 in \cite{NR:21} includes the statement that $\widetilde{X}$ is a convex subset of $X$. This guarantees that the metric on $X$ restricts to a length metric on $\widetilde{X}$. However, our approach in this paper no longer requires this feature, which simplifies the current proof compared to the proof of Lemma 5.3 in \cite{NR:21}. The statement that $X$ retracts onto $\widetilde{X}$ plays a similar role in this paper.

We remark that the proof of Lemma 5.3 in \cite{NR:21} constructs only an $\varepsilon$-isometry and not a retraction. However, the same argument as in this lemma can be inserted into the proof in \cite{NR:21} to find an $\varepsilon$-isometric retraction from $X$ onto the subspace $\widetilde{X}$ constructed in that lemma. In particular, we can guarantee that condition \ref{item:main_3} holds in the length space case covered in Theorem 1.1 in \cite{NR:21}.

\begin{proof}
Let $Y$ be a component of the boundary $\partial X$. If $Y$ is homeomorphic to $\mathbb{R}$ (resp.\ $\mathbb S^1$), we use the tubular neighborhood theorem to find a closed neighborhood $U_Y$ of $Y$ homeomorphic to the strip $\R\times [0,1]$ (resp.\ $\mathbb S^1\times [0,1]$), with $\varphi$ the corresponding homeomorphism, mapping $Y$ onto $\R\times \{0\}$ (resp.\ $\mathbb S^1\times \{0\}$). By restricting to smaller neighborhoods if needed, we may assume that $U_Y$ is contained in the $\varepsilon$-neighborhood of $Y$, and that $U_{Y_1}$ and $U_{Y_2}$ are disjoint for any distinct components $Y_1, Y_2$ of $\partial X$. In particular, the topological boundary of each $U_Y$ as a subset of $X$ is contained in the (manifold) interior of $X$. 

We fix a boundary component $Y$ homeomorphic to $\R$ (resp.\ $\mathbb S^1$). Our goal is to find a closed set $W_Y\subset U_Y$ whose topological boundary in $X$ consists of the topological boundary of $U_Y$ and a piecewise $\bar d$-geodesic so that the image of $W_Y$ under $\varphi$ is a closed strip bounded by $\R\times \{1\}$ and a proper embedding of $\R$ separating the two boundary lines  (resp.\ a closed annulus bounded by $\mathbb S^1\times \{1\}$ and by a Jordan curve contained in $\mathbb S^1\times (0,1)$ homotopic to $\mathbb S^1\times \{0\}$); in particular, $U_Y$ is homeomorphic to $W_Y$. We then find an $\varepsilon$-isometric retraction from $U_Y$ onto $W_Y$. Replacing each set $U_Y$ in $X$ with the corresponding $W_Y$ for each boundary component $Y$ gives a set $\widetilde X\subset X$ homeomorphic to $X$. Pasting together the retractions arising from different boundary components gives the desired $\varepsilon$-isometric retraction from $X$ onto $\widetilde X$. Finally, if $A$ is a compact subset of $X$, then 
$$A\cap \widetilde X = A \setminus \bigcup_Y (U_Y\setminus W_Y).$$
Since $U_Y\setminus W_Y$ is an open subset of $X$ for each boundary component $Y$, we see that $A\cap \widetilde X$ is compact and therefore the inclusion from $\widetilde X$ into $X$ is proper.

We only present the construction of $W_Y$ in the case that $Y$ is homeomorphic to $\R$, since the other case is similar and simpler. By shrinking the neighborhood $U_Y$ and modifying the homeomorphism $\varphi$, we may assume that the diameter of $R_i=\varphi^{-1}([i,i+1]\times [0,1])$ is less than $\varepsilon/4$ for each $i\in \Z$. We also set $V_Y=\varphi^{-1}(\R\times (0,1))$.

Our first goal is to find a locally rectifiable simple curve $E_Y\subset V_Y$ near $Y$ that is the concatenation of countably many rectifiable Jordan arcs $\eta_i$, each contained in the interior of the topological closed disk $R_{i-1}\cup R_{i}$, $i\in \Z$. In particular, $E_Y$ admits a parametrization $\alpha\colon \R\to E_Y$ such that if $\alpha(t)\in R_i$ for some $t\in \R$ and $i\in \Z$ then $\alpha(s)$ is disjoint from  $\bigcup_{j<i-1} R_j$ for $s\geq t$; that is $E_Y$ does not go ``back and forth" for a very long distance. Now we proceed with the construction. Let $A_i= \varphi^{-1}( [i-1/2,i+1/2]\times \{1/2\})$. The function $f_i(x)= d(x,A_i)$ in $R_{i-1}\cup R_i$ is Lipschitz, so its level sets are rectifiable by the co-area inequality of Lemma \ref{lemm:coarea}. By Lemma \ref{lemma:separation}, arbitrarily close to $A_i$ there exists a rectifiable Jordan curve contained in a level set of $f_i$ that separates $A_i$ from the boundary of $R_{i-1}\cup R_i$.  Note that Jordan curves corresponding to consecutive sets $A_i$, $i\in \Z$, intersect. Thus, we may find non-overlapping subarcs $\eta_i$, $i\in \Z$, of these Jordan curves, whose concatenation gives the curve $E_Y$. 

Next, we show how to find a piecewise $\bar d$-geodesic with the same properties as $E_Y$. Fix a path $\eta_i$ in $E_Y$ and consider a finite collection of ordered points $\{w_1,\dots,w_k\}$ that partition $\eta_i$. Since $\eta_i$ lies in the interior of $R_{i-1}\cup R_i$, if the partition is fine enough, then we may connect each point $w_j$ to $w_{j+1}$ with a $\bar d$-geodesic lying in the interior of $R_{i-1}\cup R_i$. We concatenate these geodesics to obtain a curve $\gamma_i$ in $R_{i-1}\cup R_i$ with the same endpoints as $\eta_i$. Note that $\gamma_i$ might not be a simple curve. However, we can replace $\gamma_i$ with a simple subpath with the same endpoints that is a piecewise $\bar d$-geodesic and is contained in the interior of $R_{i-1}\cup R_i$. Moreover, by passing to further subpaths we can ensure that the concatenation of the paths $\gamma_i$, $i\in \Z$, gives a simple piecewise $\bar d$-geodesic $E_Y'$, which we parametrize by $ \alpha' \colon \R \to \bigcup_{i\in \Z} |\gamma_i|$. Then $\alpha'$ has the desired property that if $\alpha'(t)\in R_i$ for some $t\in \R$ and $i\in \Z$ then $\alpha'(s)$ is disjoint from $\bigcup_{j< i-1} R_j$ for $s\geq t$.

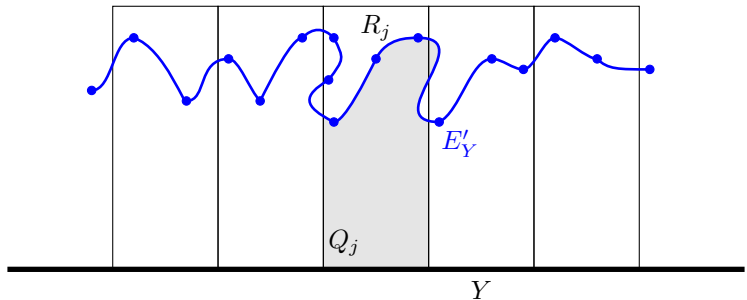
\begin{figure} 
    \centering
    \begin{tikzpicture}[scale=1.4]
        \fill[gray,opacity=.2] (0,-.5) to (.0,.95) to (.1,.9) .. controls (.2,.9) and (.4,1.3) .. (.5,1.5) .. controls (.6,1.7) and (.8,1.7) .. (.9,1.7) .. controls (1.5,1.7) and (.6,.975) .. (1.0, .92) to (1,-.5);
        \draw[line width = 2pt] (-3,-.5) to (4,-.5);
        \draw[] (-2,-.5) to (-1,-.5) to (-1, 2.0) to (-2,2.0) to (-2,-.5);
        \draw[] (-1,-.5) to (0,-.5) to (0, 2.0) to (-1,2.0) to (-1,-.5);
        \draw[] (0,-.5) to (1,-.5) to (1, 2.0) to (0,2.0) to (0,-.5);
        \draw[] (1,-.5) to (2,-.5) to (2, 2.0) to (1,2.0) to (1,-.5);
        \draw[] (2,-.5) to (3,-.5) to (3, 2.0) to (2,2.0) to (2,-.5);
        \draw[blue, line width = 1pt] (-2.2,1.2) .. controls (-2.0,1.2) and (-2.0,1.7) .. (-1.8,1.7) .. controls (-1.5,1.6) and (-1.3,1.0) .. (-1.3,1.1) .. controls (-1.1,1.1) and (-1.2,1.5) .. (-.9,1.5) .. controls (-.8,1.5) and (-.7,1.2) .. (-.6,1.1) .. controls (-.5,1.3) and (-.3,1.7) .. (-.2,1.7) .. controls (-.1,1.8) and (0,1.8) .. (.1, 1.7) .. controls (.2,1.5) and (.2,1.5) .. (.05,1.3) .. controls (-.2,1.15) and (-.2,1.05) .. (.1,.9) .. controls (.2,.9) and (.4,1.3) .. (.5,1.5) .. controls (.6,1.7) and (.8,1.7) .. (.9,1.7) .. controls (1.5,1.7) and (.5,.9) .. (1.1, .9) .. controls (1.3,1.0) and (1.4,1.5) .. (1.6, 1.5) .. controls (1.7,1.5) and (1.8,1.4) .. (1.9, 1.4) .. controls (2.0,1.4) and (2.1,1.7) .. (2.2, 1.7) .. controls (2.3, 1.7) and (2.5,1.5) .. (2.6, 1.5) .. controls (2.7,1.4) and (2.8,1.4) .. (3.1, 1.4); 
        \filldraw[blue] (-2.2,1.2) circle (.04);
        \filldraw[blue] (-1.8,1.7) circle (.04);
        \filldraw[blue] (-1.3,1.1) circle (.04);
        \filldraw[blue] (-.9,1.5) circle (.04);
        \filldraw[blue] (-.6,1.1) circle (.04);
        \filldraw[blue] (-.2,1.7) circle (.04);
        \filldraw[blue] (.1, 1.7) circle (.04);
        \filldraw[blue] (.05,1.3) circle (.04);
        \filldraw[blue] (.1,.9) circle (.04);
        \filldraw[blue] (.5,1.5) circle (.04);
        \filldraw[blue] (.9,1.7) circle (.04);
        \filldraw[blue] (1.1, .9) circle (.04);
        \filldraw[blue] (1.6, 1.5) circle (.04);
        \filldraw[blue] (1.9, 1.4) circle (.04);
        \filldraw[blue] (2.2, 1.7) circle (.04);
        \filldraw[blue] (2.6, 1.5) circle (.04);
        \filldraw[blue] (3.1, 1.4) circle (.04);
        \node[] at (1.5,-.7) {$Y$};
        \node[] at (.5,1.8) {$R_j$};
        \node[] at (.2,-.25) {$Q_j$};
        \node[blue] at (1.3,.7) {$E_Y'$};
    \end{tikzpicture}
    \caption{The topological quadrilaterals $Q_j$ retract onto the curve $E_Y'$.}
    \label{fig:retraction}
\end{figure}

Let $S_j$ the the subarc of the left edge of $R_j$ from $Y$ to the first point of intersection with $E_Y'$. The sets $Y$, $S_j$, $E_Y'$ and $S_{j+1}$ enclose a topological quadrilateral $Q_j$. By the properties of the parametrization $\alpha'$, we observe that $Q_j$ is contained in $R_{j-1} \cup R_j \cup R_{j+1}$. See \Cref{fig:retraction}. 

There is a retraction from $Q_j$ onto $Q_j \cap E_Y'$ obtained by foliating $Q_j$ by arcs connecting $Y$ to $E_Y'$ (through the map $\varphi$) and mapping each arc to its endpoint in $E_Y'$. Since $Q_j\subset R_{j-1}\cup R_j\cup R_{j+1}$ and the latter set has diameter bounded by $3\varepsilon/4$, the retraction is $\varepsilon$-isometric. 
To conclude the proof, let $W_Y$ denote the closed strip bounded by $E_Y'$ and $\varphi^{-1}(\R\times \{1\})$.  Finally, pasting the retractions of $Q_j$ onto $Q_j \cap E_Y'$, $j\in \Z$, together with the identity map on $W_Y$ gives the desired retraction from $U_Y$ onto $W_Y$.
\end{proof}

\smallskip

\subsection{Polyhedral approximation} \label{sec:polyhedral_approximation}

In this section, we prove \Cref{thm:extended_polyhedral_approximation} on polyhedral approximation of metric surfaces. The idea is to take a sufficiently fine triangulation of the original surface $X$ and replace each triangular disk with a suitable polyhedral surface, thus defining the approximating surfaces $X_n$. This follows the corresponding proof in \cite{NR:21} with one main exception: to prove \Cref{thm:extended_polyhedral_approximation} for non-length spaces, the polyhedral metric on $X_n$ is not suitable and must be modified on large scales to match the metric on $X$. This is accomplished by abstractly gluing additional line segments to $X_n$ connecting the vertices; the length of these segments is the same as the distance between the image in $X$ of the endpoints. We equip $X_n$ with the restriction of the length metric on this enlarged space. This construction can be illustrated by taking a flat piece of paper and forcing it to bend by attaching short strings to certain pairs of points.

\begin{proof}[Proof of \Cref{thm:extended_polyhedral_approximation}]

Let $(X,d)$ be a metric surface with locally finite Hausdorff $2$-measure. Choose a sequence $\{\varepsilon_n\}_{n=1}^\infty$ of positive reals satisfying $\varepsilon_n \to 0$ as $n \to \infty$. 

We apply \Cref{lemm:polygonal_boundary} to find a surface $\widetilde{X}_n \subset X$ that is homeomorphic to $X$ and has polygonal boundary and an $\varepsilon_n$-isometric retraction from $X$ onto $\widetilde X_n$. Note that the inclusion map from $\widetilde X_n$ into $X$ is an $\varepsilon_n$-isometry. The surface $\widetilde{X}_n$ is equipped with the restriction of the metric $d$ from $X$, which we continue to denote by $d$.  

Since the space $\widetilde{X}_n$ has polygonal boundary, we can apply \Cref{thm:triangulation} with the parameter $\varepsilon_n$ to obtain a decomposition $\widetilde{\mathcal{T}}_n$ of $\widetilde{X}_n$ into triangular disks with diameter and perimeter less than $\varepsilon_n$. We consider the edge graph $\widetilde{\mathcal{E}}_n = \mathcal{E}(\widetilde{\mathcal{T}}_n)$ as having the length metric $\widetilde{d}_n$ induced by the restriction of $d$ to $\widetilde{\mathcal{E}}_n$. Let $\widetilde{\mathcal{V}}_n$ be the corresponding vertex set. 

For each triangular disk $T \in \widetilde{\mathcal{T}}_n$, consider the polyhedral surface $S$ and the corresponding homeomorphism $\varphi_T\colon S \to T$ given by \Cref{thm:triangular_surface}.  By condition \ref{item:filling_3}, $\varphi_T|_{\partial S}$ is length-preserving as a map from $\partial S$ into $\widetilde{\mathcal{E}}_n$. 
We define a locally compact length metric space $X_n'$ as follows. First, we glue each disk $S$ into $\widetilde{\mathcal{E}}_n$ along the map $\varphi_T$ and obtain a polyhedral length surface. Next, for all pairs of vertices $x,y \in \widetilde{\mathcal V}_n$ we glue in a line segment $I_{xy}$ of length $d(x,y)$ connecting $x$ and $y$. Denote the resulting length metric on $X_n'$ by $d_{n}$. More formally, let $Z$ be the disjoint union of the polyhedral surfaces $S$ with the segments $I_{xy}$ and let $\rho$ be the induced metric of $Z$. The metric space $X_n'$ is obtained by taking the quotient of $Z$ with the described identifications and the metric $d_n$ is defined by
$$d_n(x,y)=\inf \left\{ \sum_{i=1}^k \rho (p_i,q_i): p_1=x, \,\, q_k=y, \,\, k\in \N \right\},$$
where the infimum is taken over all choices of $\{p_i\}_{i=1}^k$ and $\{q_i\}_{i=1}^k$ in $Z$ such that $q_i\sim p_{i+1}$ for $i\in \{1,\dots,k-1\}$. 

We now define the surface $X_n\subset X_n'$ by removing the interiors of the glued line segments $I_{xy}$. We continue to denote by $d_n$ the restriction of the length metric on $X_n'$ to $X_n$. Define the map $\Phi_n\colon X_n \to \widetilde{X}_n$ by gluing the individual maps $\varphi_T$. It is immediate that $\Phi_n$ is a homeomorphism. We set $\mathcal V_n =\Phi_n^{-1} (\widetilde{\mathcal V}_n)$.  We observe that $d_n$ is a locally geodesic metric on $X_n$ that is locally isometric to the polyhedral length metric on $X_n$. To see this, observe that if $x,y \in X_n$ satisfy $d_n(x,y) < \dist_{d_n}(x,{\mathcal V}_n\setminus \{x\})$, then $x$ and $y$ can be connected with a geodesic that avoids the interiors of the glued segments in $X_n'$ and thus is contained in $X_n$.

We claim that 
\begin{align}\label{approx:vertices}
    d_n(x,y)=d(\Phi_n(x),\Phi_n(y)) \quad \textrm{whenever $x,y\in {\mathcal V}_n$.}
\end{align}
If $x,y\in {\mathcal V}_n$, it is immediate that $d_n(x,y)\leq d(\Phi_n(x),\Phi_n(y))$ by the definition of the metric $d_n$. Conversely, we note that for any chain connecting $x$ and $y$ as in the definition $d_n$, if $p_i,q_i\in \partial S$ for some surface $S$ corresponding to a triangle $T$, then $\rho(p_i,q_i)=d_S(p_i,q_i)$, which is at least $d( \varphi_T(p_i),\varphi_T(q_i))$ by Lemma \ref{thm:triangular_surface} \ref{item:filling_4}, while if  $p_i,q_i$ are the endpoints of a glued segment $I_{p_iq_i}$, then $\rho(p_i,q_i)=d(p_i,q_i)$.  Therefore, the triangle inequality gives $d_n(x,y)\geq d(\Phi_n(x),\Phi_n(y))$.

Define $f_n \colon X_n \to X$ to the composite of $\Phi_n$ and the inclusion map from $\widetilde{X}_n$ into $X$. By Lemma \ref{lemm:polygonal_boundary} the latter inclusion is proper. Hence $f_n$ is a proper topological embedding. The set $f_n(X_n)=\widetilde X_n$ is $\varepsilon_n$-dense in $X$, since the inclusion map from $\widetilde X_n$ into $X$ is an $\varepsilon_n$-isometry. Next, for all $x,y \in X_n$, we can find $x',y' \in \mathcal{V}_n$ belonging to the same triangle as $x$ and $y$, respectively, satisfying $d_n(x',x) \leq L\varepsilon_n$ and $d_n(y',y) \leq L \varepsilon_n$; here $L>0$ is a uniform constant as in Lemma \ref{thm:triangular_surface} \ref{item:filling_1}. Moreover, $d(f_n(x),f_n(x')) \leq \varepsilon_n$ and $d(f_n(y),f_n(y')) \leq \varepsilon_n$, since the pairs $f_n(x),f_n(x')$ and $f_n(y),f_n(y')$ belong to the same triangles. By \eqref{approx:vertices}, we have $d_n(x',y')= d(f_n(x'),f_n(y'))$.
Combining these estimates, we see that
\[|d(f_n(x),f_n(y)) - d_n(x,y)| \leq (2L+4)\varepsilon_n. \] 
This verifies the claim that $f_n$, $n\in\N$, is an approximately isometric sequence, as required in \ref{item:main_1}. The $\varepsilon_n$-isometric retraction from $X$ onto $\widetilde X_n=f_n(X_n)$ given by Lemma \ref{lemm:polygonal_boundary} already verifies \ref{item:main_3}.

We now prove property \ref{item:main_2} regarding the Hausdorff 2-measure, which proceeds identically to the argument in \cite{NR:21}. Let $A \subset X$ be a compact set and fix $\delta>0$. Choose $n\in \N$ sufficiently large so that $\diam(T)<\delta$ for every triangular disk $T\in \widetilde{\mathcal T}_n$. Then $T$ is contained in the $\delta$-neighborhood $N_{\delta}(A)$ of $A$ whenever $T\cap A\neq \emptyset$. The set $f_n^{-1}(A)$, which is possibly empty, is covered by the sets $f_n^{-1}(T)$ for which $T\cap A\neq \emptyset$. Moreover, from \Cref{thm:triangular_surface} \ref{item:filling_2} it follows that $\mathcal H^2(f_n^{-1}(T)) \leq L \mathcal H^2( T)$ for each $T\in \widetilde{\mathcal T}_n$. Since the boundary of each triangle $T$ has Hausdorff $2$-measure zero, we have $$\mathcal H^2(f_n^{-1}(A)) \leq L\mathcal H^2(N_\delta(A)).$$
Hence, letting $n\to\infty$ and then $\delta\to 0$ gives
$$\limsup_{n\to\infty}\mathcal  H^2(f_n^{-1}(A)) \leq L\mathcal H^2(A).$$
This completes the proof.
\end{proof}

\bigskip

\section{Uniformization of compact surfaces}\label{sec:uniformization}

In this section we establish Theorem \ref{thm:one-sided_qc} for compact surfaces homeomorphic to $\widehat \C$ or $\br \D$ as a consequence of Theorem \ref{thm:extended_polyhedral_approximation}. We follow essentially the same argument as in \cite{NR:21} for deriving the uniformization theorem from the polyhedral approximation theorem; we only have to check that all the steps remain valid. The proof of Theorem \ref{thm:one-sided_qc} in the general case is discussed in Section \ref{sec:global} and follows from Theorem \ref{thm:uniformization_global}. In addition, the optimization of the quasiconformal dilatation to the minimal value $4/\pi$ is discussed in Section \ref{sec:global}. We start with the necessary preliminaries.

\smallskip

\subsection{Gromov--Hausdorff convergence}

A sequence of metric spaces $\{X_n\}_{n=1}^\infty$ is \textit{asymptotically uniformly locally path connected} if for each $\varepsilon>0$ there exists $\delta>0$ and $N\in \N$ such that for each $n\geq N$, every two points of $X_n$ at distance less than $\delta$ can be connected by a curve of diameter less than $\varepsilon$. Equivalently, for each positive sequence $\delta_n\to 0$ there exists a sequence $\varepsilon_n\to 0$ and $N\in \N$ such that for each $n\geq N$, every two points of $X_k$, $k\geq n$, at distance less than $\delta_n$ can be connected with a curve of diameter at most $\varepsilon_n$. Namely,  $\varepsilon_n=\sup_{x,y}  \inf_{\gamma} \diam(|\gamma|)$, where the supremum is taken over all pairs of points $x,y\in X_k$, $k\geq n$, at distance less than $\delta_n$ and the infimum is over all paths $\gamma$ joining $x$ and $y$---this set of paths is non-empty for large $n\in \N$.

The approximation theorem, Theorem \ref{thm:extended_polyhedral_approximation}, gives naturally sequences of asymptotically uniformly path connected spaces, as the next lemma shows. 
\begin{lemm}\label{lemma:aulpc}
Let $X$ be a compact locally connected metric space and $\{X_n\}_{n=1}^\infty$ be a sequence of compact metric spaces.  Suppose that there exists an approximately isometric sequence $f_n\colon X_n\to X$, $n\in \N$, of topological embeddings and an approximately isometric sequence $R_n\colon X\to f_n(X_n)$, $n\in \N$, of retractions. Then the sequence $\{X_n\}_{n=1}^\infty$ is asymptotically uniformly locally path connected.  
\end{lemm}
\begin{proof}
Since $X$ is compact and locally connected, by \cite[Thm.\ 31.4]{Wil:70}, for each $\varepsilon>0$ there exists $\delta>0$ so that if $x,y\in X$ and $d(x,y) < \delta$ then there exists a path in $X$ connecting $x$ and $y$ with diameter less than $\varepsilon/3$. Choose $N \in \mathbb{N}$ large enough so that for all $n \geq N$, $f_n$ is $\eta$-isometric from $X_n$ to $X$ with $\eta=\min \{\varepsilon/3,\delta/2\}$, and the retraction $R_n$ from $X$ onto $f_n(X_n)$ is $(\varepsilon/3)$-isometric. Now let $x,y \in X_n$ with $d_n(x,y)<\delta/2$ for some $n \geq N$. Then $d(f_n(x),f_n(y))<\delta/2+\eta \leq \delta$ so there is a curve in $X$ of diameter at most $\varepsilon/3$ connecting $f_n(x)$ to $f_n(y)$. This curve retracts onto a curve in $f_n(X_n)$ of diameter at most $2\varepsilon/3$ connecting $f_n(x)$ to $f_n(y)$. Then the preimage of the retracted curve under $f_n$ is a curve in $X_n$ of diameter at most $\varepsilon$ connecting $x$ to $y$.
\end{proof}

Next, we establish a path-lifting property, whose analogue for length spaces has been established in \cite[Prop.\ 2.2 (ii)]{NR:21}.

\begin{prop}\label{prop:asympt}
Let $\{X_n\}_{n=1}^\infty$ be a sequence of asymptotically uniformly locally path connected metric spaces converging in the Gromov--Hausdorff sense to a metric space $X$, and consider an approximately isometric sequence $f_n\colon X_n\to X$, $n\in \N$. 

\vspace{1em}
\noindent
Then for each path $\gamma\colon[0,1] \to X$ and for each sequences of points $a_n,b_n\in X_n$ with $f_n(a_n)\to \gamma(0)$ and $f_n(b_n)\to \gamma(1)$ as $n\to\infty$ there exists $N\in \N$ and a sequence of paths $\gamma_n\colon [0,1]\to X_n$, $n\geq N$, such that $\gamma_n(0)=a_n$, $\gamma_n(1)=b_n$, and $f_n\circ \gamma_n$ converges uniformly to $\gamma$ as $n\to\infty$. 
\end{prop}
\begin{proof}
Suppose that $f_n$ is an $\varepsilon_n$-isometry, where $\varepsilon_n> d_X( f_n(a_n),\gamma(0))$,  $\varepsilon_n>d_X(f_n(b_n),\gamma(1))$ for each $n\in \N$, and $\varepsilon_n\to 0$. By the uniform continuity of $\gamma$, for each $n\in \N$ there exists $\delta_n>0$ such that if $|p-q|<\delta_n$, then $d_X(\gamma(p),\gamma(q))<\varepsilon_n$. We pick a finite set $Q_n\subset [0,1]$ that contains $0$ and $1$ so that each of the complementary intervals of $Q_n$ has length less than $\delta_n$. We define $\gamma_n(0)=a_n$ and $\gamma_n(1)=b_n$. By the definition of an $\varepsilon_n$-isometry, for each $q\in Q_n\setminus \{0,1\}$ there exists a point $\gamma_n(q)\in X_n$ such that $d_X(f_n(\gamma_n(q)), \gamma(q)) <\varepsilon_n$. This defines a map $\gamma_n\colon Q_n \to X_n$. If $(q_1,q_2)$ is a complementary interval of $Q_n$, note that 
\begin{align*}
    d_{X_n}( \gamma_n(q_1),\gamma_n(q_2))&\leq \varepsilon_n + d_X(f_n(\gamma_n(q_1)), f_n(\gamma_n(q_2))\\
    &\leq \varepsilon_n+ d_X(f_n(\gamma_n(q_1)) ,\gamma(q_1))+ d_X(\gamma(q_1),\gamma(q_2)) \\
    &\qquad\qquad+ d_X(f_n(\gamma_n(q_2)) ,\gamma(q_2))   \\
    &< 4\varepsilon_n.
\end{align*}
By assumption, there exists a sequence $\varepsilon_n'\to 0$ and $N\in \N$ such that for $n\geq N$  the points $\gamma_n(q_1)$ and $\gamma_n(q_2)$ can be connected by a path of diameter less than $\varepsilon_n'$.  For $n\geq N$ we define $\gamma_n$ on $[q_1,q_2]$ to be this curve. This procedure gives rise to a path $\gamma_n\colon [0,1]\to X_n$. For each $p\in [0,1]$ there exists a complementary interval $(q_1,q_2)$ of $Q_n$ whose closure contains $p$. We have
\begin{align*} 
    d_X( \gamma(p), f_n(\gamma_n(p))) &\leq d_X(\gamma(p),\gamma(q_1)) + d_X(\gamma(q_1), f_n(\gamma_n(q_1)))\\&\qquad\qquad+ d_X( f_n(\gamma_n(q_1)), f_n(\gamma_n(p))) \\
    &\leq \varepsilon_n + \varepsilon_n +\varepsilon_n+d_{X_n} (\gamma_n(q_1),\gamma_n(p))\\
    &\leq 3\varepsilon_n+ \diam (\gamma_n([q_1,q_2]))\\
    &\leq 3\varepsilon_n+\varepsilon_n'
\end{align*}
for all $n\geq N$. Hence, $f_n\circ \gamma_n$ converges uniformly to $\gamma$, as desired.
\end{proof}

Let $X$ be a metric space. For each pair of disjoint continua $E,F\subset X$, we define $\Gamma^*(E,F;X)$ to be the family of rectifiable curves in $X\setminus (E\cup F)$ separating $E$ from $F$. That is, for each $\gamma\in \Gamma^*(E,F;X)$, the sets $E$ and $F$ lie in different components of $X\setminus |\gamma|$. 

\begin{lemm}[{cf.\ \cite[Lemma 2.4]{NR:21}}]\label{lemma:modulus_bound} \label{lemma:modulus_bound_convergence}Let $\{X_n\}_{n=1}^\infty$ be a sequence of asymptotically uniformly locally path connected compact metric spaces converging in the Gromov--Hausdorff sense to a compact metric surface $X$. Moreover, suppose that $\limsup_{n\to\infty} \mathcal H^2(X_n)<\infty$. 

\vspace{1em}
\noindent
Then for each $\delta>0$ and for any sequence of pairs of disjoint continua $E_n,F_n \subset X_n$ with $\min\{\diam(E_n),\diam(F_n)\}\geq \delta$ we have
\begin{align*}
    \limsup_{n\to\infty}\Mod \Gamma^*(E_n,F_n;X_n) <\infty.
\end{align*}
\end{lemm}
The proof is the same as \cite[Lemma 2.4]{NR:21}, where one replaces \cite[Prop.\ 2.2 (ii)]{NR:21} with Proposition \ref{prop:asympt}. We present the main ingredients here for the convenience of the reader.

\begin{proof}
It suffices to show that there exists $\eta>0$, depending on $\delta$ but not on $n$, such that if $E_n,F_n\subset X_n$ is a pair of disjoint continua with $\min \{\diam(E_n), \diam(F_n)\}\geq \delta$, then $\ell(\gamma)\geq \eta$ for every $\gamma\in \Gamma^*(E_n,F_n;X_n)$, $n\in \N$. 

We argue by contradiction. Let $f_n\colon X_n\to X$ be a sequence of $\varepsilon_n$-isometries, where $\varepsilon_n\to 0$. Suppose that there exist sequences of disjoint continua $E_n,F_n\subset X_n$ with $\min\{\diam(E_n),\diam(F_n)\}\geq \delta$ for some $\delta>0$ and a sequence of paths $\gamma_n\in \Gamma^*(E_n,F_n;X_n)$ with $\ell(\gamma_n)\to 0$ as $n\to\infty$.  After passing to a subsequence, we assume that $f_n\circ \gamma_n$ converges uniformly to a point $x_0\in X$ (\cite[Prop.\ 2.2 (i)]{NR:21}) and the sets $f_n(E_n)$ and $f_n(F_n)$ converge in the Hausdorff sense to continua $E$ and $F$, respectively, with $\min\{\diam(E),\diam(F)\}\geq \delta$ (\cite[Prop.\ 2.2 (iii)]{NR:21}). 

Since $X$ is a surface, there exists a path $\eta\colon [0,1]\to X\setminus \{x_0\}$ with $\eta(0)\in E$ and $\eta(1)\in F$. By the Hausdorff convergence of $f_n(E_n)$ and $f_n(F_n)$ to $E$ and $F$, respectively, there exist points $a_n\in E_n$ and $b_n\in F_n$ such that $f_n(a_n)$ converges to $\eta(0)$ and $f_n(b_n)$ converges to $\eta(1)$. By Proposition \ref{prop:asympt}, there exist paths $\eta_n\colon [0,1]\to X_n$ for sufficiently large $n\in \N$ such that $\eta_n(0)=a_n\in E_n$, $\eta_n(1)=b_n\in F_n$, and $f_n\circ \eta_n$ converges uniformly to $\eta$. 

Since $\gamma_n$ separates $E_n$ from $F_n$ and $\eta_n$ connects $E_n$ and $F_n$, the paths $\gamma_n$ and $\eta_n$ intersect each other for each sufficiently large $n\in \N$. The uniform convergence of $f_n\circ \gamma_n$ and $f_n\circ \eta_n$ to $x_0$ and $\eta$, respectively, implies that $\eta$ intersects the point $x_0$. This is a contradiction. 
\end{proof}

\begin{lemm}[{cf.\ \cite[Lemma 2.3]{NR:21}}]\label{lemma:boundary_convergence}
Let $X$ be a metric space homeomorphic to a topological closed disk and $\{X_n\}_{n=1}^\infty$ be a sequence of metric spaces homeomorphic to $X$. Suppose that there exists an approximately isometric sequence $f_n\colon X_n\to X$, $n\in \N$, of topological embeddings. Then 
$$\liminf_{n\to\infty} \diam(\partial X_n)\geq \diam(\partial X).$$
\end{lemm}

\begin{proof} 
Suppose that $f_n$ is an $\varepsilon_n$-isometry, where $\varepsilon_n\to 0$ as $n\to\infty$. We claim that for each $r>0$ there exists $N\in \N$ such that for $n\geq N$ and for each $x\in \partial X$ there exists a point $y_n\in \partial X_n$ with $d(f_n(y_n),x)<r$. In this case,  $\partial X\subset N_{r}(f_n(\partial X_n))$ for $n\geq N$, so
\begin{align*}
    \diam(\partial X_n) \geq \diam(f_n(\partial X_n)) -\varepsilon_n \geq \diam (\partial X) -2r-\varepsilon_n.
\end{align*}
We first let $n\to\infty$ and then $r\to 0$ to obtain the desired conclusion.

Now we prove the claim. For each $r>0$, using the local path connectivity of $ X$, we may find $\delta>0$ such that if $x,y\in X$ and $d(x,y)<\delta$, then there exists a path in $X$ containing $x$ and $y$ with diameter less than $r$. We now let $N\in \N$ such that $\varepsilon_n<\delta$ for $n\geq N$ and fix $x\in \partial X$. By the definition of an $\varepsilon_n$-isometry, there exists $x_n\in X_n$ with $x_n'=f(x_n)$ such that $d(x_n',x)<\varepsilon_n<\delta$. Thus, there exists a path $\gamma$ connecting $x$ and $x_n'$ with diameter less than $r$. Since $f_n$ is an embedding, the set $V=\Int (f_n(X_n))$ is an open Jordan region in $X$, bounded by $f_n(\partial X_n)$. The curve $\gamma$ connects a point of $\br V = f_n(X_n)$ to a point of $X\setminus V$, hence it intersects $f_n(\partial X_n)$ at a point $y_n'$. We let $y_n=f_n^{-1}(y_n')$ and the claim follows.  
\end{proof}

\smallskip

\subsection{Quasiconformal maps} \label{sec:qc_maps}
Let $X,Y$ be metric surfaces of locally finite Hausdorff $2$-measure. A homeomorphism $h\colon X\to Y$ is \textit{quasiconformal} if there exists $K\geq 1$ such that for all curve families $\Gamma$ in $X$ we have
$$K^{-1} \Mod\Gamma\leq \Mod h(\Gamma)\leq K\Mod\Gamma.$$
In this case, we say that $h$ is $K$-quasiconformal. Recall that a continuous, surjective, proper, and cell-like map $h\colon X\to Y$ is {weakly quasiconformal} if  there exists $K\geq 1$ such that for every curve family $\Gamma$ in $X$ we have
$$\Mod \Gamma\leq K \Mod h(\Gamma).$$
In this case, we say that $h$ is weakly $K$-quasiconformal. If $X$ and $Y$ are compact surfaces that are homeomorphic to each other, then we may replace cell-likeness with the requirement that $h$ is monotone; that is, the preimage of every {point} is a continuum. In this case, we also have the stronger statement that that the preimage of every connected set in $Y$ is connected in $X$ \cite[(2.2), Chap.\ VIII, p.~138]{Why:58}.
See Section \ref{sec:topological} blow for further topological properties.

The next  theorem of Williams (\cite[Thm.\ 1.1 and Cor.\ 3.9]{Wil:12}) relates the above definitions of quasiconformality with the ``analytic" definition that relies on upper gradients; see also the discussion in \cite[Sect.\ 2.4]{NR:21}.

\begin{thm}[Definitions of quasiconformality]\label{theorem:definitions_qc}
Let $X,Y$ be metric surfaces of locally finite Hausdorff $2$-measure and let $h\colon X\to Y$ be a continuous map. The following are equivalent.
\begin{enumerate}[label=\normalfont(\roman*)]
    \item\label{def:i} $h\in N^{1,2}_{\loc}(X,Y)$ and there exists a weak upper gradient $g$ of $h$ such that for every Borel set $E\subset Y$ we have
    $$\int_{h^{-1}(E)} g^2\, d\mathcal H^2 \leq K \mathcal H^2(E).$$
    \item\label{def:i'}Each point of $X$ has a neighborhood $U$ such that $h|_U\in N^{1,2}(U,Y)$ and there exists an upper gradient $g_U$ of $h|_U$ such that for every Borel set $E\subset Y$ we have
    $$\int_{(h|_{U})^{-1}(E)} g_U^2\, d\mathcal H^2 \leq K \mathcal H^2(E).$$
    \item\label{def:ii} For every curve family $\Gamma$ in $X$ we have
    $$\Mod \Gamma \leq K\Mod h(\Gamma).$$
\end{enumerate}
\end{thm}

Next, we state some boundary extension results for weakly quasiconformal maps.

\begin{lemm}
Let $X$ be a metric space and let $h\colon \br \D\to X$ be a continuous map with $h|_{\D}\in N^{1,2}(\D,X)$. Then $h\in N^{1,2}(\br \D,X)$ and $h$ has the same minimal weak upper gradient as $h|_{\D}$. 
\end{lemm}
\begin{proof}
Let $g\in L^2(\D)$ be a weak upper gradient of $h|_{\D}$. It suffices to show that $g$ is a weak upper gradient of $h$ in $\br \D$. For $r>1$, define $h_r(x)=h(r^{-1}x)$, $x\in \br \D$. Observe that $g_r(x)=r^{-1}g(r^{-1}x)$ is a weak upper gradient of $h_r$ in $\br \D$. Thus, $h_r\in N^{1,2}(\br \D, X)$. It is elementary to show, using approximation by continuous functions, that $g_r$ converges to $g$ in $L^2(\br \D)$ as $r\to 1$. Since $h_r$ converges to $h$ uniformly in $\br \D$, we conclude from (a slight variant of)  \cite[Prop.\ 7.3.7, p.~193]{HKST:15} that $h\in N^{1,2}(\br \D,X)$ with weak upper gradient $g$.
\end{proof}

Combining this lemma with Theorem \ref{theorem:definitions_qc} gives the next corollary.

\begin{cor}\label{cor:extension}
Let $X$ be a metric surface of finite Hausdorff $2$-measure that is homeomorphic to a topological closed disk and let $h\colon \br{\mathbb D}\to X$ be a continuous, surjective, and monotone map. If $h(\D)=\Int(X)$ and $h|_{\D}\colon \D \to \Int(X)$ is weakly $K$-quasiconformal for some $K\geq 1$, then $h$ is weakly $K$-quasiconformal. 
\end{cor}

The next statement is an analogue of Carath\'eodory's extension theorem. 

\begin{thm}\label{theorem:caratheodory}
Let $X$ be a metric surface of finite Hausdorff $2$-measure that is homeomorphic to a topological closed disk and let $h\colon \D \to \Int(X)$ be a $K$-quasi\-con\-formal homeomorphism for some $K\geq 1$. Then $h$ extends to a weakly $K$-quasiconformal map from $\br \D$ onto $X$. 
\end{thm}
\begin{proof}
According to a result of Ikonen \cite[Thm.\ 1.1]{Iko:21}, there exists a continuous, surjective, and monotone extension $h\colon \br \D \to X$. By Corollary \ref{cor:extension}, the extension is weakly quasiconformal. 
\end{proof}

\begin{rem}\label{remark:caratheodory}
In fact, the conclusion of Theorem \ref{theorem:caratheodory} is true under the mere assumption that $h|_{\D}$ is a weakly quasiconformal map rather than a quasiconformal homeomorphism; in particular, Corollary \ref{cor:extension} is a consequence of this stronger statement.  In order to show this, one has to modify slightly the argument in the proof of \cite[Thm.\ 1.1]{Iko:21}, which we used in the above proof of Theorem \ref{theorem:caratheodory}.
\end{rem}

\smallskip

\subsection{Classical uniformization}
Recall that a homeomorphism between Riemann surfaces is conformal if it is complex differentiable in local coordinates.  The classical uniformization theorem \cite[Thm.\ 15.12, p.~242]{Mar:19} has the following consequence.

\begin{thm}[Classical uniformization]\label{theorem:uniformization_classical}
Let $X$ be a Riemann surface homeomorphic to $\widehat{\C}$, $\br \D$, or $\C$. Then there exists a conformal homeomorphism from $\widehat{\C}$, $\br \D$, or $\D$ or $\C$, respectively, onto $X$. 
\end{thm}

We say that a metric $d$ on a Riemann surface $X$ is \textit{compatible} with the complex structure of $X$ if each local conformal chart $\varphi$ from an open subset $U$ of $X$ into the plane $\C$ is $1$-quasiconformal, where $U$ is equipped with the restriction of the metric $d$. When the metric $d$ arises from a Riemannian metric $g$, we will simply say that $g$ is compatible with the complex structure. Another implication of the classical uniformization is the next statement. 

\begin{thm}[Classical Riemannian uniformization]\label{theorem:uniformization_riemannian}
Let $X$ be a Riemann surface. Then there exists a Riemannian metric $g$ on $X$ that is complete, has constant curvature, and is compatible with the complex structure of $X$. 
\end{thm}

Theorems \ref{theorem:uniformization_classical} and \ref{theorem:uniformization_riemannian} are often used in conjunction with the existence of isothermal coordinates in Riemannian surfaces.
\begin{thm}[Isothermal coordinates]\label{theorem:isothermal}
Let $(X,g)$ be a Riemannian surface. Then there exists a complex structure on $X$ that is compatible with the Riemannian metric $g$.
\end{thm}

The next lemma give the relation between conformality and $1$-quasiconformality of maps between Riemann surfaces.

\begin{lemm}\label{lemma:conformal_compatible}
Let $X,Y$ be Riemann surfaces with metrics $d_X,d_Y$, respectively, that are compatible with the complex structures. Then each homeomorphism $h\colon X\to Y$ is conformal if and only if $h\colon (X,d_X)\to (Y,d_Y)$ is $1$-quasiconformal.  
\end{lemm}

The proof relies on the fact that conformal maps coincide with $1$-quasiconformal maps in planar domains (see \cite{LV:73}) and on the fact that global quasiconformality of a map between metric surfaces follows from the local quasiconformality, by Theorem \ref{theorem:definitions_qc}.

Each orientable polyhedral surface $X$ admits  a  complex structure  and becomes Riemann surface. The complex structure is natural in the sense that it is compatible with the polyhedral metric; see \cite[Sect.\ 2.5]{NR:21} for more details. In what follows, polyhedral Riemann surfaces refer to this natural complex structure. We note that if we equip a polyhedral surface with a metric that is locally isometric to the polyhedral metric, then it is still compatible with the natural complex structure.

\smallskip

\subsection{Proof of Theorem \ref{thm:one-sided_qc} for compact surfaces}
The proof bears no differences to the proof of \cite[Thm.\ 1.3]{NR:21} other than the use of the extended polyhedral approximation Theorem \ref{thm:extended_polyhedral_approximation} in place of \cite[Thm.\ 1.1]{NR:21} and the technical statements of Proposition \ref{prop:asympt} and Lemma \ref{lemma:modulus_bound} in place of \cite[Prop.\ 2.2 (ii)]{NR:21} and \cite[Lemma 2.4]{NR:21}, respectively. Thus, below we only emphasize the differences.

We will establish the following statement that readily implies Theorem \ref{thm:one-sided_qc}.

\begin{thm}\label{theorem:uniformization_long}
Let $\Omega=\widehat\C$ or $\Omega=\overline{\mathbb D}$. Let $X$ be a metric space homeomorphic to $\Omega$ with $\mathcal H^2(X)<\infty$ and $\{X_n\}_{n=1}^\infty$ be an asymptotically uniformly locally path connected sequence of metric spaces homeomorphic to $X$. Suppose that there exists an approximately isometric sequence $f_n\colon X_n\to X$, $n\in \N$, such that there exists $K\geq 1$ with
$$\limsup_{n\to\infty}\mathcal H^2 (f_n^{-1}(A)) \leq K\mathcal H^2(A)$$
for all compact sets $A\subset X$.  If $h_n\colon \Omega\to X_n$, $n\in \N$, is a normalized sequence of weakly $L$-quasiconformal maps for some $L\geq 1$, then $f_n\circ h_n$ has a subsequence that converges uniformly to a weakly $(K\cdot L)$-quasiconformal map $h\colon \Omega \to X$. 
\end{thm}

Here we say that a sequence $h_n\colon X_n\to Y_n$ of maps between {compact} metric spaces is \textit{normalized} if there exists a value $\delta>0$ and a sequence of triples $a_n,b_n,c_n\in X_n$ with mutual distances bounded from below by $\delta$ such that the mutual distances between the points $h_n(a_n),h_n(b_n),h_n(c_n)$ are also bounded from below by $\delta$, where $\delta$ is independent of $n\in \N$.

That Theorem \ref{theorem:uniformization_long} implies Theorem \ref{thm:one-sided_qc} in the compact case is standard (see also \cite[Sect.\ 6.1]{NR:21}), with the exception of finding the optimal value of the quasiconformal dilatation, a procedure that we describe in Section \ref{sec:global}.

\begin{proof}[Proof of Theorem \ref{thm:one-sided_qc} for compact $X$]
Theorem \ref{thm:extended_polyhedral_approximation}, in combination with Lemma \ref{lemma:aulpc}, provides us with an asymptotically uniformly locally path connected sequence of polyhedral Riemann surfaces $\{(X_n,d_{X_n})\}_{n=1}^\infty$ and an approximately isometric sequence of topological embeddings $f_n\colon X_n\to X$, $n\in \N$, such that the Hausdorff measure inequality in the assumption of Theorem \ref{theorem:uniformization_long} is satisfied. The metric $d_{X_n}$ is locally isometric to the polyhedral metric on $X_n$ and thus it is compatible with the complex structure. By the classical uniformization theorem (Theorem \ref{theorem:uniformization_classical}), for each $n\in \N$ there exists a conformal homeomorphism $h_n\colon \Omega \to X_n$, which is also $1$-quasiconformal by Lemma \ref{lemma:conformal_compatible}. Thus, in order to apply Theorem \ref{theorem:uniformization_long}, it only remains to normalize the sequence $h_n$. If $\Omega=\widehat{\C}$, then it suffices to precompose $h_n$ with a suitable M\"obius transformation of $\widehat{\C}$. If $\Omega= \overline{\mathbb D}$, then, by  Lemma \ref{lemma:boundary_convergence}, $\diam(\partial X_n)$ is uniformly bounded from below away from $0$. Hence, we may find points $a_n',b_n',c_n'\in \partial X_n$ with mutual distances uniformly bounded from below. We now precompose $h_n$ with a M\"obius transformation of $\overline {\mathbb D}$, so that the preimages of $a_n',b_n',c_n'$ are the points $1,i,-1\in \partial \mathbb D$.
\end{proof}

Next, we focus on proving Theorem \ref{theorem:uniformization_long}. Let $\{X_n\}_{n=1}^\infty$ be a sequence of metric surfaces homeomorphic to $X$ that converges to $X$ in the Gromov--Hausdorff sense. By assumption, we have 
\begin{align*}
    \limsup_{n\to\infty}\mathcal H^2(X_n) \leq K\mathcal H^2(X)
\end{align*}
for an absolute constant $K\geq 1$.  Also, let $f_n\colon X_n\to X$, $n\in \N$, be an approximately isometric sequence and let $h_n\colon \Omega\to X_n$, $n\in \N$, be a normalized sequence of weakly $L$-quasiconformal maps. We will use the assumption that the spaces $X_n$, $n\in \N$, are asymptotically uniformly locally path connected, which enables us to apply Proposition \ref{prop:asympt} and Lemma \ref{lemma:modulus_bound}. 

\begin{lemm}[Equicontinuity]
The sequence $f_n\circ h_n\colon \Omega\to X$, $n\in \N$, is asymptotically uniformly equicontinuous.
\end{lemm}
\begin{proof}
This is proved exactly in the same way as \cite[Lemma 6.3]{NR:21}, where one uses Lemma \ref{lemma:modulus_bound} in place of \cite[Lemma 2.4]{NR:21}. Concisely, if the statement fails, then there exists a sequence of continua $E_n$, $n\in \N$, in $\Omega$ with diameters converging to $0$ such that, after passing to a subsequence, $h_n(E_n)$ has diameter uniformly bounded from below away from $0$ as $n\to\infty$. Using the linear local connectivity of $\Omega$ and the fact that $h_n$ is normalized, one can find a sequence of continua $F_n$, $n\in \N$, in $\Omega$ with diameters bounded away from $0$ and with $\dist(E_n,F_n)$ bounded away from $0$ as $n\to\infty$ such that the continua $h_n(F_n)$ also have large diameter. Standard modulus estimates in $\Omega$ show that $\Mod \Gamma^*(E_n,F_n;\Omega) \to \infty$ so by the weak $L$-quasiconformality of $h_n$, $\Mod \Gamma^*(h_n(E_n),h_n(F_n); X_n) \to \infty$. This contradicts Lemma \ref{lemma:modulus_bound}.
\end{proof}

\begin{lemm}[Convergence]
The sequence $f_n\circ h_n\colon \Omega\to X$, $n\in \N$, has a subsequence that converges uniformly to a continuous, surjective, and monotone map $h\colon \Omega \to X$.
\end{lemm}
\begin{proof}
The convergence and surjectivity follow from the asymptotic uniform equi\-continuity, so we only argue for the monotonicity, which relies on the asymptotic uniform local path connectivity of $X_n$, $n\in \N$ (cf.\ \cite[Lemma 6.4]{NR:21}). If $h$ is not monotone, then a point $x\in X$ has a disconnected preimage $h^{-1}(x)$ and by planar topology there exist points $a,b\in h^{-1}(x)$ and a curve $\gamma$ in $\Omega\setminus h^{-1}(x)$ separating them.  Note that $f_n(h_n(a))$ and $f_n(h_n(b))$ converge to $x$, which we consider as a constant path. By Proposition \ref{prop:asympt}, there exists a sequence of paths $\gamma_n\colon [0,1]\to X_n$, $n\geq N$, such that $\gamma_n(0)=h_n(a)$, $\gamma_n(1)=h_n(b)$, and $f_n\circ \gamma_n$ converges uniformly to the constant path $x$ as $n\to\infty$. The monotonicity of $h_n$ implies that $h_n^{-1}(|\gamma_n|)$ is a continuum joining $a$ and $b$.  Since $\gamma$ separates $a$ and $b$, we conclude that $\gamma$ intersects $h_n^{-1}(|\gamma_n|)$. Therefore, the paths $h_n\circ \gamma$ and $\gamma_n$ intersect. By the uniform convergence of $f_n\circ h_n\circ \gamma$ to $h\circ \gamma$ and of $f_n\circ \gamma_n$ to $x$, we conclude that $h\circ \gamma$ intersects $x$, a contradiction.
\end{proof}

Next, we discuss the regularity properties of $h$. It is crucial here that $f_n$ has the property that for every compact set $A\subset X$ we have
\begin{align}\label{ineq:unif_measure}
    \limsup_{n\to\infty}\mathcal H^2(f_n^{-1}(A))\leq K\mathcal H^2(A)
\end{align}
for some uniform constant $K>0$.  By Theorem \ref{theorem:definitions_qc} \ref{def:i}, $h_n$ has an upper gradient $g_{h_n}$ with the property that 
\begin{align}\label{ineq:unif_integral}
    \int_{h_n^{-1}(E)} g_{h_n}^2 \, d\mathcal H^2 \leq L\mathcal H^2(E)
\end{align}
for each Borel set $E\subset X_n$. With absolutely no changes, the proof of the following lemma is the same as in \cite[Sect.\ 6.1.2]{NR:21}.

\begin{lemm}[Quasiconformality]
The sequence of upper gradients $g_{h_n}$ of $h_n$, $n\in \N$, has a subsequence that converges weakly in $L^2(\Omega)$ to a function $g_h$ that is a weak upper gradient of $h$. Moreover, for each Borel set $E\subset X$ we have
$$\int_{h^{-1}(E)}g_h^2\, d\mathcal H^2 \leq KL \mathcal H^2(E).$$
\end{lemm}

Based on this lemma, the weak quasiconformality of $h$ follows from Theorem \ref{theorem:definitions_qc} and thus the proof of Theorem \ref{theorem:uniformization_long} is completed. We provide a sketch of the proof of the lemma. The weak limit $g_h$ of $g_{h_n}$ is obtained by the Banach--Alaoglu theorem \cite[Thm.\ 2.4.1]{HKST:15}, due to the uniform bound of the $L^2$ norm of $g_{h_n}$, given by \eqref{ineq:unif_measure} and \eqref{ineq:unif_integral}. Mazur's lemma \cite[p.~19]{HKST:15} allows one to upgrade weak convergence of $g_{h_n}$ to strong convergence of convex combinations of $g_{h_n}$. The fact that $g_h$ is a weak upper gradient of $h$ then follows from Fuglede's lemma \cite[p.~131]{HKST:15}. The integral inequality for $g_h^2$ in the conclusion of the lemma is a consequence of \eqref{ineq:unif_measure}, \eqref{ineq:unif_integral}, and the weak convergence of $g_{h_n}$ to $g_h$.  

\bigskip

\section{Local to global uniformization}\label{sec:global}

In this section we prove the global uniformization result of Theorem \ref{thm:uniformization_global} and we use it to derive Theorem \ref{thm:one-sided_qc} in the general case. We also establish the minimal value of the dilatation of the weak quasiconformal parametrizations in these theorems.

\smallskip

\subsection{Minimizing dilatation}
The next general result allows us to minimize the dilatation of a weakly quasiconformal map from a Riemannian surface onto a metric surface.

\begin{thm}\label{thm:minimize_dilatation}
Let $X$ be a metric surface of locally finite Hausdorff $2$-measure,  $(Z,g)$ be a Riemannian surface, and $h\colon (Z,g) \to X$ be a weakly quasiconformal map. Then there exists a Riemannian surface $(\widetilde Z,\widetilde g)$ and a quasiconformal map $\psi \colon (\widetilde Z,\widetilde g)\to (Z,g)$ such that $h\circ \psi\colon (\widetilde Z,\widetilde g)\to X$ is weakly $(4/\pi)$-quasiconformal. 
\end{thm}
The constant $4/\pi$ is optimal, as was observed by Rajala \cite[Example 2.2]{Raj:17}. The proof follows the steps of Rajala \cite[Sect.\ 14]{Raj:17}, so we only provide a sketch. 
\begin{proof} 
Since $h$ belongs to $N^{1,2}_{\loc}(Z,X)$, $h$ is approximately metrically differentiable. This means that for a.e.\ $x \in Z$, there is a unique seminorm $N_x \colon \mathbb{R}^2 \to [0, \infty)$ for which (identifying a neighborhood of $x$ in $Z$ with a subset of $\mathbb{R}^2$ via the exponential map)
\[\underset{y\to x}{\aplim}\frac{d(h(y),h(x)) - N_x(y-x)}{|y-x|} = 0,\]
where $\aplim$ denotes the approximate limit; see Definition 4.1 in \cite{LW:17}. The correspondence $x \mapsto N_x$ is measurable as a map from $Z$ into the metric space of seminorms in $\R^2$, where the distance between two seminorms $s$ and $s'$ is defined as $\sup \{|s(y)-s'(y)|: y \in \mathbb{S}^1\}$; see Proposition 4.3 in \cite{LW:17}. Let $B_x = \{y \in \mathbb{R}^2: N_x(y) \leq 1\}$ be the closed unit ball in the seminorm $N_x$. If $N_x$ is a norm, then $B_x$ is a symmetric convex body in $\mathbb{R}^2$. Let $L_x =  \sup\{N_x(y): y \in \mathbb{S}^1\}$ denote the maximal stretching of $N_x$. As shown in Lemma 3.1 of \cite{LW:18a}, $L_x$ is a representative of the minimal weak upper gradient of $h$ that we denote by $g_h$. 

We define the Jacobian of $N_x$ to be $J_x = \pi/|B_x|$, where $|B_x|$ is the Lebesgue $2$-measure of $B_x$. Note that $|B_x|=\infty$ and $J_x=0$ if $N_x$ is not a norm.  It was shown in \cite[Thm.\ 7.1]{NR:21} that the weakly quasiconformal map $h$ has a Jacobian $J_h$ that satisfies $g_h(x)^2\leq K J_h(x)$ for a.e.\ $x\in Z$ and which is by definition the Radon--Nikodym derivative of the measure $\mathcal H^2 \circ h$. Moreover, this measure agrees with the measure \[E\mapsto \int \# (h^{-1}(y)\cap E)\, d\mathcal H^2.\]See \cite[Remark 7.2]{NR:21}. Finally, the absolutely continuous part of the latter measure agrees with the measure \[E\mapsto \int_E J_x\, d\mathcal H^2.\] This is true in general for Lipschitz maps \cite[Thm.\ 7]{Kir:94} and can be established for the map $h$ in the Sobolev space by decomposing $h$ into countably many Lipschitz maps up to a set of measure zero \cite[Lemma 14.1]{Raj:17}. Altogether, we have $J_h(x)=J_x$ for a.e.\ $x\in Z$. Thus $L_x^2 \leq KJ_x$ for a.e.\ $x \in Z$.  In particular, we see that $N_x$ is the zero seminorm for a.e.\ $x \in Z$ for which $N_x$ is not a norm, in which case we have $B_x = \mathbb{R}^2$.

Recall that the \textit{John ellipse} of a planar symmetric convex body is the unique ellipse of maximum area it contains. We take $E_x$ to be the John ellipse corresponding to $B_x$ whenever $N_x$ is a norm. We define $E_x$ to be the Euclidean closed unit ball otherwise. The John ellipse is continuous as a function of the norm $N_x$; see for example Section 3 in \cite{LYZ:05}. Moreover, the set $\{x \in Z: N_x = 0\}$ is measurable. Thus the correspondence $x \mapsto E_x$ is measurable.  By John's theorem (see Theorem 3.1 in \cite{Ball:97} for a statement), $E_x$ satisfies $E_x \subset B_x \subset \sqrt{2} E_x$ whenever $N_x$ is a norm. The latter property and the relation $L_x^2 \leq KJ_x$ imply that the ellipse field $E_x$ has uniformly bounded eccentricity. 

Let $\mu$ be the corresponding Beltrami differential in the Riemann surface $Z$. By \cite[Prop.\ 4.8.12]{Hub:06} there exists a conformal structure on $Z$ giving rise to a Riemann surface $\widetilde Z$ such that the identity map $\psi \colon \widetilde Z \to Z$ is  $K'$-quasiconformal in local coordinates for some $K'\geq 1$ and maps infinitesimal balls to the corresponding infinitesimal ellipses specified by $\mu$.  Using isothermal coordinates (Theorem \ref{theorem:isothermal}), we consider a Riemannian metric $\widetilde g$ compatible with the complex structure of $\widetilde Z$. 

The composite $\widetilde{h} = h \circ \psi$ is  a weakly quasiconformal map and in particular belongs to the space $N^{1,2}_{\loc}(\widetilde{Z},X)$. Thus $\widetilde{h}$ is approximately metrically differentiable. Let $z \in \widetilde{Z}$ be a point of approximate metric differentiability of $h$. Denote by $\widetilde{N}_z$ the approximate metric derivative and $\widetilde{B}_z$ the corresponding closed unit ball.  For all $z \in \widetilde{Z}$ at which $\psi$ is differentiable, let $D\psi_z$ denote the derivative of $\psi$ at $z$. Then for a.e.\ $z\in \widetilde Z$, $(D\psi_z)^{-1}$ takes the ellipse $E_{\psi(z)}$ onto the closed disk of radius $r_z$ centered at $0$ for some $r_z>0$. For a.e.\ $z \in \widetilde{Z}$ for which $N_{\psi(z)}$ is a norm, we have $\widetilde N_z =N_{\psi(z)} \circ D\psi_z$ and $D\psi_z (\widetilde{B}_z) =B_{\psi(z)}$. In this case $\widetilde{B}_z$ contains a maximal disk of radius $r_z$, and so $\widetilde{N}_z$ has maximal stretching equal to $1/r_z$. Moreover, this disk is the John ellipse for $\widetilde{B}_z$. As a consequence of John's theorem (cf. Theorem 6.2 and the following discussion in \cite{Ball:97}), the set $\{y \in \mathbb{R}^2: \widetilde{N}_z(y) \leq 1\}$ is contained in a square of area $4r_z^2$. In particular, the Jacobian $\widetilde{J}_z$ satisfies $\widetilde{J}_z \geq \pi/(4r_z^2)$. In the next case, for a.e.\ $z \in \widetilde{Z}$ for which $N_{\psi(z)}=0$, we have $\widetilde{N}_z=0$ as well, namely at any point $z$ at which $\psi$ is differentiable.   In all cases, we have $\widetilde{L}_z^2 \leq (4/\pi)\widetilde{J}_z$ for a.e.\ $z \in \widetilde{Z}$, which implies that $\widetilde{h}$ is weakly $(4/\pi)$-quasiconformal.
\end{proof}

\smallskip

\subsection{Topological preliminaries}\label{sec:topological}
Let $X,Y$ be topological spaces and $\nu\colon X\to Y$ be a continuous map. Recall that $\nu$ is monotone if the preimage of each point is a continuum. Also, $\nu$ is {proper} if the preimage of each compact set is compact. A compact set $A$ in $X$ is \textit{cell-like in $X$} if for each open neighborhood $U$ of $A$ in $X$, $A$ can be contracted within $U$ to a point. Finally, the map $\nu$ \textit{cell-like} if the preimage of each point in $Y$ is a cell-like set in $X$. 

In $1$-manifolds, cell-like sets coincide with homeomorphic images of compact intervals. In topological surfaces without boundary, cell-like sets coincide with continua having a simply connected open neighborhood that they do not separate \cite[Cor.\ 15.2B and Cor.\ 15.4C]{Dav:86}.

\begin{thm}[Approximation for cell-like maps]\label{theorem:cell_like_approximation}
Let $\nu \colon X\to Y$ be a proper cell-like map between metric surfaces such that $\nu(\partial X)\subset \partial Y$ and the restriction $\nu|_{\partial X} \colon \partial X \to \partial Y$ is a proper cell-like map. Then $\nu$ can be approximated uniformly by homeomorphisms. Moreover, if $U$ is an open subset of $Y$, then $\nu|_{\nu^{-1}(U)}\colon \nu^{-1}(U) \to U$ can be approximated uniformly by homeomorphisms.
\end{thm}
See \cite[Cor.\ 25.1A]{Dav:86} for the case of surfaces without boundary and \cite{Sieb:72} for the general case. For compact surfaces one has the following  stronger statement.

\begin{thm}[Approximation for monotone maps]\label{theorem:monotone_approximation}
Let $\nu \colon X\to Y$ is a continuous and surjective map between compact metric surfaces that are homeomorphic. The following are equivalent.
\begin{enumerate}[label=\normalfont(\arabic*)]
    \item  $\nu$ is monotone.
    \item $\nu $ is cell-like.
    \item $\nu $ can be approximated uniformly by homeomorphisms.
\end{enumerate}
In this case, we have $\nu(\partial X)=\partial Y$ and $\nu^{-1}(\Int(Y))\subset \Int(X)$.
\end{thm}
\begin{proof}
If $\nu$ is monotone, then a result of Youngs \cite[p.~92]{You:48} implies that $\nu$ is the uniform limit of homeomorphisms between $X$ and $Y$. If $\nu$ can be approximated by homeomorphisms, then it is cell-like; this is true for maps between spaces that are absolute neighborhood retracts (abbr.\ ANR) by \cite[Thm.\ 17.4]{Dav:86} and every compact manifold is an ANR \cite[Cor.\ 14.8A]{Dav:86}. Finally, a cell-like map is trivially monotone, since cell-like sets are connected \cite[Cor.\ 16.3A]{Dav:86}. This completes the proof of the equivalences. The last properties in the statement of the lemma are immediate for uniform limits of homeomorphisms.
\end{proof}

\smallskip

\subsection{Proof of Theorems \ref{thm:uniformization_global} and \ref{thm:one-sided_qc}}

This section is devoted to the proofs of Theorem \ref{thm:uniformization_global} and of the general case in Theorem \ref{thm:one-sided_qc}. We split the proof of Theorem \ref{thm:uniformization_global} into two main cases: surfaces without boundary and surfaces with boundary.

\begin{proof}[Proof of Theorem \ref{thm:uniformization_global}: surfaces without boundary]
By Theorem \ref{thm:extended_polyhedral_approximation}, there exists a sequence of polyhedral surfaces $\{X_n\}_{n=1}^\infty$ converging to $X$ and a sequence of approximately isometric proper topological embeddings $f_n\colon X_n \to X$, $n\in \N$; in fact, since $X$ has no boundary, the maps $f_n$ are homeomorphisms. By the paracompactness of the surface $X$, there exists a locally finite countable collection of topological open disks $D_k$, $k\in \N$, whose union covers the surface $X$. For each $k\in \N$ and $n\in \N$, we define the topological closed disk $\br{D_{k}^n} =f_n^{-1}(\br{D_k})$ and equip it with the restriction of the metric $d_{X_n}$. We note that for each $k\in \N$, the sequence $f_n|_{\br{D_k^n}}\colon \br {D_{k}^n}\to \br{D_k}$, $n\in \N$, satisfies the conclusions of Theorem \ref{thm:extended_polyhedral_approximation}.  In particular, by Lemma \ref{lemma:aulpc}, the spaces $\{\br{D_k^n}\}_{n=1}^\infty$ are asymptotically uniformly locally path connected. 

\vspace{1em}
\noindent
\textit{Step 1: Normalizations in the spaces $X_n$ and $X$.} For fixed $k\in \N$ we consider distinct points $p_k,q_k\in D_k$. For each $n\in\N$ the map $f_n|_{\br{D_k^n}}\colon \br {D_{k}^n}\to \br{D_k}$ is a homeomorphism, so there exist points $p_k^n,q_k^n\in D_k^n$ such that $f_n(p_k^n)=p_k$ and $f_n(q_k^n)= q_k$. We have $f_n^{-1}(\partial D_k) =\partial D_{k}^n$, so the fact that $f_n$ is approximately isometric implies that
$$ \liminf_{n\to\infty} \dist(p_k^n, \partial D_k^n) \geq \dist(p_k,\partial D_k)>0.$$
Thus, the distance from  from $p_k^n$ to $\partial D_k^n$ is bounded away from $0$ as $n\to\infty$. The same conclusion holds for $q_k^n$.

\vspace{1em}
\noindent
\textit{Step 2: Uniformization by disks and normalizations in the plane.} By the classical uniformization theorem, Theorem \ref{theorem:uniformization_classical}, for each $k\in \N$ and $n\in \N$ there exists a conformal map from $\D$ onto ${D_k^n}$. By precomposing with a M\"obius transformation, we obtain a conformal map $h_k^n$ from a disk $B(0,r_k^n)$ with $r_k^n>1$ onto ${D_k^n}$ such that $h_k^n(0)=p_k^n$ and $h_k^n(1)=q_k^n$. Note that $h_k^n$ is $1$-quasiconformal in $B(0,r_k^n)$ by Lemma \ref{lemma:conformal_compatible}. By Theorem \ref{theorem:caratheodory}, the map $h_k^n$ extends to a weakly $1$-quasiconformal map from $\br B(0,r_k^n)$ onto $\br{D_k^n}$. 

We claim that for each fixed $k\in \N$ the sequence $r_k^n$, $n\in \N$, is bounded from above. Let $E$ be the unit interval {$[0,1]\times \{0\}$} inside $B(0,r_k^n)$ and $F_n=\partial B(0,r_k^n)$. Consider the continua $h_k^n(E)$ and $h_{k}^n(F_n)=\partial D_k^n$. From Lemma \ref{lemma:boundary_convergence}, $\partial D_k^n$ has diameter uniformly bounded below away from $0$ as $n\to\infty$. Since $p_k^n,q_k^n\in h_k^n(E)$, that set also has diameter uniformly bounded away from $0$. From Lemma \ref{lemma:modulus_bound_convergence}, we conclude that $\Mod \Gamma^*( h_k^n(E), h_k^n(F_n); \br{D_k^n})$ is uniformly bounded above as $n\to\infty$. Since $h_k^n$ is weakly $1$-quasiconformal, it follows that $\Mod \Gamma^*( E,F_n; \br B(0,r_k^n))$ is uniformly bounded above. On the other hand, the family $\Gamma^*( E,F_n; \br B(0,r_k^n))$ contains the circles $\partial B(0,r)$ for all $1<r<r_k^n$, so
$$\frac{1}{2\pi} \log \left(r_k^n\right)\leq  \Mod \Gamma^*( E,F_n; \br B(0,r_k^n)).$$
The boundedness of $r_k^n$ follows.

For fixed $k\in \N$, consider the sequence $g_n(z)= h_k^n(r_k^nz)$, $n\in \N$, from $\br{\mathbb D}$ onto $\br{D_k^n}$. We show that this sequence is normalized in the sense of Theorem \ref{theorem:uniformization_long}. Note that the points $0$, $1/r_k^n$, and $-1$ in $\br{\mathbb D}$ have mutual distances uniformly bounded away from $0$ as $n\to \infty$. Moreover, we have $g_n(0)= p_k^n$, $g_n(1/r_k^n)=q_k^n$, and $g_n(-1)\in \partial D_k^n$, and by Step 1 the mutual distances of these points are also bounded away from $0$. Thus, the sequence $g_n$ is normalized, as claimed. 

\vspace{1em}

\noindent
\textit{Step 3: Weakly quasiconformal parametrizations.} By Theorem \ref{theorem:uniformization_long}, for each $k\in \N$, there exists a subsequence of $f_n\circ g_n$, $n\in \N$, that converges uniformly on $\br{\mathbb D}$ to a weakly $K$-quasiconformal map onto $\overline{D_k}$. Since $r_k^n$ is bounded above and below in $n\in \N$, we conclude that there exists a subsequential limit $r_k$ {of} $r_k^n$ such that the sequence $f_n\circ h_k^n$ has a subsequence that converges to a weakly $K$-quasiconformal map $h_k\colon \br B(0,r_k) \to \br{D_k}$.   By passing to a diagonal subsequence, we assume that $r_k^n$ converges to $r_k$ and $f_n\circ h_k^n$ converges to $h_k$ for each $k\in \N$ (in an appropriate sense, since the domains are variable). Note that for each compact set $A\subset B(0,r_k)$ we have $A\subset B(0,r_k^n)$ for all sufficiently large $n\in \N$. Hence, $f_n\circ h_k^n$ converges to $h_k$ uniformly on each compact subset of $B(0,r_k)$.

\vspace{1em}

\noindent
\textit{Step 4: Compatibility of parametrizations.} Suppose that $D_k\cap D_l\neq \emptyset$. Our goal is to show that $h_l^{-1}\circ h_k$ defines a conformal map from $U_{k,l}=h_k^{-1}(D_k\cap D_l)$ onto $U_{l,k}=h_l^{-1}(D_k\cap D_l)$; this composition has to be interpreted appropriately since $h_l$ might not be invertible. 

The monotonicity of $h_k$ implies that $U_{k,l}\subset B(0,r_k)$, by the last part of Theorem \ref{theorem:monotone_approximation}. Let $A'\subset D_k\cap D_l$ be a non-degenerate continuum. Then $A=h_k^{-1}(A')$ is a continuum in $U_{k,l}\subset B(0,r_k)$.  For large $n\in \N$, the set $A_n=h_k^n(A)$ is defined and $f_n(A_n)=f_n\circ h_k^n(A)$ converges in the Hausdorff sense to $h_k(A)=A'$ as $n\to\infty$ by the uniform convergence. Thus, $A_n$ is contained in $D_k^n\cap D_l^n$ for all sufficiently large $n\in \N$ and the diameter of $A_n$ is  uniformly bounded from below away from $0$. Now, the uniform convergence of $f_n\circ h_l^n$ to $h_l$ implies that $(h_l^n)^{-1}(A_n)$ has diameter that is uniformly bounded from below.

Now, let $V\supset A$ be a connected open set whose closure is compact and is contained in $h_k^{-1}(D_k\cap D_l)$. By the previous argument, $h_k^n(V)\subset D_k^n\cap D_l^n$ for all sufficiently large $n\in \N$. Thus, for all sufficiently large $n\in \N$ we have
$$f_n \circ h_l^n  \circ (h_l^n)^{-1}\circ h_k^n =f_n\circ h_k^n$$
on $V$. Consider the sequence of conformal embeddings $(h_l^n)^{-1}\circ h_k^n\colon V\to B(0,r_l^n)$, $n\in \N$. Note that the balls $B(0,r_l^n)$ are uniformly bounded in $n\in \N$. By Montel's theorem \cite[Thm.\ 10.7, p.~160]{Mar:19}, as $n\to\infty$ this sequence {subconverges}  locally uniformly in $V$ to a holomorphic map ${\varphi_{k,l}}$ on $V$. By considering larger and larger open sets $V$, we may assume that the convergence is locally uniform in the component of $h_k^{-1}(D_k\cap D_l)$ that contains the continuum $A$. Moreover, by Hurwitz's theorem \cite[Cor.\ 8.9, p.~129]{Mar:19}, the limiting map $\varphi_{k,l}$ is either constant or conformal. Since $(h_l^n)^{-1}\circ h_k^n(A)$ has diameter uniformly bounded from below, we conclude that $\varphi_{k,l}$ is conformal. Arguing in the same way, we may obtain a limiting map $\varphi_{k,l}$ that is conformal in each component of $U_{k,l}=h_k^{-1}(D_k\cap D_l)$.

By {passing to} a diagonal subsequence, we may assume that $(h_l^n)^{-1}\circ h_k^n$ converges locally uniformly in $U_{k,l}$ to $\varphi_{k,l}$ for each $k,l\in \N$ with $D_k\cap D_l\neq \emptyset$ and
$$h_l\circ \varphi_{k,l} =h_k.$$
This also shows that $\varphi_{k,l}$ is injective an all of $U_{k,l}$ and $\varphi_{k,l}(U_{k,l})=h_l^{-1}(D_k\cap D_l)=U_{l,k}$.  Moreover, it is immediate that $\varphi_{l,k}=\varphi_{k,l}^{-1}$, since $(h_k^n)^{-1}\circ h_l^n$ is the inverse of $(h_l^n)^{-1}\circ h_k^n$. Finally, if $U_{k,l}\cap U_{k,m}=h_k^{-1}(D_k\cap D_l\cap D_m)\neq \emptyset$, then $\varphi_{k,l}(U_{k,l}\cap U_{k,m})= U_{l,k}\cap U_{l,m}$ and $\varphi_{l,m}\circ \varphi_{k,l}= \varphi_{k,m}$ on $U_{k,l}\cap U_{k,m}$. When $D_k\cap D_l=\emptyset$, we also define $U_{k,l}=\emptyset$ and $\varphi_{k,l}$ to be the empty map.

\vspace{1em}
\noindent 
\textit{Step 5: Gluing and construction of a Riemann surface $Z$.} For $k\in \N$, let $U_k=h_k^{-1}(D_k)$ and note that for each $k\in \N$ there exist only finitely many $l\in \N$ with $U_{k,l}\neq \emptyset$. Consider the space $ \bigsqcup_{k\in \N} U_k$ and define an equivalence relation so that $z\sim w$ if $z\in U_{k,l}$, $w\in U_{l,k}$ and $\varphi_{k,l}(z)=w$. By definition, all other points have trivial equivalence classes. The symmetry and transitive properties follow from the compatibility properties of Step 4. 

Let $Z$ be the resulting quotient space. Then $Z$ is automatically a two-dimensional topological manifold, once we verify that it is a Hausdorff space. For this, it suffices to show that for each $k,l\in\N$ the set $\{(z,\varphi_{k,l}(z)): z\in U_{k,l}\}$ is a closed subset of $U_k\times U_l$ \cite[Prop.\ 3.57, p.~68]{Lee:11}. To see this, if $z_n\in U_{k,l}$, $n\in \N$, is a sequence converging to a point $z\in U_k$ lying in the topological boundary of $U_{k,l}$ relative to $U_k$, then, by continuity, $h_{k}(z_n)$ converges to the point $h_k(z)\in D_k\cap \partial D_l $. Thus, $h_l\circ \varphi_{k,l}(z_n)$ converges to a point of $D_k\cap \partial D_l \subset \partial {D_l}$, which implies that $\varphi_{k,l}(z_n)$ cannot converge to a point of $U_l=h_l^{-1}(D_l)$. 

The charts of $Z$ are simply the inclusions of $U_k$, $k\in \N$, into the plane and the transition maps are the conformal maps $\varphi_{k,l}$, $k,l\in \N$. Hence $Z$ is a Riemann surface.

We define a map from $\bigsqcup_{k\in \N}U_k$ onto $X$ by $z\mapsto h_k(z)$ if $z\in U_k$. The compatibility relation $h_k=h_l\circ \varphi_{k,l}$ implies that this map projects to a map $F\colon Z\to X$. From the local coordinate representation of $F$, we see that $F$ is continuous. Since for each $k\in \N$, the map $h_k\colon \br B(0,r_k) \to \br{D_k}$ is cell-like, the set $h_k^{-1}(x)$ is a cell-like subset of $U_k$ for each $x\in D_k$. Therefore, again from local coordinates we see that $F$ is cell-like. Since the cover of $X$ by $\{D_k\}_{k\in \N}$, is locally finite, we see that each compact subset of $X$ can be written as a finite union of compact sets, each contained in a set $D_k$. From this we conclude that $F$ is a proper map. Since $\bigcup_{k=1}^\infty D_k=X$, the map $F$ is also surjective. By Theorem \ref{theorem:cell_like_approximation}, $F$ can be approximated by homeomorphisms and in particular $Z$ is homeomorphic to $X$. Finally, we note that $F$ is weakly $K$-quasiconformal in local coordinates, since $h_k$ is weakly $K$-quasiconformal.

\vspace{1em}
\noindent 
\textit{Step 6: Uniformization by a Riemannian manifold.} By the classical Riemannian uniformization theorem (Theorem \ref{theorem:uniformization_riemannian}), there exists a Riemannian metric $g$ on $Z$ such that $(Z,g)$ is complete and has constant curvature and such that  $g$ is compatible with the complex structure of $Z$. Since the map $F\colon Z\to X$ is weakly $K$-quasiconformal in local coordinates, we conclude that the map $F$ from the metric space $(Z,g)$ onto $X$ is locally weakly $K$-quasiconformal. By Theorem \ref{theorem:definitions_qc}, $F$ is globally weakly $K$-quasiconformal. Finally, by Theorem \ref{thm:minimize_dilatation}, there exists a Riemannian surface $(\widetilde Z,\widetilde g)$ and a quasiconformal map $\psi\colon (\widetilde Z,\widetilde g)\to (Z,g)$ such that $F\circ \psi \colon (\widetilde Z,\widetilde g)\to X$ is weakly $(4/\pi)$-quasiconformal.
\end{proof}

\begin{proof}[Proof of Theorem \ref{thm:uniformization_global}: surfaces with boundary]
We apply Theorem \ref{thm:uniformization_global} for surfaces without boundary to $\Int(X)$ to find a Riemannian surface $W$ (together with a compatible complex structure arising from Step 5 or from passing to isothermal coordinates) and a weakly $K$-quasiconformal map $G\colon W\to  \Int(X)$, where $K=4/\pi$. In particular, $G$ is proper and cell-like. Our goal is to construct a new Riemannian surface $Z$ with boundary, homeomorphic to $X$, and a weakly $K$-quasiconformal map $F\colon Z\to X$.

We repeat the procedure of Step 4 of the previous proof. We first cover $\Int(X)$ by a locally finite collection of countably many topological open disks $D_k$, $k\in \N$, so that for each point of $\partial X$ there exist finitely many disks $D_k$ with the property that $\partial X\cap \partial D_k$ is a non-degenerate arc containing that point in its interior. We define $U_k=G^{-1}(D_k)$, $k\in \N$. By the last part of Theorem \ref{theorem:cell_like_approximation}, we note that $U_k$ is homeomorphic to $D_k$ for each $k\in \N$. By the classical uniformization theorem (Theorem \ref{theorem:uniformization_classical}), there exist conformal parametrizations $h_k\colon   \D \to  {U_k}$, $k\in \N$, which are also $1$-quasiconformal by Lemma \ref{lemma:conformal_compatible}. For each $k\in \N$ the map $g_k=G\circ h_k\colon \D \to D_k$ is weakly $K$-quasiconformal. By Remark \ref{remark:caratheodory}, it extends to a weakly $K$-quasiconformal map of the closures. 

Let $U_{k,l}=U_k\cap U_l$, $k,l\in \N$. The transition maps $\varphi_{k,l}=h_l^{-1}\circ h_k$, when defined,  are conformal and they satisfy the relation $g_l\circ \varphi_{k,l}=g_k$ and the compatibility properties of Step 4. Let $J\subset \partial X$ be an open arc contained in $\br {D_k}\cap \br{D_l}$. Let $J_k,J_l\subset \partial \D $ be its preimages under $g_k$, $g_l$, respectively. Then $\varphi_{k,l}(z)$ accumulates at $J_l$ as $z$ approaches $J_k$. By a refined version of Carath\'eodory's extension theorem \cite[Thm.\ 3.1]{Pom02}, $\varphi_{k,l}$ extends to a homeomorphism from a neighborhood of $J_k$ in $\br \D$ onto a neighborhood of $J_l$ in $\br \D$. By Schwarz reflection, $\varphi_{k,l}$ extends to a conformal map from a neighborhood of $J_k$ onto a neighborhood of $J_l$ in the plane. Moreover, the extension satisfies the compatibility properties of Step 4.

We denote by $\D_k$ the union of the unit disk $\D$ together with all open arcs $J_k$ arising as preimages under $g_k$ of open arcs $J\subset \partial X \cap \br {D_k}$. As in Step 5, we can glue together the sets $ \D_k$ with the identifications $z\sim \varphi_{k,l}(z)$ to obtain a Riemann surface $Z$ with boundary; specifically, the boundary arises from the arcs of $\D_k$, $k\in \N$, lying on the unit circle. We also obtain a map $F\colon Z\to X$, given in local coordinates by $g_k|_{\D_k}$. By construction, $F$ is surjective and $F^{-1}(\partial X)=\partial Z$. The map $F$ is continuous, proper, and cell-like, as can be seen by the local coordinate representation. Indeed, recall that $g_k$ extends to a cell-like map from $\br \D$ onto $\br{D_k}$ that respects the interiors and boundaries. Moreover,  $F|_{\partial Z}$ is cell-like. Therefore, by Theorem \ref{theorem:cell_like_approximation}, $Z$ is homeomorphic to $X$. Finally, we note that $F\colon Z\to X$ is locally weakly $K$-quasiconformal. One now continues exactly as in the previous proof to find a Riemannian metric $g$ on $Z$, so that $F\colon (Z,g)\to X$ is globally weakly $K$-quasiconformal. 
\end{proof}

\begin{proof}[Proof of Theorem \ref{thm:one-sided_qc} in the general case]
Suppose that $X$ is homeomorphic to $\widehat{\C}$, $\br\D$, or $\C$. By Theorem \ref{thm:uniformization_global}, there exists a Riemannian surface $(Z,g)$ homeomorphic to $X$ and a weakly $(4/\pi)$-quasiconformal map $F\colon (Z,g)\to X$. By passing to isothermal coordinates (Theorem \ref{theorem:isothermal}), we may find a complex structure on $Z$ that is compatible with $g$. The classical uniformization theorem (Theorem \ref{theorem:uniformization_classical}) implies that there exists a conformal map from $\widehat \C$, $\br \D$, or $\D$ or $\C$, respectively onto $Z$. By Lemma \ref{lemma:conformal_compatible}, this conformal map is also $1$-quasiconformal. Therefore, by composing it with $F$, we obtain a weakly $(4/\pi)$-quasiconformal map from $\widehat \C, \br \D$, or $\D$ or $\C$, respectively onto $Z$.
\end{proof}

\bigskip

\section{Reciprocity of metric surfaces}\label{sec:reciprocal}
In this section we prove Theorem \ref{theorem:reciprocal} together with \Cref{theorem:reciprocal:interior}. We recall some definitions from the introduction. Let $X$ be a metric surface of locally finite Hausdorff $2$-measure. A quadrilateral in $X$ is a closed Jordan region $Q$ together with a partition of $\partial Q$ into four non-overlapping edges $\zeta_1,\zeta_2,\zeta_3,\zeta_4\subset \partial Q$ in cyclic order. When we refer to a quadrilateral $Q$, it will be implicitly understood that there exists such a marking on its boundary. We define $\Gamma(Q)=\Gamma(\zeta_1,\zeta_3;Q)$  and $\Gamma^*(Q) =\Gamma(\zeta_2,\zeta_4;Q)$. The metric surface $X$ is {reciprocal} if there exist constants $\kappa,\kappa'\geq 1$ such that 
\begin{align}\label{reciprocality:12}
    \kappa^{-1}\leq \Mod \Gamma(Q) \cdot \Mod\Gamma^*(Q) \leq \kappa' \quad \textrm{for each quadrilateral $Q\subset X$}
\end{align}
and 
\begin{align}\label{reciprocality:3}
        &\lim_{r\to 0} \Mod \Gamma( \br B(a,r), X\setminus B(a,R);X )=0 \quad \textrm{for each ball $B(a,R)$.} 
\end{align}
A metric surface $X$ is upper reciprocal if there exists $\kappa'\geq 1$ such that the the right inequality in \eqref{reciprocality:12} holds for each quadrilateral.

\smallskip

\subsection{Further properties of weakly quasiconformal maps}
We first discuss some preliminary statements. Assume that $X$ is a metric surface of locally finite Hausdorff $2$-measure. Let $E\subset X$ be a Borel set and $\Gamma$ be a curve family in $X$. We say that a Borel function $\rho\colon X \to [0,\infty]$ is \textit{admissible for $\Gamma|E$} if $\rho=0$ outside the set $E$ and  $$\int_{\gamma}\rho\, ds\geq 1$$
for all $\gamma\in \Gamma$.  We define 
$$\Mod (\Gamma| E) =\inf_{\rho} \int \rho^2\, d\mathcal H^2$$ 
where the infimum is over all functions $\rho$ that are admissible for $\Gamma|E$. The next lemma gives a finer regularity property of weakly quasiconformal maps. 

\begin{lemm}\label{lemma:extension}
Let $X,Y$ be metric surfaces of locally finite Hausdorff $2$-measure and $h\colon X\to Y$ be a continuous, surjective, proper, and cell-like map. Let $A\subset Y$ be a closed set and suppose that $h|_{X\setminus h^{-1}(A)}\colon X\setminus h^{-1}(A)\to Y\setminus A$ is weakly $K$-quasiconformal for some $K\geq 1$. If $\mathcal H^1(A)=0$, then
$$\Mod \Gamma \leq  \Mod (\Gamma | X\setminus h^{-1}(A)) \leq K \Mod h(\Gamma)$$
for every curve family $\Gamma$ in $X$ and $h\colon X\to Y$ is weakly $K$-quasiconformal. In particular, this is true if the set $A$ is closed and countable.
\end{lemm}
\begin{proof}
By Theorem \ref{theorem:definitions_qc}, there exists a weak upper gradient $g$ of $h|_{X\setminus h^{-1}(A)}$  such that 
\begin{align*}
    \int_{h^{-1}(E)} g^2\, d\mathcal H^2\leq K \mathcal H^2(E)
\end{align*}
for each Borel set $E\subset Y\setminus A$. In particular, this implies that 
\begin{align}\label{lemma:extension:qc_integral}
    \int_{X\setminus h^{-1}(A)} (\rho\circ h) g^2\, d\mathcal H^2 \leq K \int_{Y\setminus A} \rho \, d\mathcal H^2
\end{align}
for each Borel function $\rho\colon X\setminus h^{-1}(A) \to [0,\infty]$. 

Note that the first inequality in the statement of the lemma is trivial, since any function that is admissible for $\Gamma | X\setminus h^{-1}(A)$ is also admissible for $\Gamma$. Thus, we only show the second inequality. Let $\Gamma$ be any curve family in $X$ and $\rho$ be an admissible function for $h(\Gamma)$ that vanishes on the set $A$; such a function exists since $\mathcal H^1(A)=0$. Since $g$ is a weak upper gradient of $h|_{X\setminus h^{-1}(A)}$, by \cite[Lemma 2.5 (i)]{NR:21}, there exists an exceptional family of curves $\Gamma_0$ in $X\setminus h^{-1}(A)$ with $\Mod\Gamma_0=\Mod(\Gamma_0| X\setminus h^{-1}(A))=0$ such that for for all locally rectifiable curves $\gamma$ in $X\setminus h^{-1}(A)$ outside $\Gamma_0$ we have
\begin{align*}
   \int_{h\circ \gamma} \rho \, ds \leq \int_{\gamma} (\rho\circ h) g \, ds. 
\end{align*}
Moreover, if $\Gamma_1$ is the family of paths in $X$ that have a non-trivial subpath in $\Gamma_0$, then $\Mod(\Gamma_1| X\setminus h^{-1}(A))=0$.

Consider a path $\gamma\colon [a,b] \to X$ lying in $\Gamma\setminus \Gamma_1$. Let $[a_i,b_i]$, $i\in I$, be the closures of the components of $\gamma^{-1}(X\setminus h^{-1}(A))$ and define $\gamma_i=\gamma|_{[a_i,b_i]}$. No subpath of $\gamma_i$ lies in $\Gamma_0$, hence
\begin{align*}
   \int_{h\circ \gamma_i} \rho \, ds \leq \int_{\gamma_i} (\rho\circ h) g \, ds 
\end{align*}
for all $i\in I$. It follows that 
\begin{align*}
   1&\leq \int_{h\circ \gamma} \rho \, ds=\int_{h\circ \gamma} \rho\chi_{Y\setminus A} \, ds = \sum_{i\in I} \int_{h\circ \gamma_i} \rho \, ds  \leq \sum_{i\in I} \int_{\gamma_i} (\rho\circ h) g \, ds\leq  \int_{\gamma} (\rho\circ h) g \, ds. 
\end{align*}
Thus, $(\rho\circ h) g$ is a Borel function that vanishes on the set $h^{-1}(A)$ and is admissible for $\Gamma\setminus \Gamma_1|X\setminus h^{-1}(A)$. It follows that
\begin{align*}
    \Mod (\Gamma|X\setminus h^{-1}(A))&=\Mod (\Gamma\setminus \Gamma_1|X\setminus h^{-1}(A))   \\
    &\leq \int_{X\setminus h^{-1}(A)} (\rho\circ h)^2 g^2 \, d\mathcal H^2 \leq K \int_Y \rho^2\, d\mathcal H^2,
\end{align*}
where the latter inequality follows from \eqref{lemma:extension:qc_integral}. Infimizing over all admissible functions $\rho$ gives the conclusion. 
\end{proof}

The next proposition allows one to upgrade weak quasiconformality to quasiconformality.

\begin{prop}[{\cite[Prop.\ 3.3]{MW:21}}]\label{proposition:meier:wenger}
Let $X$ be a metric surface of finite Hausdorff $2$-measure that is homeomorphic to a topological closed disk and let $h\colon \br \D \to X$ be a weakly quasiconformal homeomorphism. If $X$ is upper reciprocal, then $h$ is a quasiconformal homeomorphism, quantitatively.
\end{prop}

\smallskip

\subsection{Proof of Theorems \ref{theorem:reciprocal} and \ref{theorem:reciprocal:interior}} The proofs of of both theorems are given at the end of the section. We first establish several preliminary statements.

\begin{lemm}\label{lemma:boundary_homeo}
Let $X$ be a metric surface of finite Hausdorff $2$-measure that is homeomorphic to a topological closed disk and let $h\colon \br \D \to X$ be a weakly quasiconformal map. If $X$ is upper reciprocal and $h^{-1}(\partial X)=\partial \D$, then $h|_{\partial \D}$ is a homeomorphism. 
\end{lemm}
\begin{proof}
Suppose that $h\colon \br \D \to X$ is a weakly $K$-quasiconformal map for some $K\geq 1$ with $h^{-1}(\partial X)=\partial \D$. Suppose to the contrary that there exists $x_0\in \partial X$ such that  $h^{-1}(x_0)$ is a non-degenerate closed arc $E\subset \partial \D$. We let $\zeta_1$ be a non-degenerate closed arc in $\partial X\setminus \{x_0\}$. For $0<r<\dist( x_0,\zeta_1)$, let $\zeta_3(r)$ be the component of $\partial X\cap \br B(x_0,r)$ that contains $x_0$.  We let $\zeta_2(r),\zeta_4(r)$ be the closures of the remaining two arcs of $\partial X$ and define $Q(r)$ to be the quadrilateral $X$ with edges $\zeta_1,\zeta_2(r),\zeta_3(r)$, and $\zeta_4(r)$. By the weak quasiconformality of $h$, for all $r<\dist(x_0,\zeta_1)$ we have
\begin{align*}
    \Mod \Gamma(Q(r)) &=\Mod \Gamma( \zeta_1,\zeta_3(r); X) \geq \Mod\Gamma(\zeta_1,\{x_0\};X)\\&\geq K^{-1} \Mod\Gamma( h^{-1}(\zeta_1), E; \br \D).
\end{align*}
The latter quantity is positive, since $h^{-1}(\zeta_1)$ is disjoint from $E=h^{-1}(x_0)$; for example, this can be shown using the Loewner property of Euclidean space. We claim that $\Mod \Gamma^*(Q(r)) \to \infty$ as $r\to 0$, which contradicts the upper reciprocity of $X$.

Let $x_1,x_2$ be the endpoints of the arc $E$. Let 
$$\delta_0=\min\{ \dist(E, h^{-1}(\zeta_1) ), |x_1-x_2|/2\}$$
and for $0<t<\delta_0$ and $i=1,2$, let $\gamma_t^i$ be the intersection of the circle $\partial B(x_i,t)$ with $\br \D$, parametrized as a simple curve. For $\delta<\delta_0$, let $\Gamma_{\delta}$ be the family of curves arising as the concatenation of $\gamma_t^1$, $\gamma_t^2$, and a subarc of $E$, where $t\in (\delta,\delta_0)$. 

We estimate $\Mod(\Gamma_{\delta}| \br \D \setminus E)$ from below. Let $\rho$ be a function that is admissible for $\Gamma_{\delta}$ and vanishes on $E$. Then for $t\in (\delta,\delta_0)$ we have 
\begin{align*}
    1\leq \int_{\gamma_t^1}\rho\, ds +\int_{\gamma_t^2}\rho\, ds \leq \left(\int_{\gamma_t^1}\rho^2\, ds +\int_{\gamma_t^2}\rho^2\, ds\right)^{1/2} (4\pi t)^{1/2}. 
\end{align*}
Integrating over $t\in (\delta,\delta_0)$ gives
$$ \int \rho^2 \, d\mathcal H^2 \geq \frac{1}{4\pi}\log\left( \frac{\delta_0}{\delta}\right).$$
Thus, $\Mod(\Gamma_\delta|\br \D\setminus E)\geq (4\pi)^{-1}\log(\delta_0/\delta)$.

For each fixed $\delta<\delta_0$, if $r$ is sufficiently small so that $h^{-1}(\zeta_3(r))\subset N_{\delta}(E)$, then each path of $h(\Gamma_\delta)$ connects $\zeta_2(r)$ and $\zeta_4(r)$. Therefore, by Lemma \ref{lemma:extension}
\begin{align*}
    \Mod \Gamma^*(Q(r))&= \Mod \Gamma(\zeta_2(r),\zeta_4(r) ;X) \geq \Mod h(\Gamma_{\delta})  \\
    &\geq K \Mod(\Gamma_\delta|\br \D\setminus E)\geq K(4\pi)^{-1}\log(\delta_0/\delta).
\end{align*}
We first let $r\to 0$ and then $\delta\to 0$ to obtain that $\Mod\Gamma^*(Q(r))\to \infty$.
\end{proof}

\begin{lemm}\label{lemma:topological}
Let $X$ be a metric surface of finite Hausdorff $2$-measure that is homeomorphic to a topological closed disk and let  $h\colon \br{\mathbb D}\to X$ be a weakly $K$-quasiconformal map for some $K\geq 1$. Let $x_0\in \Int(X)$ be a point such that $h^{-1}(x_0)$ is a non-degenerate continuum. For each $\varepsilon>0$ there exists a Jordan region $V\subset B(x_0,\varepsilon)$ containing $x_0$ and a weakly $K$-quasiconformal map $f\colon \br \D \to \br V$ such that $f^{-1}(\partial V)=\partial \D$ and $f^{-1}(x_0)= \br B(0,1-\delta)$ for some $\delta\in (0,1)$ that depends on $\varepsilon$. Moreover, $\delta\to 0$ as $\varepsilon\to 0$.  
\end{lemm}
\begin{proof}
Consider the distance function $\psi(x)=d(x,x_0)$ on $\Int(X)$, which is Lipschitz continuous. By the co-area inequality (Lemma \ref{lemm:coarea}), for a.e.\ $t>0$ the level set $\psi^{-1}(t)$ is a compact set of finite Hausdorff $1$-measure. By \cite[Thm.\ 1.5]{Nt:monotone}, for a.e.\ $t>0$, each connected component of $\psi^{-1}(t)$ is homeomorphic to a point or an interval or $\mathbb S^1$. 

Consider the function $\alpha(z)=d( h(z),  x_0)= \psi\circ  h(z)$ in $\D$. Then $\alpha$ is continuous and  each weak upper gradient of $h$ lying in $L^2(\br \D)$ is also a weak upper gradient of $\alpha$. Thus, $\alpha$ lies in the classical Sobolev space $W^{1,2}(\D)$. By the co-area formula for Sobolev spaces \cite{MSZ:03,Fed:69}, a.e.\ level set of $\alpha$ has finite Hausdorff $1$-measure. Thus, as above, for a.e.\ $t>0$, each connected component of $\alpha^{-1}(t)$ is homeomorphic to a point or an interval or $\mathbb S^1$. 

Let $\varepsilon>0$ and we fix $0<t<\min \{\dist( x_0, \partial X), \varepsilon\}$ such that each component of $\psi^{-1}(t)$ and of $\alpha^{-1}(t)$ is homeomorphic to a point or an interval or $\mathbb S^1$. Note that $\psi^{-1}(t)$ is a compact subset of $\Int(X)$. By the last part of Theorem \ref{theorem:monotone_approximation} we see that $\alpha^{-1}(t)= h^{-1}(\psi^{-1}(t))$ is a compact subset of $\D$.  One component of $\psi^{-1}(t)$, say $J$, separates $x_0$ and $\partial X$, so it is homeomorphic to  $\mathbb S^1$. By continuity, the set $h^{-1}(J)$ separates the sets $h^{-1}( x_0)$ and $ h^{-1}(\partial \D)$. Moreover, by monotonicity, $ h^{-1}( J)$ is a continuum that is contained in a connected component of $\alpha^{-1}(t)$. The only possibility is that this connected component is homeomorphic to $\mathbb S^1$, since it separates the plane, and $ h^{-1}(J)$ is equal to that component. 

Let $V$ be the Jordan region bounded by $J$ and $U$ be the Jordan region bounded by $h^{-1}(J)$. By continuity, the set $h(U)$ is connected and contained in either $V$ or $X\setminus \br V$. Since $U\cap h^{-1}(x_0)\neq \emptyset$, we conclude that $h(U)\subset V$. Similarly, $h(\br \D\setminus \br U)$ is connected and contained in either $V$ or $X\setminus \br V$. The fact that $h$ is surjective implies that $h(U)=V$ and $h(\br \D\setminus \br U)= X\setminus \br V$. In particular, $h^{-1}(V)=U$. Consider the map $h\colon  U\setminus h^{-1}(x_0) \to V\setminus \{x_0\}$ and note that it satisfies the assumptions of Theorem \ref{theorem:cell_like_approximation}. Therefore, $U\setminus h^{-1}(x_0)$ is homeomorphic to $V\setminus \{x_0\}$, and in particular, it is a topological annulus.

Consider a conformal map $\varphi\colon A(0;1-\delta,1) \to U\setminus h^{-1}(x_0)$ for some $\delta\in (0,1)$ such that  $\varphi(z)\to \partial U$ as $z\to \partial \D$; here $A(0;1-\delta,1)$ denotes the annulus $\{z\in \C: 1-\delta<|z|<1\}$. Observe that as $\varepsilon\to 0$, the region $V$ approaches $x_0$, so the region $U$ approaches $h^{-1}(x_0)$ in the Hausdorff sense. This implies that $\Mod\Gamma(\partial U, h^{-1}(x_0);\D)$ approaches $\infty$ (e.g.\ from the Loewner property of Euclidean space), thus by conformality, the quantity
$$ \Mod(\partial \D, \partial B(0,1-\delta);\C) = 2\pi (\log(1-\delta)^{-1})^{-1}$$
approaches $\infty$. Therefore, $\delta\to 0$. 

Since $\partial U$ is a Jordan curve, by a refined version of Carath\'eodory's extension theorem \cite[Thm.\ 3.1]{Pom02}, $\varphi$ extends to a homeomorphism from $\br \D \setminus B(0,1-\delta)$ onto $\br U \setminus h^{-1}(x_0)$. Consider the map $f=h\circ \varphi$, which extends to a continuous, surjective, and monotone map from $\br \D$ onto $\br V$ such that $f^{-1}(\partial V)= \partial \D$ and $f^{-1}(x_0)=\br B(0,1-\delta)$. Moreover,  $f$ is weakly $K$-quasiconformal in $A(0;1-\delta,1)=\D \setminus  f^{-1}(x_0)$. From Lemma \ref{lemma:extension} we see that $f$ is weakly $K$-quasiconformal in $\D$. Finally, by Corollary \ref{cor:extension}, $f$ is weakly $K$-quasiconformal on $\br \D$. 
\end{proof}

\begin{lemm}\label{theorem:interior}
Let $X$ be a metric surface of finite Hausdorff $2$-measure that is homeomorphic to a topological closed disk and let $h\colon \br{\mathbb D}\to X$ be a weakly quasiconformal map. If $\Int(X)$ is upper reciprocal, then $h|_{h^{-1}(\Int (X))}\colon h^{-1}(\Int(X))\to \Int(X) $ is a homeomorphism.
\end{lemm}

\begin{proof}
Let $h\colon \br \D\to X$ be a weakly $K$-quasiconformal map for some $K\geq 1$. By the invariance of domain theorem, it suffices to show that $h^{-1}(x_0)$ is a singleton for each $x_0\in \Int(X)$. Suppose that $h^{-1}(x_0)$ is a non-degenerate continuum for some $x_0\in \Int(X)$. By Lemma \ref{lemma:topological}, for a sequence of arbitrarily small $\delta>0$ there exists a Jordan region $V_{\delta} \subset \br{V_{\delta}}\subset \Int(X)$ containing $x_0$ and a weakly $K$-quasiconformal map $f_{\delta}\colon \br \D \to \br {V_{\delta}}$ such that $f_{\delta}^{-1}(\partial V_{\delta})=\partial \D$ and $f^{-1}(x_0)= \br B(0,1-\delta)$. By Lemma \ref{lemma:boundary_homeo}, since $\br {V_{\delta}}$ is upper reciprocal, $f|_{\partial \D}: \partial \D\to \partial V_{\delta}$ is a homeomorphism. 

Consider a partition of $\partial \D$ into four non-overlapping arcs of equal length that form a quadrilateral $Q$ with edges $\zeta_1,\dots,\zeta_4$. We claim that
\begin{align}\label{theorem:interior:quadrilateral}
    \Mod(\Gamma(Q)| \br \D \setminus \br B(0,1-\delta))= \Mod(\Gamma^*(Q)| \br \D \setminus \br B(0,1-\delta)) \to \infty
\end{align}
as $\delta\to 0$. The equality follows immediately by symmetry. For the limiting statement, let $\rho$ be an admissible function for $\Gamma(Q)$ that vanishes on $\br B(0,1-\delta)$. Without loss of generality, the edge $\zeta_1$ (resp.\ $\zeta_3$) of $Q$ lies in the right (resp.\ left) half-plane and is symmetric with respect to the $x$-axis. For $t\in [-a,a]$, where $a=\sqrt{2}/2$, the path $\gamma_t$ that is the intersection of the horizontal line $y=t$ with the disk $\br \D$ joins $\zeta_1$ and $\zeta_3$. Thus, 
$$\int_{\gamma_t} \rho\, ds \geq 1$$
for all $t\in [-a,a]$. Integrating, gives
\begin{align*}
    2a\leq \int_{\br \D \setminus \br B(0,1-\delta)} \rho \, d\mathcal H^2\leq \left( \int \rho^2\, d\mathcal H^2 \right)^{1/2} ( \pi - \pi(1-\delta)^2)^{1/2}.
\end{align*}
Therefore,
\begin{align*}
    \Mod(\Gamma(Q)| \br \D\setminus \br B(0,1-\delta)) \geq 2\pi^{-1}\delta^{-1}(2-\delta)^{-1}.
\end{align*}
This shows the claim.

Since $f_{\delta}$ is a homeomorphism on $\partial \D$, the image of $Q$ determines a quadrilateral $Q(\delta)$ in $X$ with $\Gamma(Q(\delta)) \supset f_{\delta}(\Gamma(Q))$. By the weak quasiconformality of $f_{\delta}$ and Lemma \ref{lemma:extension}, we have
\begin{align*}
    \Mod \Gamma(Q(\delta)) \geq K\Mod(\Gamma(Q)| \br \D \setminus \br B(0,1-\delta)) 
\end{align*}
which converges to $\infty$ as $\delta\to 0$ by \eqref{theorem:interior:quadrilateral}. Similarly $\Mod \Gamma^*(Q(\delta)) \to \infty$ as $\delta\to 0$. This contradicts the upper reciprocity of $\Int(X)$.
\end{proof}

\begin{cor}\label{corollary:interior}
Let $X$ be a metric surface of finite Hausdorff $2$-measure that is homeomorphic to a topological closed disk such that $\Int(X)$ is upper reciprocal. Then there exists a weakly quasiconformal map $h\colon \br \D \to X$ such that $h|_{\D}:\D \to \Int(X)$ is a quasiconformal homeomorphism.  
\end{cor}
\begin{proof}By Theorem \ref{thm:one-sided_qc} there exists a weakly quasiconformal map $h\colon \br \D \to X$. By Lemma \ref{theorem:interior}, $h|_{h^{-1}(\Int(X))}$ is a homeomorphism onto $\Int(X)$. In particular, $h^{-1}(\Int(X))$ is simply connected and by the Riemann mapping theorem there exists a conformal map $\varphi\colon \D \to h^{-1}(\Int(X))$. The composition $f=h\circ \varphi \colon \D\to \Int(X)$ is a weakly quasiconformal homeomorphism. By Proposition \ref{proposition:meier:wenger}, since $\Int(X)$ is upper reciprocal, $f$ is $K$-quasiconformal in each compact disk $\br B(0,r)$, $0<r<1$, for some $K\geq 1$. By the locality of quasiconformality, as given by Theorem \ref{theorem:definitions_qc}, we see that $f\colon \D \to \Int(X)$ is a $K$-quasiconformal homeomorphism. By Theorem \ref{theorem:caratheodory}, $f$ extends to a weakly $K$-quasiconformal map from $\br \D $ onto $X$.
\end{proof}

We are finally ready for the proofs of \Cref{theorem:reciprocal} and \Cref{theorem:reciprocal:interior}.

\begin{proof}[Proof of Theorem \ref{theorem:reciprocal}]
Suppose that $X$ is upper reciprocal. It suffices to show that each closed Jordan region $Y\subset X$ is quasiconformally equivalent to $\br \D$. By Corollary \ref{corollary:interior}, there exists a weakly quasiconformal map $h\colon \br \D \to Y$ such that $h|_{\D}\colon \D \to \Int(Y)$ is a homeomorphism. By Lemma \ref{lemma:boundary_homeo}, $h\colon \br \D \to Y$ is a homeomorphism. By Proposition \ref{proposition:meier:wenger}, $h$ is quasiconformal, as desired.
\end{proof}

\begin{proof}[Proof of Theorem \ref{theorem:reciprocal:interior}]
Suppose that $\Int(X)$ is upper reciprocal and \eqref{reciprocality:3} holds at each point of $\partial X$. By Theorem \ref{theorem:reciprocal}, $\Int(X)$ is reciprocal and  \eqref{reciprocality:3} holds at each point of $X$. As in the above proof, it suffices to show that each closed Jordan region $Y\subset X$ is quasiconformally equivalent to $\br \D$. By Corollary \ref{corollary:interior}, there exists a weakly quasiconformal map $h\colon \br \D \to Y$ such that $h|_{\D}\colon \D \to \Int(Y)$ is a quasiconformal homeomorphism; alternatively, since $\Int(Y)$ is reciprocal, one can invoke directly Rajala's uniformization theorem, Theorem \ref{thm:rajala}, in combination with Theorem \ref{theorem:caratheodory} to obtain the map $h$. A result of Ikonen \cite[Prop.\ 2.1]{Iko:21} now asserts that if $h|_{\D}\colon \D \to \Int(Y)$ is a quasiconformal homeomorphism and  \eqref{reciprocality:3} holds at each point of $\partial Y$, then $h$ extends to a quasiconformal homeomorphism from $\br \D$ onto $Y$, as desired. 
\end{proof}

\bigskip

\section{Examples}\label{sec:examples}

\begin{example}[Approximating the unit disk with cracked surfaces]\label{example:crack}
Let $X$ be the open unit disk equipped with the Euclidean metric. We present examples of sequences of metric surfaces $\{X_n\}_{n\in \N}$ converging to $X$ for which all the conclusions of Theorem \ref{thm:extended_polyhedral_approximation} are satisfied except for the approximately isometric retractions in \ref{item:main_3}. In each case, the failure of condition \ref{item:main_3} prevents the sequence of uniformizing conformal maps from $\mathbb{D}$ onto $X_n$ from converging to a conformal map onto $X$. 

As the first example, we let $X_n=\D \setminus \{ re^{i\theta}: |\theta|\leq \frac{\pi}{2n}, 1/2\leq r\leq 1\}$ for all $n\in \N$ and equip $X_n$ with the ambient Euclidean metric. Clearly \ref{item:main_3} fails for the sequence $X_n$, $n\in \N$, while the other conclusions of Theorem \ref{thm:extended_polyhedral_approximation} remain true. For this reason, the sequence of conformal maps $h_n\colon  \D \to X_n$, normalized so that $h_n(0)=0$ and $h_n'(0)>0$, does not {sub}converge to a (weakly) conformal map from $ \D$ onto $X$, as it would in the presence of condition \ref{item:main_3}, but rather  to a conformal map from $\D$ onto a slit disk (e.g.\ by Carath\'eodory's kernel convergence theorem \cite[Chapter I, Theorem 1.8]{Pom:92}). 

For our second example, we let $S(\varphi,n)=\{re^{i\theta}: |\theta -\varphi|\leq \frac{\pi}{2n} , 1/2\leq r\leq 1\}$ and $Y_n = \bigcup_{k=1}^n S(2k\pi/n,n)$. Define $X_n$ to be the surface $\mathbb{D} \setminus Y_n$, again equipped with the ambient Euclidean metric.  Although $X_n$ converges to $\D$ in the Gromov--Hausdorff sense, the sequence of conformal maps $h_n \colon \D \to X_n$ such that $h_n(0)=0$ subconverges locally uniformly in $\D$ but not uniformly to a conformal map from $\D$ onto $B(0,1/2)$; this can be justified using Carath\'eodory's kernel convergence theorem.

By enlarging the cracks in the previous example, one can even have that the areas of $X_n$ tend to $0$ and the conformal maps $h_n$ converge to a constant map. 
\end{example}

\smallskip

\begin{example}[Areas need not converge]\label{example:area}
We present an example showing that in Theorem \ref{thm:extended_polyhedral_approximation}  the areas of the surfaces $X_n$ do not necessarily converge to the area of $X$. Thus, we cannot expect that conclusion \ref{item:main_2} can be strengthened to convergence or that the constant $K$ can be equal to $1$.

The proof  relies on the next theorem, which is a variant of the classical Besicovitch inequality and is proved exactly in the same way; see  \cite[Thm.\ 13.11]{Pet:20} or \cite[Sect.\ 5.6]{BBI:01}. 

\begin{thm}\label{thm:besicovitch}
Let $Y$ be a metric space homeomorphic to the unit square $[0,1]^2$ such that the metric of $Y$ outside finitely many points is locally isometric to a Riemannian metric. If  the distances between opposite sides of $Y$ are at least $1$, then $\mathcal H^2(Y)\geq 1$.
\end{thm}

Let $X$ be the closed unit square in $\R^2$ with the $\ell^\infty$ metric, and let $\{X_n\}_{n=1}^\infty$ be a sequence of polyhedral surfaces homeomorphic to $X$ that are equipped with a metric that is locally isometric to the polyhedral length metric. Let $\varepsilon_n>0$ be a sequence converging to $0$. Suppose that there exist $\varepsilon_n$-isometric sequences $f_n\colon X_n\to X$ and $R_n\colon X\to f_n(X_n)$ such that each $f_n$ is a topological embedding and each $R_n$ is a retraction. We claim that
$$ \liminf_{n\to\infty} \mathcal H^2(X_n) \geq 1 = \frac{4}{\pi} \mathcal H^2(X).$$
See \cite[Lemma 6]{Kir:94} or \cite[pp.~2--3]{EBC:21} for the equality on the right-hand side.

For each large $n\in \N$, the retraction $R_n$ induces a subdivision of $f_n(\partial X_n)$, and thus of $\partial X_n$, into four non-overlapping arcs, corresponding to the sides of the square $X$. The distances of opposite sides of $f_n(\partial X_n)$ is at least $1-\varepsilon_n$. Thus, the distance of opposite sides of $X_n$ is at least $1-2\varepsilon_n$. The metric of $X_n$ is locally isometric to a flat Riemannian metric outside the vertices. Therefore, if we rescale the metric of $X_n$ by $(1-2\varepsilon_n)^{-1}$ and apply Theorem \ref{thm:besicovitch}, we obtain $$\mathcal H^2(X_n) \geq (1-2\varepsilon_n)^2.$$
Finally, we let $n\to\infty$ to obtain the desired conclusion.
\end{example}

\smallskip

\begin{example}[Condition \eqref{ireciprocality:3} is strictly weaker than reciprocity]\label{example:rec3}
We show that there exists a metric surface $X$ of locally finite Hausdorff $2$-measure that satisfies \eqref{ireciprocality:3} but is not reciprocal. Moreover, the constructed surface can be written as the union of two reciprocal surfaces, as claimed in Proposition \ref{prop:example:rec3}. To see the later, we note that by construction the surface $X$ is homeomorphic to $\C$ and is locally isometric to a planar domain outside a Cantor set $E$. Consider a Jordan curve $J$ in $X$ such that $E\subset J$; the existence of $J$ follows from the tameness of planar Cantor sets \cite[Sect.\ 13, Thm.\ 7, p.~93]{Moi:77}. Then $J$ separates $X$ into two surfaces $X_1,X_2$ such that $J$ is their common boundary. The interior of each of these surfaces is locally isometric to a planar domain so it is reciprocal. By Theorem \ref{theorem:reciprocal:interior}, both $X_1$ and $X_2$ are reciprocal.

Ahlfors--Beurling \cite[Thm.\ 16]{AB:50} construct a Cantor set $E\subset \C$ of positive area with the property that for each point $a\in E$ there exists a sequence of nested topological annuli $A_n\subset \C\setminus E$, $n\in \N$, with disjoint closures, each surrounding $a$, with the property that 
\begin{align}\label{example:reciprocal:divergence}
    \sum_{n=1}^\infty \frac{1}{M_n} =\infty,
\end{align}
where $M_n$ is the modulus of the family of curves  joining the boundary components of $A_n$ in $\br {A_n}$. Moreover, by the construction, the annuli $A_n$ converge to the point $a$; this is also implied by the divergence of the sum above.

For $x,y\in \C$, we let \[d(x,y)= \inf_{\gamma}\int_{\gamma} \chi_{\C\setminus E}\, ds,\] where the infimum is taken over all rectifiable paths joining $x,y$. By \cite[Prop.\ 8.1]{NR:21}, $(\C,d)$ is a length space with locally finite Hausdorff $2$-measure, homeomorphic to $\C$. We denote by $X$ this metric space. Since $E$ has positive Lebesgue measure, $X$ is not reciprocal; this can be shown by following the argument of \cite[Example 2.1]{Raj:17}.

We will show that for each point of $X$, condition \eqref{ireciprocality:3} is true. This is trivially true for points of $X\setminus E$, since the identity map from $(\C\setminus E,|\cdot |)$ to $(X\setminus E,d)$ is locally isometric and thus preserves modulus. 

Next, fix a point $a\in E$ and consider the corresponding annuli $A_n$, $n\in \N$, as above. We fix $R>0$. Since the annuli $A_n$ converge to $a$, by discarding finitely many annuli, we may assume that $A_n\subset B_d(a,R)$ for all $n\in \N$. Next, we fix $N\in \N$. Then, for all sufficiently small $r>0$ the closed ball $\br{B}_d(a,r)$ is disjoint from $A_N$ and is surrounded by $A_N$.

For $i\in \{1,\dots,N\}$, let $\rho_i$ be an admissible function for the modulus of curves joining the boundary components of $A_i$, such that 
\begin{align}\label{example:reciprocal:modulus}
    \int \rho_i^2 \, d\mathcal H^2_{|\cdot |}\leq 2M_i.
\end{align}
We may take $\rho_i$ to be supported in $\br{A_i}$, so the functions $\rho_i$, $i\in \{1,\dots,N\}$, have disjoint supports. Now, we define 
$$\rho= \sum_{i=1}^N c_i\rho_i,$$
where $c_i= M_i^{-1} (\sum_{j=1}^N M_j^{-1})^{-1}$. 

Each curve $\gamma$ joining  $X\setminus B_d(a,R)$ to $\br {B}_d(a,r)$ has disjoint subpaths $\gamma_i$, $i\in \{1,\dots,N\}$, such that $|\gamma_i|\subset \br {A_i}$ and $\gamma_i$ joins the boundary components of $A_i$. Therefore,
$$\int_{\gamma}\rho \, ds \geq \sum_{i=1}^N c_i \int_{\gamma_i} \rho_i\, ds \geq \sum_{i=1}^N c_i =1.$$
This shows that $\rho$ is admissible for $\Gamma( \br {B}_d(a,r), X\setminus B_d(a,R);X )$. Therefore, by \eqref{example:reciprocal:modulus}, we have
\begin{align*}
  \Mod\Gamma( \br {B}_d(a,r), X\setminus B_d(a,R);X ) &\leq \int \rho^2\, d\mathcal H^2_d= \sum_{i=1}^N c_i^2 \int \rho_i^2\, d\mathcal H^2_{|\cdot|}  \\
  &\leq 2\sum_{i=1}^N c_i^2 M_i=2 \left(\sum_{j=1}^N \frac{1}{M_j}\right)^{-1}.
\end{align*}
The latter converges to $0$ as $N\to\infty$ by \eqref{example:reciprocal:divergence}. Thus, letting first $r\to 0$ and then $N\to \infty$ shows that \eqref{ireciprocality:3} holds at the point $a$.
\end{example}

\smallskip

\begin{example}[Condition \eqref{ireciprocality:3} can fail on a continuum]\label{example:continuum}
We show that there exists a metric surface $X$ of locally finite Hausdorff $2$-measure that is homeomorphic to $\C$ such that \eqref{ireciprocality:3} fails for all points in a non-degenerate continuum $E\subset X$. 

Let $C\subset [0,1]$ be a Cantor set containing the points $0,1$. For each component $J$ of $[0,1]\setminus C$ consider a positive number $\eta(J)$ and two vertical line segments of length $2\eta(J)$ passing through the endpoints of $J$ and symmetric with respect to the $x$-axis. Let $H(J)$ be the union of these two line segments with $J\times \{0\}$. We choose the numbers $\eta(J)$ so that for each $\delta>0$ there are only finitely many components $J$ of $[0,1]\setminus C$ with $\eta(J)>\delta$; later we will make a specific choice. We now define $F= (C\times \{0\})\cup \bigcup_{J} H(J)$ and observe that $\mathcal H^2_{|\cdot |}(F)=0$; see Figure \ref{fig:f_set}.

\begin{figure}
    \centering
    \begin{tikzpicture}
           \node () at (0,0) {\includegraphics[width=.5\textwidth]{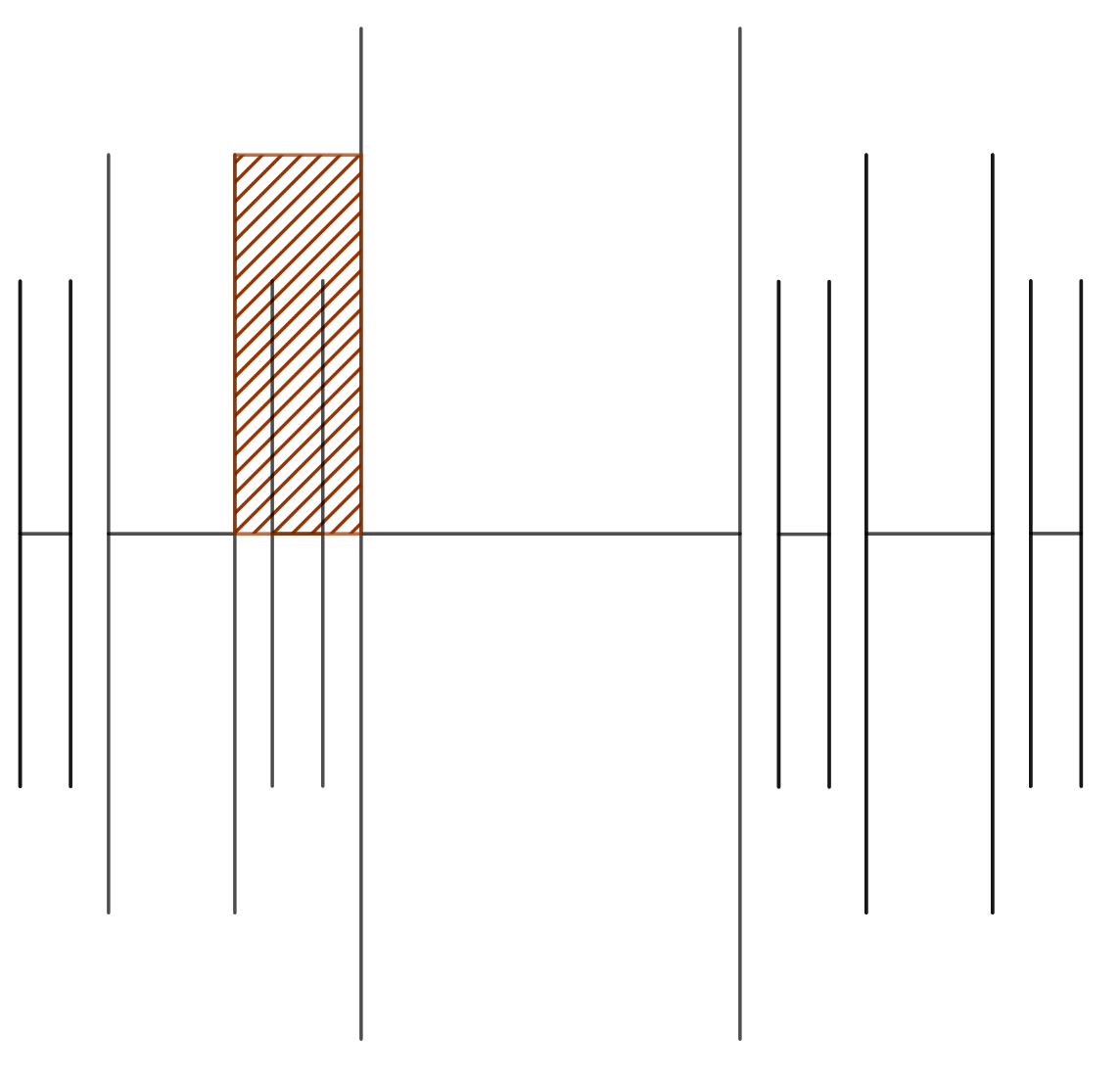}};
           \node[anchor=south] () at (0,0) {$H(J_1)$};
           \node[anchor=south] () at (-2.2,0) {$\scriptscriptstyle{H(J_2)}$};
           \node[anchor=south] () at (-1.5,2.2) {$R_1$};
    \end{tikzpicture}
    \caption{The set $F$.}
    \label{fig:f_set}
\end{figure}
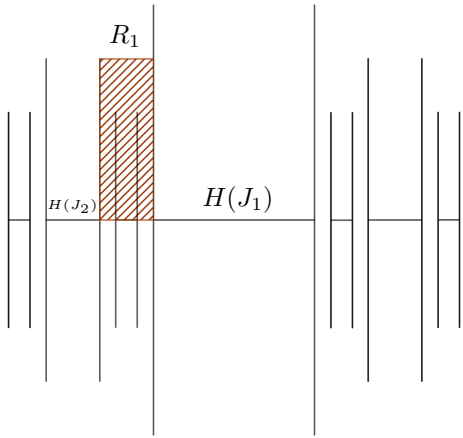

We define an equivalence relation $R$ on $\C$ that is trivial outside $\bigcup_{J}H(J)$ and collapses each $H$-shape $H(J)$ to a point. In particular, each equivalence class does not separate the plane. It is easy to see that the relation $R$ is upper semi-continuous (that is, the limit points of every two sequences of equivalent points are equivalent). If $Y=\C/R$ is the quotient space and $h_Y\colon \C \to Y$ is the natural projection, then $Y$ is homeomorphic to $\C$ by Moore's decomposition theorem \cite[p.~3]{Dav:86}, and $h_Y$ is continuous, surjective, and monotone. The set $h_Y(F)$ is a non-degenerate continuum in $Y$ and each point of $h_Y(F)$ has a preimage that is either an $H$-shape or a point of $C\times \{0\}$. We will realize this topological model with a metric. 

We require that the Cantor set $C$ has the additional property that $\mathcal H^1( C\cap I) >0$ for each open interval $I\subset \R$ intersecting $C$; later we will specify even further the Cantor set. Define the quotient pseudometric  $d_R$ on $\C$ by 
$$d_R(x,y)=\inf \left\{ \sum_{i=1}^k |p_i-q_i|: p_1=x, \,\, q_k=y, \,\, k\in \N \right\},$$
where the infimum is taken over all choices of $\{p_i\}_{i=1}^k$ and $\{q_i\}_{i=1}^k$ such that $q_i\sim_R p_{i+1}$ for $i\in \{1,\dots,k-1\}$. Note that $d_R(x,y)=0$ if $x\sim_R y$ and $d_R(x,y)>0$ if $x\notin F$ and $y\neq x$. 

We will show that the equivalence relation $d_R=0$ coincides with $R$. For this, it suffices to show that if $x,y\in F$ and $x\not\sim_R y$, then $d_R(x,y)>0$.  For a point $z\in \C$, denote by $\widetilde z$ its projection onto the $x$-axis. If $x\not\sim_R y$ and $x,y\in F$, then $\widetilde x\neq \widetilde y$. Without loss of generality, suppose that $\widetilde x<\widetilde y$. Let $\{p_i\}_{i=1}^k$ and $\{q_i\}_{i=1}^k$ be points such that $p_1=x$, $q_k=y$, and $q_i\sim_R p_{i+1}$ for $i\in \{1,\dots,k-1\}$. Note that the polygonal path joining $p_1$ to $q_1$ to $p_2$ to $q_2$ to $\dots$ to $q_k$ projects to a line segment in the $x$-axis that covers the interval $[\widetilde x,\widetilde y]$. Moreover, since $q_i\sim_R p_{i+1}$, we see that $q_i$ and $p_{i+1}$ lie in the same $H$-shape or $q_i=p_{i+1}$. Hence, the projection of the segment from $q_i$ to $p_{i+1}$ intersects $C$ in at most two points. We conclude that the intervals from $\widetilde p_i$ to $\widetilde q_i$, $i\in \{1,\dots,k\}$, cover the set $C\cap [\widetilde x,\widetilde y]$ except for finitely many points. Therefore,
$$\sum_{i=1}^k |p_i-q_i| \geq \sum_{i=1}^k |\widetilde p_i-\widetilde q_i| \geq \mathcal H^1(C\cap [\widetilde x,\widetilde y]).$$
Note that the latter is positive, since $x\not\sim_R y$ and in particular they do not lie in the same $H$-shape. It follows that
$$d_R(x,y)\geq \mathcal H^1(C\cap [\widetilde x,\widetilde y])>0.$$
This shows the claim. In fact, the reverse inequality also holds. That is, if $x,y\in F$, then 
\begin{align*}
    d_R(x,y)=\mathcal H^1(C\cap [\widetilde x,\widetilde y]).
\end{align*}
This can be seen by connecting the projections $\widetilde x$ and $\widetilde y$ with a specific polygonal path contained in $[0,1]\times \{0\}$.

Consider the metric space $X= \C/d_R$ with metric $d=d_R$ and let $h\colon \C\to X$ be the projection map. Since the equivalence relation $d_R=0$ coincides with $R$, we conclude that $X$ is homeomorphic to $Y=\C/R$ and thus it is homeomorphic to $\C$; see \cite[Exercise 3.1.13, p.~63]{BBI:01}. Observe that $h$ is $1$-Lipschitz continuous. Thus, $X$ has locally finite Hausdorff $2$-measure. Moreover, $h$ is surjective, proper, and cell-like.  We claim that $h$ is weakly $1$-quasiconformal. For this, by Theorem \ref{theorem:definitions_qc} it suffices to see that the function $g=1$ is an upper gradient of $h$ because $h$ is $1$-Lipschitz, and has the property that 
\begin{align*}
    \int_{h^{-1}(A)} g^2\, d\mathcal H^2_{|\cdot |} \leq \mathcal H^2_{d}(A)  
\end{align*}
for each Borel set $A\subset X$. {The latter, in fact, holds trivially with equality, because  $h$ is locally isometric outside $F$ and  $\mathcal H^2_{|\cdot|}(F)=\mathcal H^2_{d}(h(F))=0$} by the $1$-Lipschitz property of $h$. 

Finally, we verify that \eqref{ireciprocality:3} fails at each point of the set $E=h(F)$. Since $h$ is weakly quasiconformal, this is immediate for points whose preimage is a $H$-shape, and in particular has positive diameter. We have to argue for the points whose preimage is a singleton in $C\times \{0\}$. To achieve this, we will give a specific construction of the Cantor set $C$ and a specific choice of the numbers $\eta(J)$.

We construct the fat Cantor set $C$ in the standard way. Namely, we call $I_0=[0,1]$ an interval of level $0$ and  once the intervals of level $n-1$ have been defined, we remove a middle interval of relative length $a_n\in (0,1)$ from each interval of level $n-1$ and obtain $2^{n}$ intervals of level $n$. The Cantor set $C$ is defined as the intersection in $n\in \N\cup \{0\}$ of the unions of all intervals of level $n$. If $\prod_{n=1}^\infty(1-a_n)>0$, or equivalently if $\sum_{n=1}^\infty a_n<\infty$, then the Cantor set $C$ has positive measure. We choose the sequence $\{a_n\}_{n\in \N}$ arbitrarily so that $C$ has positive measure. If $I_{n-1}$ is an interval of level $n-1$, then $\mathcal H^1(I_{n-1})\simeq 2^{-n+1}\simeq 2^{-n}$. If $J_{n}$ is a middle interval removed from $I_{n-1}$ at the next stage, we define $\eta(J_n)=1/n$. To simplify the notation and avoid taking product sets, we consider the real line $\R$ as a subset of the plane $\C$, identified with $\R\times \{0\}$. 

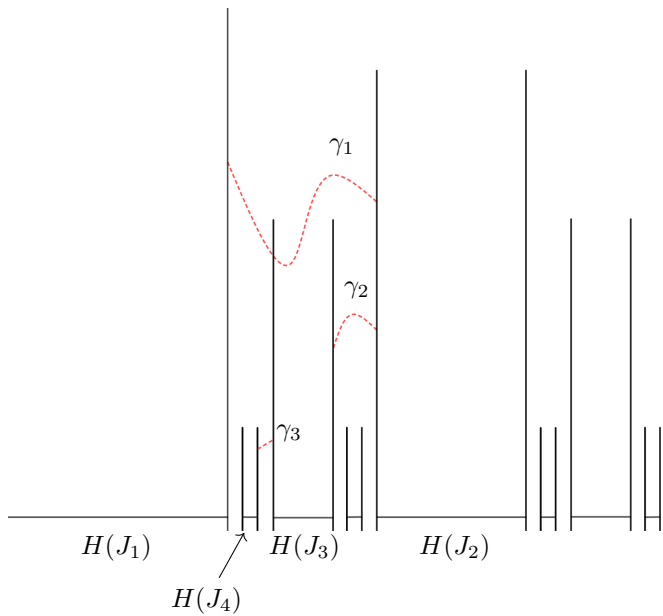
\begin{figure}
    \centering
    \begin{tikzpicture}
           \node () at (0,0) {\includegraphics[width=.7\textwidth]{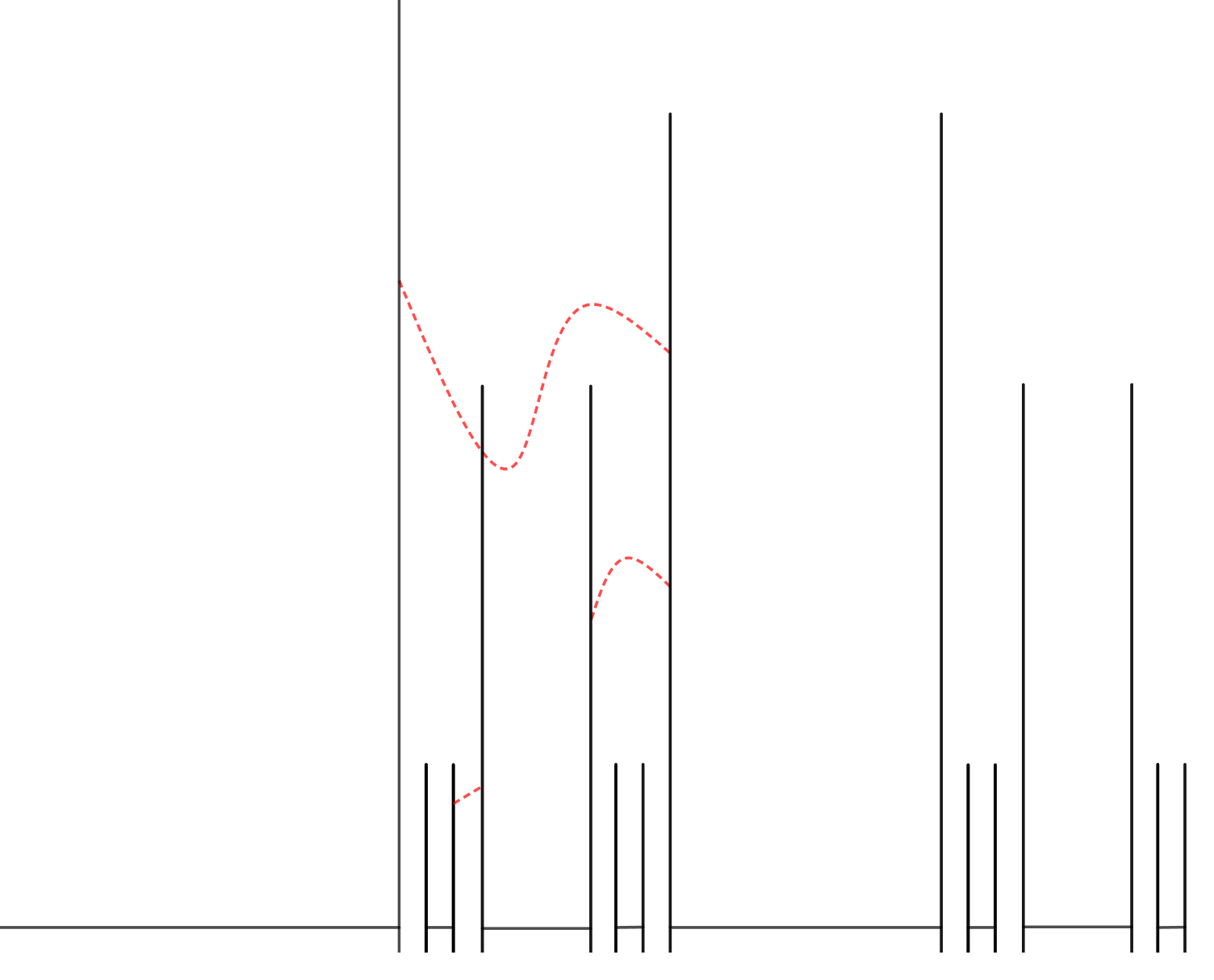}};
           \node[anchor=south] () at (-3,-3.9) {$H(J_1)$};
           \node[anchor=south] () at (1.5,-3.9) {$H(J_2)$};
           \node[anchor=south] () at (-0.5,-3.9) {${H(J_3)}$};
           \node[anchor=north] (J) at (-1.8,-4) {$H(J_4)$};
           \draw [->] (J) edge (-1.3,-3.4);
           \node[anchor=south] () at (0,1.5) {$\gamma_1$};
           \node[anchor=south] () at (0.2,-.4){$\gamma_2$};
           \node[anchor=south] () at (-0.7,-2.3) {$\gamma_3$};
    \end{tikzpicture}
    \caption{A depiction of the paths $\gamma_n$ in the plane. In fact, the paths are taken in the space $X$.}
    \label{fig:paths}
\end{figure}

Let $x_0\in C$ and for each $n\in \N\cup \{0\}$ consider the interval $I_n$ of level $n$ containing $x_0$, so that $\{x_0\}=\bigcap_{n=0}^\infty I_n$. Denote by $J_n$ the middle interval removed from $I_{n-1}$, $n\in \N$. For $n\in \N$ consider the rectangle $R_n$ in the upper half-plane with height $\eta(J_{n+1})$, whose vertical sides are contained in $H(J_n)$ and $H(J_{n+1})$, and one of whose horizontal sides coincides with a remaining interval of level $n+1$; this interval might or might not be $I_{n+1}$. See Figure \ref{fig:f_set} for an illustration.  Note that the modulus in $\C$ of the family of curves joining the vertical sides of $R_n$ is $M_n=\frac{\eta(J_{n+1})}{\mathcal H^1(I_{n+1})}\simeq n^{-1}2^n$. By the weak $1$-quasiconformality of $h$, the modulus of curves in $X$ joining the points $h(H(J_n))$ and $h(H(J_{n+1}))$ is at least $M_n$. Observe that the sequence of points $h(H(J_n))$, $n\in \N$, converges to $h(x_0)$ as $n\to\infty$.

Suppose that $x_0$ is not an endpoint of a complementary interval of the Cantor set $C$. We fix a small radius $R>0$ so that $B_d(h(x_0),R)$ does not contain the point $h(H(J_1))$. Let $r<R$ be arbitrary and fix $m\in \N$ such that $h(H(J_{m+1}))\in B_d(h(x_0),r)$. Let $\rho\in L^2(X)$ be an admissible function for $\Gamma( \br B_d( h(x_0),r), X\setminus B_d(h(x_0),R);X )$. For each $n\in \N$ we may find a path $\gamma_n$ connecting $h(H(J_n))$ to $h(H(J_{n+1}))$ with 
$$\int_{\gamma_n} \rho\, ds \leq 2 M_n^{-1/2} \|\rho\|_{L^2(X)}.$$
See Figure \ref{fig:paths} for a depiction. The concatenation of $\gamma_1,\dots,\gamma_m$ gives a rectifiable path joining the point $h(H(J_1))$ to $h(H(J_{m+1}))$ and thus lying in $\Gamma( \br B_d( h(x_0),r), X\setminus B_d(h(x_0),R);X )$. Therefore, by the admissibility of $\rho$, we have
$$1\leq \sum_{n=1}^m \int_{\gamma_n} \rho\, ds \leq  2 \sum_{n=1}^m M_n^{-1/2 } \|\rho\|_{L^2(X)} \simeq \|\rho\|_{L^2(X)} \sum_{n=1}^m \sqrt{n}2^{-n/2}\lesssim \|\rho\|_{L^2(X)}.$$
This implies that 
$$\Mod \Gamma( \br B_d( h(x_0),r), X\setminus B_d(h(x_0),R);X ) \gtrsim 1 $$
for all sufficiently small $r>0$. This completes the proof. 
\end{example}

\bigskip

\bibliographystyle{abbrv}
\bibliography{bibliography}

\begin{thebibliography}{10}

\bibitem{AB:50}
L.~Ahlfors and A.~Beurling.
\newblock Conformal invariants and function-theoretic null-sets.
\newblock {\em Acta Math.}, 83:101--129, 1950.

\bibitem{AZ:67}
A.~D. Aleksandrov and V.~A. Zalgaller.
\newblock {\em Intrinsic geometry of surfaces}, volume~15 of {\em Translations
  of Mathematical Monographs}.
\newblock American Mathematical Society, Providence, RI, 1967.
\newblock Translated from the Russian by J. M. Danskin.

\bibitem{Ball:97}
K.~Ball.
\newblock An elementary introduction to modern convex geometry.
\newblock In {\em Flavors of geometry}, volume~31 of {\em Math. Sci. Res. Inst.
  Publ.}, pages 1--58. Cambridge Univ. Press, Cambridge, 1997.

\bibitem{BK:02}
M.~Bonk and B.~Kleiner.
\newblock Quasisymmetric parametrizations of two-dimensional metric spheres.
\newblock {\em Invent. Math.}, 150(1):127--183, 2002.

\bibitem{BM:17}
M.~Bonk and D.~Meyer.
\newblock {\em Expanding {T}hurston maps}, volume 225 of {\em Mathematical
  Surveys and Monographs}.
\newblock American Mathematical Society, Providence, RI, 2017.

\bibitem{BH:99}
M.~R. Bridson and A.~Haefliger.
\newblock {\em Metric spaces of non-positive curvature}, volume 319 of {\em
  Grundlehren der Mathematischen Wissenschaften}.
\newblock Springer-Verlag, Berlin, 1999.

\bibitem{BBI:01}
D.~Burago, Y.~Burago, and S.~Ivanov.
\newblock {\em A course in metric geometry}, volume~33 of {\em Graduate Studies
  in Mathematics}.
\newblock American Mathematical Society, Providence, RI, 2001.

\bibitem{Can:94}
J.~W. Cannon.
\newblock The combinatorial {R}iemann mapping theorem.
\newblock {\em Acta Math.}, 173(2):155--234, 1994.

\bibitem{CR:21}
P.~Creutz and M.~Romney.
\newblock Triangulating metric surfaces.
\newblock {\em Proc. Lond. Math. Soc. (3)}, 2022.
\newblock To appear.

\bibitem{Dav:86}
R.~J. Daverman.
\newblock {\em Decompositions of manifolds}, volume 124 of {\em Pure and
  Applied Mathematics}.
\newblock Academic Press, Inc., Orlando, FL, 1986.

\bibitem{EBC:21}
S.~Eriksson-Bique and P.~Poggi-Corradini.
\newblock On the sharp lower bound for duality of modulus.
\newblock {\em Proc. Amer. Math. Soc.}, 150:2955--2968, 2022.

\bibitem{EH:21}
B.~Esmayli and P.~Haj{\l}asz.
\newblock The coarea inequality.
\newblock {\em Ann. Fenn. Math.}, 46(2):965--991, 2021.

\bibitem{Fed:69}
H.~Federer.
\newblock {\em Geometric measure theory}, volume 153 of {\em Grundlehren der
  Mathematischen Wissenschaften}.
\newblock Springer-Verlag New York Inc., New York, 1969.

\bibitem{FW:21}
M.~Fitzi and S.~Wenger.
\newblock Area minimizing surfaces of bounded genus in metric spaces.
\newblock {\em J. Reine Angew. Math.}, 770:87--112, 2021.

\bibitem{GW:20}
C.-Y. Guo and S.~Wenger.
\newblock Area minimizing discs in locally non-compact metric spaces.
\newblock {\em Comm. Anal. Geom.}, 28(1):89--112, 2020.

\bibitem{Hei:01}
J.~Heinonen.
\newblock {\em Lectures on analysis on metric spaces}.
\newblock Universitext. Springer-Verlag, New York, 2001.

\bibitem{Hei:05}
J.~Heinonen.
\newblock {\em Lectures on {L}ipschitz analysis}, volume 100 of {\em Report.
  University of Jyv\"{a}skyl\"{a} Department of Mathematics and Statistics}.
\newblock University of Jyv\"{a}skyl\"{a}, Jyv\"{a}skyl\"{a}, 2005.

\bibitem{HKST:15}
J.~Heinonen, P.~Koskela, N.~Shanmugalingam, and J.~T. Tyson.
\newblock {\em Sobolev spaces on metric measure spaces: An approach based on
  upper gradients}, volume~27 of {\em New Mathematical Monographs}.
\newblock Cambridge University Press, Cambridge, 2015.

\bibitem{Hub:06}
J.~H. Hubbard.
\newblock {\em Teichm\"{u}ller theory and applications to geometry, topology,
  and dynamics. {V}ol. 1}.
\newblock Matrix Editions, Ithaca, NY, 2006.

\bibitem{Iko:21}
T.~Ikonen.
\newblock Quasiconformal {J}ordan domains.
\newblock {\em Anal. Geom. Metr. Spaces}, 9(1):167--185, 2021.

\bibitem{Iko:19}
T.~Ikonen.
\newblock Uniformization of metric surfaces using isothermal coordinates.
\newblock {\em Ann. Fenn. Math.}, 47(1):155--180, 2022.

\bibitem{IR:20}
T.~{Ikonen} and M.~{Romney}.
\newblock {Quasiconformal geometry and removable sets for conformal mappings}.
\newblock {\em J. Anal. Math.}, 2022.
\newblock To appear.

\bibitem{Kir:94}
B.~Kirchheim.
\newblock Rectifiable metric spaces: local structure and regularity of the
  {H}ausdorff measure.
\newblock {\em Proc. Amer. Math. Soc.}, 121(1):113--123, 1994.

\bibitem{Lee:11}
J.~M. Lee.
\newblock {\em Introduction to topological manifolds}, volume 202 of {\em
  Graduate Texts in Mathematics}.
\newblock Springer, New York, second edition, 2011.

\bibitem{LV:73}
O.~Lehto and K.~I. Virtanen.
\newblock {\em Quasiconformal mappings in the plane}, volume 126 of {\em
  Grundlehren der Mathematischen Wissenschaften}.
\newblock Springer-Verlag, New York-Heidelberg, second edition, 1973.

\bibitem{LYZ:05}
E.~Lutwak, D.~Yang, and G.~Zhang.
\newblock {$L_p$} {J}ohn ellipsoids.
\newblock {\em Proc. London Math. Soc. (3)}, 90(2):497--520, 2005.

\bibitem{LW:17}
A.~Lytchak and S.~Wenger.
\newblock Area minimizing discs in metric spaces.
\newblock {\em Arch. Ration. Mech. Anal.}, 223(3):1123--1182, 2017.

\bibitem{LW:17b}
A.~Lytchak and S.~Wenger.
\newblock Energy and area minimizers in metric spaces.
\newblock {\em Adv. Calc. Var.}, 10(4):407--421, 2017.

\bibitem{LW:18a}
A.~Lytchak and S.~Wenger.
\newblock Intrinsic structure of minimal discs in metric spaces.
\newblock {\em Geom. Topol.}, 22(1):591--644, 2018.

\bibitem{LW:20}
A.~Lytchak and S.~Wenger.
\newblock Canonical parameterizations of metric disks.
\newblock {\em Duke Math. J.}, 169(4):761--797, 2020.

\bibitem{MSZ:03}
J.~Mal\'{y}, D.~Swanson, and W.~P. Ziemer.
\newblock The co-area formula for {S}obolev mappings.
\newblock {\em Trans. Amer. Math. Soc.}, 355(2):477--492, 2003.

\bibitem{Mar:19}
D.~E. Marshall.
\newblock {\em Complex {A}nalysis}.
\newblock Cambridge Mathematical Textbooks. Cambridge University Press,
  Cambridge, 2019.

\bibitem{MW:21}
D.~Meier and S.~Wenger.
\newblock Quasiconformal almost parametrizations of metric surfaces.
\newblock 2021.
\newblock Preprint arXiv:2106.01256.

\bibitem{Mey:10}
D.~Meyer.
\newblock Snowballs are quasiballs.
\newblock {\em Trans. Amer. Math. Soc.}, 362(3):1247--1300, 2010.

\bibitem{Moi:77}
E.~E. Moise.
\newblock {\em Geometric topology in dimensions {$2$} and {$3$}}, volume~47 of
  {\em Graduate Texts in Mathematics}.
\newblock Springer-Verlag, New York-Heidelberg, 1977.

\bibitem{Nt:monotone}
D.~Ntalampekos.
\newblock Monotone {S}obolev functions in planar domains: level sets and smooth
  approximation.
\newblock {\em Arch. Ration. Mech. Anal.}, 238(3):1199--1230, 2020.

\bibitem{NR:21}
D.~Ntalampekos and M.~Romney.
\newblock Polyhedral approximation of metric surfaces and applications to
  uniformization.
\newblock {\em Duke Math. J.}, 2022.
\newblock To appear.

\bibitem{Pet:20}
A.~Petrunin.
\newblock Lectures on metric geometry.
\newblock 2020.
\newblock https://anton-petrunin.github.io/metric-geometry/tex/lectures.pdf.
  Accessed 05-31-2022.

\bibitem{Pom:92}
C.~Pommerenke.
\newblock {\em Boundary behaviour of conformal maps}, volume 299 of {\em
  Grundlehren der Mathematischen Wissenschaften}.
\newblock Springer-Verlag, Berlin, 1992.

\bibitem{Pom02}
C.~Pommerenke.
\newblock Conformal maps at the boundary.
\newblock In {\em Handbook of complex analysis: geometric function theory,
  {V}ol. 1}, pages 37--74. North-Holland, Amsterdam, 2002.

\bibitem{Raj:17}
K.~Rajala.
\newblock Uniformization of two-dimensional metric surfaces.
\newblock {\em Invent. Math.}, 207(3):1301--1375, 2017.

\bibitem{RRR:19}
K.~Rajala, M.~Rasimus, and M.~Romney.
\newblock Uniformization with infinitesimally metric measures.
\newblock {\em J. Geom. Anal.}, 31(11):11445--11470, 2021.

\bibitem{RR:19}
K.~Rajala and M.~Romney.
\newblock Reciprocal lower bound on modulus of curve families in metric
  surfaces.
\newblock {\em Ann. Acad. Sci. Fenn. Math.}, 44(2):681--692, 2019.

\bibitem{Sem:96}
S.~Semmes.
\newblock Finding curves on general spaces through quantitative topology, with
  applications to {S}obolev and {P}oincar\'{e} inequalities.
\newblock {\em Selecta Math. (N.S.)}, 2(2):155--295, 1996.

\bibitem{Sieb:72}
L.~C. Siebenmann.
\newblock Approximating cellular maps by homeomorphisms.
\newblock {\em Topology}, 11:271--294, 1972.

\bibitem{Why:58}
G.~T. Whyburn.
\newblock {\em Topological analysis}.
\newblock Princeton Mathematical Series. No. 23. Princeton University Press,
  Princeton, N. J., 1958.

\bibitem{Wil:10}
K.~Wildrick.
\newblock Quasisymmetric structures on surfaces.
\newblock {\em Trans. Amer. Math. Soc.}, 362(2):623--659, 2010.

\bibitem{Wil:70}
S.~Willard.
\newblock {\em General topology}.
\newblock Dover Publications, Inc., Mineola, NY, 2004.
\newblock Reprint of the 1970 original [Addison-Wesley, Reading, MA].

\bibitem{Wil:12}
M.~Williams.
\newblock Geometric and analytic quasiconformality in metric measure spaces.
\newblock {\em Proc. Amer. Math. Soc.}, 140(4):1251--1266, 2012.

\bibitem{You:48}
J.~W.~T. Youngs.
\newblock Homeomorphic approximations to monotone mappings.
\newblock {\em Duke Math. J.}, 15:87--94, 1948.

\end{thebibliography}

\end{document}